\documentclass[12pt]{amsart}

\usepackage{amsfonts,amssymb,amsmath,amsthm,amstext}
\usepackage{xcolor}
\usepackage{geometry}
\usepackage{mathrsfs}
\usepackage{txfonts}
\usepackage{url}
\usepackage[colorlinks, linkcolor=blue, citecolor=blue, urlcolor=blue, hypertexnames=false]{hyperref}
\allowdisplaybreaks
\numberwithin{equation}{section}
\newtheorem{theorem}{Theorem}[section]
\newtheorem{proposition}[theorem]{Proposition}
\newtheorem{lemma}[theorem]{Lemma}
\newtheorem{corollary}[theorem]{Corollary}
\theoremstyle{remark}
\newtheorem{remark}[theorem]{Remark}

\newcommand{\R}{\mathbb R}
\newcommand{\bbC}{\mathbb C}

\newcommand{\Lipd}{\operatorname{Lip}_d}
\newcommand{\dist}{\operatorname{dist}}
\newcommand{\one}{\mathbf 1}
\newcommand{\supp}{\operatorname{supp}}

\title[Calder\'on commutators and chamber lifting]
{Calder\'on-type commutators and chamber lifting\\ in the Dunkl setting}

\author{Yongsheng Han}
\address{Department of Mathematics, Auburn University, Auburn, AL 36849-5310, U.S.A.}
\email{hanyong@auburn.edu}

\author{Ming-Yi Lee}
\address{Department of Mathematics, National Central University, Chung-Li 320, Taiwan, Republic of China}
\email{mylee@math.ncu.edu.tw}

\author{Ji Li}
\address{School of Mathematical and Physical Sciences, Macquarie University, NSW 2109, Australia}
\email{ji.li@mq.edu.au}

\author{Eric Sawyer}
\address{Department of Mathematics and Statistics, McMaster University, 1280 Main Street West, Hamilton, Ontario L8S 4K1, Canada}
\email{sawyer@mcmaster.ca}

\author{Liangchuan Wu}
\address{School of Mathematical Sciences, Anhui University, Hefei, 230601, P.R.~China}
\email{wuliangchuan@ahu.edu.cn}

\subjclass[2020]{Primary 42B20; Secondary 42B35, 33C52}
\keywords{Calder\'on-type commutators, Dunkl Riesz transforms, chamber lifting, orbit Lipschitz symbols, heat kernels}

\date{}

\begin{document}

\begin{abstract}
Let $G$ be a finite reflection group, and let $T_j$, $\Delta_\kappa$, and $d\omega$ denote its Dunkl operators,
Laplacian, and measure, respectively.  Write $\mathcal R_j=-T_j(-\Delta_\kappa)^{-1/2}$ for the $j$th Dunkl
Riesz transform.  We study the Calder\'on-type commutator $[M_b,T_i\mathcal R_j]$ on the full space
$L^p(\R^N,d\omega)$, without assuming $G$-invariance of the functions.  For $b\in\Lipd$, a
full-space Dunkl $T1$ argument for $[M_b,(-\Delta_\kappa)^{1/2}]$, combined with the exact factorization
$$
    [M_b,T_i\mathcal R_j]   =-M_{\partial_i b}\mathcal R_j-\mathcal R_iM_{\partial_jb}
     +\mathcal R_i[M_b,(-\Delta_\kappa)^{1/2}]\mathcal R_j,
$$
implies boundedness on $L^p(d\omega)$ for every $1<p<\infty$.  We also prove that the prescribed heat truncations are uniformly
bounded on $L^2(d\omega)$ and converge jointly to this operator.  To control these truncations, we lift all reflected values of an
arbitrary function to separate coordinates on a fixed Weyl chamber.  This chamber lifting
retains all non-$G$-invariant information and places every possible orbit singularity on the ordinary chamber diagonal of
a finite matrix operator.  Wall-layer estimates and scalar $T1$ bounds for the integrated entries are uniform in both heat
endpoints.  A dense-core argument then proves joint, path-independent
weak-operator convergence to the factorized commutator.  The lifted limit is associated, in the separated-support
sense, with a finite matrix of scalar Calder\'on--Zygmund kernels on the chamber.
\end{abstract}

\maketitle

\section{Introduction}

Calder\'on-type commutators are among the earliest influential examples of non-convolution singular integrals.
In his 1965 paper \cite{Calderon1965}, Calder\'on introduced
$$
     C_1f(x)=\frac1\pi\operatorname{p.v.}\int_{\R}\frac{a(x)-a(y)}{(x-y)^2}f(y)\,dy.
$$
This operator may be viewed formally as the commutator of multiplication by a Lipschitz coefficient with
$\frac d{dx}H$, where $H$ is the Hilbert transform.  For an ordinary derivative,
the Leibniz rule reduces the corresponding commutator to multiplication by $a'$.  Calder\'on's theorem showed that
the same relation between coefficient regularity and boundedness persists for a genuinely singular, nonlocal operator.

Calder\'on's proof was not a routine Fourier-multiplier argument.  The factor $a(x)-a(y)$ destroys translation
invariance, and the proof used complex function theory, a bilinear analytic representation, and area-integral
estimates.  More concretely, Calder\'on reduced the relevant bilinear form to a Hardy-space estimate for an expression
of the form
$$
     B(F,G)(x)=i\int_0^\infty F'(x+iy)G(x+iy)\,dy.
$$
This expression may be regarded as an early predecessor of the paraproduct point of view.  The equivalence between
Hardy-space norms and area integrals was a central part of the argument.  Thus the commutator was, from the beginning,
a model problem in which coefficient regularity, cancellation, and non-convolution singular integrals meet.

The David--Journ\'e $T1$ theorem \cite{DavidJourne1984} later gave a real-variable criterion for such operators.
In that theory, $L^2$ boundedness is reduced to kernel estimates, weak boundedness, and the conditions
$T1,T^*1\in\operatorname{BMO}$.  For Calder\'on's commutator, integration by parts formally identifies $C_1(1)$,
up to the convention for $H$, with $H(a')$, which belongs to BMO when $a'\in L^\infty$.  This helps explain why
Calder\'on's commutator became a standard test object for Calder\'on--Zygmund theory beyond convolution kernels.
Calder\'on also considered higher-dimensional commutators associated with homogeneous singular kernels.  Subsequent work developed these
ideas in multilinear, endpoint, BMO, and geometric directions; see, for example,
\cite{Calderon1980,CoifmanMeyer1975,ChristJourne1987,Hofmann1994} and the references therein.

The present paper studies an analogue of this question in the rational Dunkl setting.  The classical Fourier
transform extends to an isometry on $L^2(\R^N)$ and is compatible with translations, dilations, and rotations.
To study analysis associated with finite reflection groups, Dunkl introduced the differential-difference operators
in \cite{Dunkl1989}; the associated transform theory was subsequently developed in \cite{Dunkl1992,deJeu1993}.  Let
$R\subset\R^N$ be a normalized root system, let $G$ be the finite reflection group generated by $R$, and let
$\kappa:R\to[0,\infty)$ be a nonnegative
$G$-invariant multiplicity function.  We write $T_j$ for the Dunkl operators, $d\omega$ for the associated measure,
and
$$
      \Delta_\kappa=\sum_{j=1}^NT_j^2,   \qquad   d(x,y)=\min_{\mathsf g\in G}\|x-\mathsf g y\|
$$
for the Dunkl Laplacian and the orbit pseudometric, respectively.

The Dunkl transform plays the role of the Fourier transform and diagonalizes $\Delta_\kappa$, while $d$ records
the singular geometry of the $G$-orbits.  The general theory of the transform, intertwining operators, and radial
translation formulas was developed in \cite{deJeu1993,Rosler1999,Rosler2003,DunklXu}.  Related convolution and
transference results appear in \cite{ThangaveluXu2005,DaiWang2010}.
The resulting harmonic-analysis theory includes Riesz transforms, Hardy spaces, spectral multipliers,
Littlewood--Paley operators, variation and oscillation operators, and commutators; see, for example,
\cite{BetancorCiaurriVarona2007,AmriGasmiSifi2010,ThangaveluXu2007,AmriSifi2012,AnkerDziubanskiHejna2019,Dziubanski2016,
DunklMultiplier2019,DziubanskiHejna2020Atomic,ABFR2024,Hejna2023} and the references therein.

The sharp $L^p$ bounds for Dunkl area integrals in \cite{ChangGongLiLiangWickWuArea}, obtained by the third and fifth
authors with other collaborators, are also conceptually close to the present work.  The same chamber-wise viewpoint
is used here, but the present commutator problem is more delicate.  The finite-coordinate chamber lifting retains the
values of an arbitrary function on all reflected chambers and represents the prescribed heat truncations by a finite
matrix on one fixed chamber.  Uniform $L^2$ bounds for these truncations are not inherited
from a square-function identity and do not follow from the factorization below alone; they are proved by scalar
$T1$ estimates for the integrated matrix entries.

Many classical operators in Dunkl analysis retain a transform or convolution structure.  Write
$$
      D=(-\Delta_\kappa)^{1/2},   \qquad   \mathcal R_j=-T_jD^{-1},    \qquad j=1,\ldots,N,
$$
so that $\mathcal R_j$ is the $j$th Dunkl Riesz transform.  We study the first-order commutator
$[M_b,T_i\mathcal R_j]$, where $M_bf=bf$.  Although the
Dunkl transform diagonalizes $T_i\mathcal R_j$, it does not diagonalize multiplication by a nonconstant $b$.
Consequently, the commutator remains a genuine non-convolution singular integral.

The coefficient class used in this paper is the orbit Lipschitz space $\Lipd(\R^N)$.  Since $d$ is a
pseudometric, the condition on the orbit diagonal must be included explicitly: $b(x)=b(y)$ whenever $d(x,y)=0$, and
$$
       \|b\|_{\Lipd}:=\sup_{d(x,y)>0}\frac{|b(x)-b(y)|}{d(x,y)}<\infty.
$$
It follows that
$$
     |b(x)-b(y)|\leq\|b\|_{\Lipd}d(x,y)
$$
for all $x,y$.  This condition is adapted to the full orbit geometry used in the kernel estimates.  In the classical
commutator, $a(x)-a(y)$ compensates one order of the singularity on the ordinary diagonal $x=y$.  In the Dunkl
setting, differentiated heat kernels also contain terms evaluated at reflected points.  The factor $b(x)-b(y)$ is
therefore estimated against the distance to the orbit diagonal $d(x,y)=0$, not only against $\|x-y\|$.

Euclidean Lipschitz regularity alone need not provide this uniform cancellation for the full reflection group.
If $b$ is not $G$-invariant, then $b(x)-b(\mathsf g x)$ need not vanish on a reflected diagonal.  On the other hand, if
$b\in\Lipd$, then $b(\mathsf g x)=b(x)$ for every $\mathsf g\in G$ and $b$ satisfies the Euclidean Lipschitz
estimate since $d(x,y)\leq\|x-y\|$.  Conversely, if $b$ is $G$-invariant and Euclidean Lipschitz, then, for every
$\mathsf g\in G$,
$$
     |b(x)-b(y)|=|b(x)-b(\mathsf g y)|   \leq\operatorname{Lip}_{\rm Euc}(b)\|x-\mathsf g y\|.
$$
Here $\operatorname{Lip}_{\rm Euc}(b)$ denotes the Euclidean Lipschitz seminorm.  Taking the minimum over
$\mathsf g$ proves the orbit Lipschitz estimate.  Thus $\Lipd$ is exactly the class of
$G$-invariant Euclidean Lipschitz functions, with the same seminorm.

The class $\Lipd$ serves two distinct purposes in the proofs.  Its defining estimate supplies one factor of
$d(x,y)$ to cancel an orbit singularity in the kernels of $[M_b,D]$ and of the heat regularizations.  Its
$G$-invariance permits, in the weak sense on the natural first-order domain, the ordinary product rule
$$
    T_k(bf)=(\partial_kb)f+bT_kf,
$$
which is needed for the exact factorization.  Euclidean Lipschitz regularity also ensures
$\partial_kb\in L^\infty$ and controls the kernel differences used in the $T1$ arguments.  Thus $\Lipd$ is a
sufficient coefficient class for the full orbit formulation; we make no necessity or optimality claim.  If some
multiplicities vanish, the corresponding reflection terms do not occur in the Dunkl operators, and the full-orbit
condition may be stronger than required.  In particular, when $\kappa\equiv0$, it selects the $G$-invariant symbols
from the usual Euclidean Lipschitz class.

The invariance of the symbol should not be confused with a reduction to invariant functions.  The commutator acts on
an arbitrary function.  Multiplication by $b$ is not diagonalized by the Dunkl transform unless $b$ is
constant, so the operator remains a non-convolution singular integral on the full space.  Passing to
$\R^N/G$ would identify the values of a function on different reflected chambers and would therefore change the
operator.  Many effective tools in Dunkl analysis are adapted to $G$-invariant functions or to scalar objects on the
orbit space, whereas that reduction is unavailable here.  All reflected values of the function must instead be
retained, which is why chamber lifting leads to finite matrix kernels rather than to one scalar kernel on the quotient.

The kernel geometry presents a further difficulty.  The heat kernel is localized by $d$, whereas its regularity is
measured through Euclidean displacements.  Applying $T_{i,x}T_{j,x}$ to the heat kernel produces Euclidean
coefficients such as
$$
    \frac{(y_i-x_i)(y_j-x_j)}{t^2},
$$
as well as heat kernels evaluated at reflected points.  Hence possible singularities occur along the full orbit
diagonal, while regularity estimates compare Euclidean distances.  Since $d(x,y)$ and $\|x-y\|$ are not uniformly
comparable, the proof uses a Euclidean correction to the orbit-Gaussian heat-kernel bound to absorb these
coefficients.  At the same time, the orbit Lipschitz estimate contributes a factor of $d(x,y)$, which is absorbed by
the orbit-Gaussian decay in the scale estimates.

Among the closest published results, \cite{HanLeeLiWick2024} treats the zero-order commutators
$[M_b,\mathcal R_j]$ and their BMO and VMO symbol classes.  Those results do not cover the order-one operator
$[M_b,T_i\mathcal R_j]$.  The mixed Euclidean--orbit $T1$ theorem of \cite{TanHanHanLeeLi2025} supplies a
full-space boundedness criterion, but it neither constructs this commutator by itself nor controls the prescribed
two-ended heat truncations.  We therefore separate two questions.  The first is whether the commutator has a bounded
realization on the full $L^p$ space.  The second is whether the prescribed two-ended heat truncations are uniformly
bounded on $L^2$ and converge jointly to that realization.  The first result answers the bounded-realization question by an
exact factorization through $[M_b,D]$.  The second uses chamber lifting to retain all reflected values of a general
function, proves uniform $L^2$ control of the heat truncations, and identifies the resulting finite matrix kernel on a
Weyl chamber.

Our first result constructs the commutator without chamber lifting and without using the prescribed heat
truncations.  Throughout the paper the complex $L^2$ pairing is linear in the first variable.

\begin{theorem}\label{cc:thm:operator}
Let $b\in\Lipd(\R^N)$ and fix $i,j\in\{1,\ldots,N\}$.  For $f,g\in C_c^\infty(\R^N)$, define
$$
     \mathfrak c_b^{ij}(f,g)    :=\langle bT_i\mathcal R_jf,g\rangle-\langle T_i\mathcal R_j(bf),g\rangle.
$$
The distributional commutator $B_b=[M_b,D]$ extends boundedly to $L^p(\R^N,d\omega)$ for every
$1<p<\infty$.  For every such $p$, the form $\mathfrak c_b^{ij}$ has a unique bounded realization
$\mathcal C_{b,\mathrm{abs}}^{ij}$ on $L^p(\R^N,d\omega)$, and
\begin{equation}\label{cc:eq:factor}
     \mathcal C_{b,\mathrm{abs}}^{ij}  =-M_{\partial_i b}\mathcal R_j
     -\mathcal R_iM_{\partial_jb}    +\mathcal R_iB_b\mathcal R_j.
\end{equation}
Here $\partial_kb$ is the almost-everywhere defined weak derivative of the Euclidean Lipschitz representative of
$b$.  Moreover,
$$
      \|B_b\|_{L^p\to L^p}   +\|\mathcal C_{b,\mathrm{abs}}^{ij}\|_{L^p\to L^p} 
      \leq C_p\|b\|_{\Lipd},    \qquad 1<p<\infty,
$$
where $C_p$ depends only on $p$ and the fixed Dunkl structure.
\end{theorem}

The proof is entirely on the full space.  We first treat $B_b=[M_b,D]$.  The coefficient difference turns the
kernel of $D$ into a Dunkl--Calder\'on--Zygmund kernel.  We verify its associated distribution, weak boundedness,
and the two BMO tests, and then apply the full-space Dunkl $T1$ theorem of
\cite[Theorem~1.3]{TanHanHanLeeLi2025}.  The
$G$-invariance of $b$ permits the ordinary product rule for the Dunkl derivatives.  Combining this rule with the
identities relating $D$, $T_j$, and $\mathcal R_j$ yields precisely \eqref{cc:eq:factor}.  Neither a Weyl
chamber nor a choice of heat approximation is used in this construction.

The second question concerns a fixed approximation family rather than the existence of some bounded realization.
Let $H_t=e^{t\Delta_\kappa}$ and let $h_t(x,y)$ be its heat kernel.  For $0<a<B<\infty$, let
\begin{equation}\label{cc:eq:heat-trunc}
    C_{a,B}^{ij,b}f(x)  :=-\frac1{\sqrt\pi}\int_{a^2}^{B^2}\int_{\R^N}
    \bigl(b(x)-b(y)\bigr)T_{i,x}T_{j,x}h_t(x,y)f(y)\,d\omega(y)\frac{dt}{\sqrt t}.
\end{equation}
For a fixed finite interval the formula is first defined on bounded compactly supported functions and extends to
$L^2$ by a two-sided Schur estimate.  The factorization identifies the endpoint limit on a dense first-order core,
but it does not imply a uniform $L^2$ bound for this particular two-ended family.

Single-scale kernel estimates and square Carleson tests do not yield the required uniform bound.
Proposition~\ref{cc:prop:single-scale} exhibits a scalar approximate-identity family for which
the square tests hold but the norm of the integrated truncation grows like $\log\log B$.  We therefore treat every
integrated truncation as one Calder\'on--Zygmund operator and prove estimates uniform in both endpoints.

The geometric step in this second argument is to retain all reflected values rather than identify them.  Fix an
open Weyl chamber $\mathcal C$, enumerate
$$
     G=\{\mathsf g_1,\ldots,\mathsf g_{M_G}\},   \qquad M_G=|G|,    \qquad \mathsf g_1=\mathsf e,
$$
where $\mathsf e$ is the identity element, and define
$$
       Uf(x)=\bigl(f(\mathsf g_1x),\ldots,f(\mathsf g_{M_G}x)\bigr),    \qquad x\in\mathcal C.
$$
The map $U$ is an isometric isomorphism from $L^p(\R^N,d\omega)$ onto
$L^p(\mathcal C,d\omega;\ell^p_{M_G})$.  We refer to this finite-coordinate representation as chamber lifting.
For $1\leq\rho,\tau\leq M_G$, the chamber identity
$$
      d(\mathsf g_\rho x,\mathsf g_\tau y)=\|x-y\|,    \qquad x,y\in\overline{\mathcal C},
$$
places every possible orbit singularity on the common chamber diagonal, while the reflected values of an arbitrary
function remain as separate matrix components.  This is a finite matrix representation, not a quotient reduction.

The chamber representation introduces one additional testing issue.  Chamber indicators are discontinuous across the
reflection walls, so
testing the lifted entries requires the functions
$$
      \chi_\tau=\one_{\mathsf g_\tau\mathcal C}.
$$
Their estimates use smooth chamber cutoffs and wall-layer bounds for heat-scale neighborhoods of the reflecting
hyperplanes.  These component estimates are then used in the integrated $T1$ tests.  Writing $B(x,r)$ for the
Euclidean ball, one must also compare the full-space normalization
$$
     V(x,y,r)=\max\left\{\omega(B(x,r)),\omega(B(y,r))\right\}
$$
with the chamber normalization
\begin{equation}\label{cc:eq:VC}
      V_{\mathcal C}(x,y,r)    =\max\left\{\omega(B(x,r)\cap\mathcal C),\omega(B(y,r)\cap\mathcal C)\right\}.
\end{equation}
Thus the chamber construction is not merely a reformulation: it turns the orbit-singular heat truncation into a
finite matrix whose integrated entries are ordinary Calder\'on--Zygmund operators on the chamber, all with the common
possible singular set $x=y$.

The second principal result identifies the factorized operator with the canonical heat realization.

\begin{theorem}\label{cc:thm:heat-limit}
Under the hypotheses of Theorem~\ref{cc:thm:operator}, let $C_{a,B}^{ij,b}$ be defined by
\eqref{cc:eq:heat-trunc}.  Then
$$
      \sup_{0<a<B<\infty}   \|C_{a,B}^{ij,b}\|_{L^2(\R^N,d\omega)\to L^2(\R^N,d\omega)}    \leq C\|b\|_{\Lipd}.
$$
Moreover, for every $f,g\in L^2(\R^N,d\omega)$,
$$
      \langle C_{a,B}^{ij,b}f,g\rangle   \longrightarrow    \langle\mathcal C_{b,\mathrm{abs}}^{ij}f,g\rangle
$$
as $a\to0+$ and $B\to+\infty$, jointly and independently of the path.  Equivalently,
$$
    UC_{a,B}^{ij,b}U^{-1}   \longrightarrow    U\mathcal C_{b,\mathrm{abs}}^{ij}U^{-1}
$$
in the weak-operator topology on $L^2(\mathcal C,d\omega;\ell^2_{M_G})$.  Thus the joint heat limit exists, is
unique, and coincides with the operator in Theorem~\ref{cc:thm:operator}.
\end{theorem}

The canonical heat realization has the following precise chamber-kernel description.  It is stated here with the
main results and proved in Section~\ref{cc:sec:matrix}.

\begin{corollary}\label{cc:cor:matrix-kernel}
Let $\mathbb T_b^{ij}=U\mathcal C_{b,\mathrm{abs}}^{ij}U^{-1}$ and let
$d\omega_{\mathcal C}=d\omega|_{\mathcal C}$.  For $d(z,w)>0$, let
$$
      K_b^{ij}(z,w)    =-\frac1{\sqrt\pi}\bigl(b(z)-b(w)\bigr)    \int_0^\infty T_{i,z}T_{j,z}h_t(z,w)\frac{dt}{\sqrt t},
$$
and, for $x\ne y$ in $\mathcal C$, let
$$
     \mathbb K_b^{ij,\rho\tau}(x,y)   =K_b^{ij}(\mathsf g_\rho x,\mathsf g_\tau y).
$$
The first integral is absolutely convergent.  For $1\leq\rho,\tau\leq M_G$, the $(\rho,\tau)$ entry of
$\mathbb T_b^{ij}$ is associated with
$\mathbb K_b^{ij,\rho\tau}$ in the separated-support sense on $\mathcal C$.  
If $r=\|x-y\|$, then
$$
     |\mathbb K_b^{ij,\rho\tau}(x,y)|    \leq    C\|b\|_{\Lipd}\frac1{V_{\mathcal C}(x,y,r)}.
$$
If $x,x',y\in\mathcal C$ and $\|x-x'\|\leq r/4$, then
$$
      |\mathbb K_b^{ij,\rho\tau}(x,y)-\mathbb K_b^{ij,\rho\tau}(x',y)|    \leq    C\|b\|_{\Lipd}
       \frac{\|x-x'\|}{r}   \frac1{V_{\mathcal C}(x,y,r)}.
$$
Here $V_{\mathcal C}(x,y,r)$ is defined in \eqref{cc:eq:VC}.
If $x,y,y'\in\mathcal C$ and $\|y-y'\|\leq r/4$, the analogous estimate holds in the second variable.  The
constants are independent of $\rho$ and $\tau$.
\end{corollary}

For each finite heat interval, the lifted truncation is treated entrywise.  Wall-layer estimates control the
discontinuities of the chamber indicators.  We then prove uniform $T1$ and $T^*1$ bounds in
$\operatorname{BMO}(\mathcal C)$, a uniform weak boundedness property, and uniform standard-kernel estimates for the
integrated entries.  The scalar $T1$ theorem on spaces of homogeneous type
\cite[Theorem~12.3]{AuscherHytonenWavelets} supplies the required $L^2$ bounds.  On a dense core the heat multipliers
converge at both endpoints to the form in Theorem~\ref{cc:thm:operator}; uniform boundedness extends this
convergence jointly to all of $L^2$.  Finally, separated-support convergence identifies the finite matrix kernel of
the lifted limit.

Thus the two arguments have different roles: factorization provides the full-space $L^p$ bound, whereas chamber lifting
establishes uniform $L^2$ bounds and joint weak convergence for the prescribed heat truncations and identifies the matrix
kernel of their limit.  The orbit Lipschitz hypothesis is sufficient for both arguments; the present paper does not
characterize all admissible symbols.

This paper is organised as follows.  Section~\ref{cc:sec:prelim} collects the orbit geometry, heat-kernel
estimates, and chamber notation used below.  Section~\ref{cc:sec:full} proves the half-Laplacian
commutator theorem and the exact factorization.  Section~\ref{cc:sec:kernels} establishes the scale and
integrated kernel bounds.  Section~\ref{cc:sec:testing} proves the wall and component estimates.
Section~\ref{cc:sec:lifting} establishes the integrated tests, uniform $L^2$ bounds, and joint
heat limit.  Finally, Section~\ref{cc:sec:matrix} identifies the lifted matrix kernel and records the
entrywise Calder\'on--Zygmund consequence.

\section{Preliminaries}\label{cc:sec:prelim}

Let $R\subset\R^N$ be a normalized root system and let $G$ be the finite reflection group generated by $R$. Throughout this paper,
$$
        \kappa:R\to[0,\infty)
$$
is a nonnegative $G$-invariant multiplicity function.
The chamber lemmas collected in this section are reserved for the canonical heat-realization argument.  The proof of
Theorem~\ref{cc:thm:operator} uses only the full-space orbit geometry, the product rule, and the heat-kernel estimates.
We use the convention
\begin{equation}\label{cc:eq:dunkl-T}
      T_\xi f(x)=\partial_\xi f(x)+\sum_{\alpha\in R}\frac{\kappa(\alpha)}2 
      \langle\alpha,\xi\rangle\frac{f(x)-f(\sigma_\alpha x)}{\langle\alpha,x\rangle}.
\end{equation}
Equivalently, one may sum over a positive subsystem $R_+$. The associated Dunkl measure is
\begin{equation*}
    d\omega(x)=w_\kappa(x)\,dx,  \qquad
    w_\kappa(x)=\prod_{\alpha\in R_+}|\langle x,\alpha\rangle|^{2\kappa(\alpha)}.
\end{equation*}
Let
$$
       \mathbf N:=N+\sum_{\alpha\in R}\kappa(\alpha)
$$
be the homogeneous dimension. The Dunkl Laplacian is
$\Delta_\kappa=\sum_{j=1}^N T_j^2, $
and $H_t=e^{t\Delta_\kappa}$ denotes the heat semigroup with kernel $h_t(x,y)$. Unless otherwise specified, the orbit distance is the full
$G$-orbit distance
\begin{equation*}
     d(x,y)=d_G(x,y):=\min_{\mathsf g\in G}\|x-\mathsf g y\|.
\end{equation*}
For $r>0$, let
\begin{equation*}
    V(x,y,r):=\max\big\{\omega(B(x,r)),\omega(B(y,r))\big\},
\end{equation*}
and
\begin{equation*}
    \mathcal O(B(x,r)):=\{y:d(x,y)<r\}=\bigcup_{\mathsf g\in G}B(\mathsf g x,r).
\end{equation*}

\subsection{Orbit geometry and volume estimates}

We recall several fundamental facts regarding the orbit geometry and volume estimates. See for example
\cite{BetancorCiaurriVarona2007,AmriGasmiSifi2010,ThangaveluXu2007,AmriSifi2012,
Dziubanski2016,DunklMultiplier2019,DziubanskiHejna2020Atomic,ABFR2024,Hejna2023}.

\begin{lemma}
The function $d$ is $G$-invariant in both variables, satisfies the triangle
inequality, and obeys $d(x,y)\leq\|x-y\|$.
\end{lemma}

We shall repeatedly use
\begin{equation*}
    \omega(B(x,r))\approx r^N\prod_{\alpha\in R_+}
    \big(|\langle x,\alpha\rangle|+r\big)^{2\kappa(\alpha)}.
\end{equation*}

\begin{lemma}\label{cc:lem:volume}
For $0<r\leq r_2$,
$$
   c\left(\frac {r_2}r\right)^N\omega(B(x,r))\leq  \omega(B(x,r_2))
   \leq C\left(\frac {r_2}r\right)^{\mathbf N}\omega(B(x,r)).
$$
\end{lemma}

\begin{lemma}
For every $x\in\R^N$ and $r>0$,
$\omega(\mathcal O(B(x,r)))\approx \omega(B(x,r))$.
\end{lemma}

\begin{lemma}
If $r=d(x,y)>0$, then
$\omega(B(x,r))\approx\omega(B(y,r)),\ V(x,y,r)\approx\omega(B(x,r))\approx\omega(B(y,r))$.
\end{lemma}

\begin{lemma}
For every $\mathsf g,\mathsf h\in G$,
$\omega(B(\mathsf g x,r))=\omega(B(x,r)),\ V(\mathsf g x,\mathsf h y,r)=V(x,y,r)$.
\end{lemma}

\subsection{Chamber geometry}

Let $\mathcal C$ be the fixed open Weyl chamber.  Enumerate
$$
     G=\{\mathsf g_1,\ldots,\mathsf g_{M_G}\},
        \qquad M_G=|G|,\qquad \mathsf g_1=\mathsf e,
$$
and write $\Omega_\rho:=\mathsf g_\rho\mathcal C$.

\begin{lemma}\label{cc:lem:partition}
For $\alpha\in R$, let
$\alpha^\perp:=\{x\in\R^N:\langle x,\alpha\rangle=0\}$ and
$\mathcal W:=\bigcup_{\alpha\in R}\alpha^\perp$. Then
$$
     \R^N\setminus\mathcal W=\bigsqcup_{\rho=1}^{M_G}\Omega_\rho,
     \qquad \omega(\mathcal W)=0.
$$
\end{lemma}

\begin{proof}
The first identity is the Weyl chamber decomposition.
Since $R$ is finite and each $\alpha^\perp$ is a Lebesgue-null linear hyperplane,
$$
     0\leq\omega(\mathcal W)  =\int_{\mathcal W}w_\kappa(x)\,dx
     \leq\sum_{\alpha\in R}\int_{\alpha^\perp}w_\kappa(x)\,dx=0.
$$
\end{proof}

\begin{lemma}\label{cc:lem:closest}
If $x,y\in\overline{\mathcal C}$, then
$d(x,y)=\|x-y\|$.  Consequently,
\begin{equation*}
     d(\mathsf g_\rho x,\mathsf g_\tau y)=\|x-y\|,  \qquad x,y\in\overline{\mathcal C}.
\end{equation*}
\end{lemma}

\begin{proof}
If $z$ and $y\in\overline{\mathcal C}$ lie on opposite sides of
$\alpha^\perp$, then
\begin{align*}
     \|\sigma_\alpha z-y\|^2-\|z-y\|^2
     &=4\frac{\langle z,\alpha\rangle\langle y,\alpha\rangle}{\|\alpha\|^2}  \leq0.
\end{align*}
Successively reflecting $\mathsf g x$ across separating walls into
$\overline{\mathcal C}$ therefore shows, for every $\mathsf g\in G$, that
\begin{equation*}
     \|x-y\|\leq\|\mathsf g x-y\|,   \qquad x,y\in\overline{\mathcal C},
\end{equation*}
because the orbit of $x$ has the unique representative $x$ in $\overline{\mathcal C}$.  Hence
$d(x,y) =\min_{\mathsf g\in G}\|x-\mathsf g y\|=\|x-y\|,$
and
\begin{align*}
     d(\mathsf g_\rho x,\mathsf g_\tau y)  &=\min_{\mathsf h\in G}
     \|x-\mathsf g_\rho^{-1}\mathsf h\mathsf g_\tau y\|  =d(x,y)=\|x-y\|.
\end{align*}
\end{proof}

Let $d\omega_{\mathcal C}=d\omega|_{\mathcal C}$ and write
$\omega_{\mathcal C}(E)=\omega(E)$ for measurable $E\subset\mathcal C$. For $x\in\overline{\mathcal C}$, let
$$
       B_{\mathcal C}(x,r)=B(x,r)\cap\mathcal C,\qquad   
       V_{\mathcal C}(x,y,r)=\max\{\omega(B_{\mathcal C}(x,r)),\omega(B_{\mathcal C}(y,r))\}.
$$

\begin{lemma}\label{cc:lem:chamber-volume}
For every $x\in\overline{\mathcal C}$ and $r>0$,
$\omega(B(x,r))\approx \omega(B_{\mathcal C}(x,r))$.  Consequently,
$$
       V(\mathsf g_\rho x,\mathsf g_\tau y,r)\approx V_{\mathcal C}(x,y,r),
$$
uniformly in $\rho,\tau$. In particular $(\mathcal C,\|\cdot\|,d\omega_{\mathcal C})$ is a space of homogeneous type.
\end{lemma}

\begin{proof}
The inequality $\omega(B_{\mathcal C}(x,r))\leq \omega(B(x,r))$
follows from the inclusion $B_{\mathcal C}(x,r)\subset B(x,r)$.
We prove the reverse inequality.
By Lemma~\ref{cc:lem:partition}, modulo a null set, $B(x,r)$ is the disjoint union of its intersections with the chambers
$\mathsf g\mathcal C$.  By the invariance of $d\omega$,
$$
      \omega(B(x,r)\cap\mathsf g\mathcal C)   =  \omega(B(\mathsf g^{-1}x,r)\cap\mathcal C).
$$
We claim that
$$
       B(\mathsf g^{-1}x,r)\cap\mathcal C    \subset B(x,r)\cap\mathcal C.
$$
Indeed, if $y\in B(\mathsf g^{-1}x,r)\cap\mathcal C$, then $\|\mathsf g^{-1}x-y\|<r$.
By Lemma~\ref{cc:lem:closest}, the chamber representative $x$ is no farther from $y$ than any other
representative of its orbit; therefore $\|x-y\|\leq \|\mathsf g^{-1}x-y\|<r$.
Thus the inclusion holds.  Summing over the finitely many chambers,
\begin{align*}
      \omega(B(x,r))   &=\sum_{\mathsf g\in G}\omega(B(x,r)\cap\mathsf g\mathcal C)
      \leq M_G\,\omega(B(x,r)\cap\mathcal C)   =M_G\,\omega(B_{\mathcal C}(x,r)).
\end{align*}
Thus $\omega(B(x,r))\approx\omega(B_{\mathcal C}(x,r))$.

Now let $x,y\in\overline{\mathcal C}$.  By reflection invariance of the measure,
$$
       \omega(B(\mathsf g_\rho x,r))=\omega(B(x,r)),   \qquad
      \omega(B(\mathsf g_\tau y,r))=\omega(B(y,r)).
$$
Using the comparison just proved for $x$ and $y$, we get
\begin{align*}
      V(\mathsf g_\rho x,\mathsf g_\tau y,r)   &=\max\{\omega(B(x,r)),\omega(B(y,r))\}\\
       &\approx    \max\{\omega(B_{\mathcal C}(x,r)),\omega(B_{\mathcal C}(y,r))\}  =V_{\mathcal C}(x,y,r),
\end{align*}
with constants independent of $\rho,\tau$.  Finally,
$$
     \omega(B_{\mathcal C}(x,2r))   \approx \omega(B(x,2r))
     \lesssim \omega(B(x,r))   \approx \omega(B_{\mathcal C}(x,r)),
$$
hence $(\mathcal C,\|\cdot\|,d\omega_{\mathcal C})$ is a space of homogeneous type.
\end{proof}

\subsection{Orbit Lipschitz symbols}

\begin{lemma}
If $b\in\operatorname{Lip}_d(\R^N)$, then $b(\mathsf g x)=b(x)$ for every $\mathsf g\in G$, and
$|b(x)-b(y)|\leq\|b\|_{\operatorname{Lip}_d}\|x-y\|$.
\end{lemma}

\begin{proof}
For any $\mathsf g\in G$, the two points $x$ and $\mathsf g x$ lie on the same orbit, hence
$d(x,\mathsf g x)=0$.
The orbit-Lipschitz condition forces $|b(x)-b(\mathsf g x)|\leq \|b\|_{\operatorname{Lip}_d}d(x,\mathsf g x)=0,$
so $b(\mathsf g x)=b(x)$.  The Euclidean Lipschitz estimate is obtained by using the identity element in the orbit-distance minimum:
$$
     d(x,y)=\min_{\mathsf h\in G}\|x-\mathsf h y\|\leq\|x-y\|.
$$
Therefore $|b(x)-b(y)|\leq \|b\|_{\operatorname{Lip}_d}d(x,y) \leq \|b\|_{\operatorname{Lip}_d}\|x-y\|$.
\end{proof}

\begin{lemma}\label{cc:lem:smooth-symbol}
Let $b\in\operatorname{Lip}_d(\R^N)$.  There exist $b_\varepsilon\in C^\infty(\R^N)$ such that $b_\varepsilon$ is $G$-invariant, and
\begin{itemize}
\item $\displaystyle
    \|b_\varepsilon\|_{\operatorname{Lip}_d}\leq\|b\|_{\operatorname{Lip}_d},\ 
    \max_{1\leq k\leq N}\|\partial_kb_\varepsilon\|_\infty
    \leq\|b\|_{\operatorname{Lip}_d},$

\item $\displaystyle \partial_kb_\varepsilon\longrightarrow\partial_kb \quad\text{a.e.},\  k=1,\ldots,N,$

\item $\displaystyle   \|b_\varepsilon-b\|_\infty
    \lesssim\varepsilon\|b\|_{\operatorname{Lip}_d}.$
\end{itemize}
\end{lemma}

\begin{proof}
Take $b_\varepsilon=b*\varphi_\varepsilon$ with a radial mollifier.
Radiality preserves $G$-invariance.  Convolution does not increase the Euclidean Lipschitz seminorm, so every coordinate difference
quotient of $b_\varepsilon$ is bounded by $\|b\|_{\operatorname{Lip}_d}$.  Passing to the limit,
$ \max_{1\leq k\leq N}\|\partial_kb_\varepsilon\|_\infty \leq\|b\|_{\operatorname{Lip}_d}.$ 

Since  $b_\varepsilon$ is $G$-invariant,
$$
     |b_\varepsilon(x)-b_\varepsilon(y)|=|b_\varepsilon(x)-b_\varepsilon(\mathsf g y)|
     \leq\|b\|_{\operatorname{Lip}_d}\|x-\mathsf g y\|.
$$
Taking the infimum over $\mathsf g\in G$ establishes the $\operatorname{Lip}_d$ bound.
Moreover,
$$
     |b_\varepsilon(x)-b(x)| \leq\|b\|_{\operatorname{Lip}_d}
     \int_{\R^N}\|z\|\varphi_\varepsilon(z)\,dz
     \lesssim\varepsilon\|b\|_{\operatorname{Lip}_d},
$$
so $b_\varepsilon\to b$ uniformly on $\R^N$.
The derivative convergence holds at the Lebesgue points of $\partial_kb$.
\end{proof}

\begin{lemma}\label{cc:lem:product}
If $b,f\in C^1(\R^N)$ and $b$ is $G$-invariant, then
$T_j(bf)=(\partial_jb)f+bT_jf$.
\end{lemma}

\begin{proof}
Using the definition of $T_j$,
\begin{align*}
      T_j(bf)(x)   &=\partial_j(bf)(x)   + \sum_{\alpha\in R}\frac{\kappa(\alpha)}2\alpha_j
      \frac{b(x)f(x)-b(\sigma_\alpha x)f(\sigma_\alpha x)}{\langle\alpha,x\rangle}.
\end{align*}
Since $b$ is $G$-invariant, $b(\sigma_\alpha x)=b(x)$.  Thus the difference part factors as
$$
    b(x)  \sum_{\alpha\in R}\frac{\kappa(\alpha)}2\alpha_j
    \frac{f(x)-f(\sigma_\alpha x)}{\langle\alpha,x\rangle}.
$$
The ordinary derivative contributes
$\partial_j(bf)=(\partial_jb)f+b\partial_jf$.
Combining the derivative term with the factored difference part yields
$$
     T_j(bf)=(\partial_jb)f+bT_jf.
$$
The summands with $\kappa(\alpha)=0$ are present only notationally and vanish.
\end{proof}

\subsection{Heat-kernel estimates}

We record the heat-kernel estimates used below.  The second derivative bound
requires more than the bare orbit-Gaussian estimate
$$
     |h_t(x,y)|    \lesssim     V(x,y,\sqrt t)^{-1}  \exp\left(-c\frac{d(x,y)^2}{t}\right)
$$
which alone does not control the Euclidean coefficient arising after one more Dunkl differentiation.
The refined estimate of Dziuba\'nski and Hejna \cite[Theorem~3.1]{DziubanskiHejna2020Atomic} contains the additional factor
$$
      \bigg(1+\frac{\|x-y\|}{\sqrt t}\bigg)^{-2},
$$
together with the corresponding heat-scale difference estimate in the second variable.
The first-variable difference estimate used below follows from the symmetry of the heat kernel.
We also use their first Dunkl derivative identity \cite[Lemma~3.3]{DziubanskiHejna2020Atomic}, which is consistent with the standard
explicit formula for the Dunkl heat kernel.

\begin{lemma}\label{cc:lem:heat-buffer}
There exist constants $C,c>0$ such that
\begin{equation*}
      |h_t(x,y)|\leq C\frac1{V(x,y,\sqrt t)}   \left(1+\frac{\|x-y\|}{\sqrt t}\right)^{-2}
     \exp\left(-c\frac{d(x,y)^2}{t}\right).
\end{equation*}
Moreover, if $\|y-y'\|\leq\sqrt t$, then
\begin{align*}
      |h_t(x,y)-h_t(x,y')|   &\leq C\frac{\|y-y'\|}{\sqrt t}\frac1{V(x,y,\sqrt t)}
      \left(1+\frac{\|x-y\|}{\sqrt t}\right)^{-2}   \exp\left(-c\frac{d(x,y)^2}{t}\right).
\end{align*}
If $\|x-x'\|\leq\sqrt t$, then
\begin{align*}
     |h_t(x,y)-h_t(x',y)|  &\leq C\frac{\|x-x'\|}{\sqrt t}\frac1{V(x,y,\sqrt t)}
     \left(1+\frac{\|x-y\|}{\sqrt t}\right)^{-2}   \exp\left(-c\frac{d(x,y)^2}{t}\right).
\end{align*}
\end{lemma}

\subsection{Regularized operators}

For $i,j=1,\ldots,N$, let
$$
    A_t^{ij}(x,y):=T_{i,x}T_{j,x}h_t(x,y).
$$
For $0<a<B<\infty$, let $C_{a,B}^{ij,b}$ be the finite heat truncation defined in
\eqref{cc:eq:heat-trunc}.
The formula is first defined for $f\in L_c^\infty(\R^N,d\omega)$, with kernel
$$
       K_{a,B}^{ij,b}(x,y)   :=-\frac1{\sqrt\pi}\bigl(b(x)-b(y)\bigr)   \int_{a^2}^{B^2}A_t^{ij}(x,y)\frac{dt}{\sqrt t}.
$$
Corollary~\ref{cc:cor:theta} and Lemma~\ref{cc:lem:gaussian-mass} below, applied in both variables, will yield
$$
     \sup_x\int_{\R^N}|K_{a,B}^{ij,b}(x,y)|\,d\omega(y)   +\sup_y\int_{\R^N}|K_{a,B}^{ij,b}(x,y)|\,d\omega(x)
      \lesssim \|b\|_{\operatorname{Lip}_d}\int_a^B\frac{ds}{s}   =\|b\|_{\operatorname{Lip}_d}\log\frac Ba.
$$
Thus each $C_{a,B}^{ij,b}$ has an $L^2$ extension, with a bound that may depend on the endpoints.
Theorem~\ref{cc:thm:heat-limit} later proves the uniform bound and identifies the joint weak limit with
$\mathcal C_{b,\mathrm{abs}}^{ij}$.

\begin{remark}\label{cc:rem:normalization}
With the convention
$\mathcal R_j=-T_j(-\Delta_\kappa)^{-1/2}$, write $c_\ast=-\pi^{-1/2}$ for the coefficient in
\eqref{cc:eq:heat-trunc}.  The representation of
$(-\Delta_\kappa)^{-1/2}$ is understood as an improper spectral integral on its natural domain; the operator in
\eqref{cc:eq:heat-trunc} uses only a finite heat interval.  If $t=s^2$, then $dt/\sqrt t=2\,ds$. With
$$\theta_s(x,y):=s(b(x)-b(y))A_{s^2}^{ij}(x,y),$$
we have 
$$
     (b(x)-b(y))A_{s^2}^{ij}(x,y)\frac{dt}{\sqrt t}    =2\theta_s(x,y)\frac{ds}{s}.
$$
\end{remark}

\subsection{The full-space Dunkl T1 theorem}

We record the precise external interface used in Section~\ref{cc:sec:full}.  For
$0<\eta\leq1$, let $C_0^\eta(\R^N)$ be the compactly supported $\eta$-H\"older functions.  A distribution
$T:C_0^\eta(\R^N)\to(C_0^\eta(\R^N))'$ is a Dunkl--Calder\'on--Zygmund singular integral if it is associated,
for test functions with separated supports, with a kernel $K(x,y)$ satisfying, for some $0<\varepsilon\leq1$,
\begin{equation}\label{cc:eq:dCZ-size}
     |K(x,y)|\leq C   \left(\frac{d(x,y)}{\|x-y\|}\right)^\varepsilon   \frac1{V(x,y,d(x,y))},   \qquad d(x,y)>0,
\end{equation}
and
\begin{equation}\label{cc:eq:dCZ-smooth}
\begin{cases}\displaystyle
      |K(x,y)-K(x',y)|  \leq C    \left(\frac{\|x-x'\|}{\|x-y\|}\right)^\varepsilon   \frac1{V(x,y,d(x,y))},
      \qquad \|x-x'\|\leq\frac{d(x,y)}2, \\   \displaystyle
      |K(x,y)-K(x,y')|   \leq C   \left(\frac{\|y-y'\|}{\|x-y\|}\right)^\varepsilon   \frac1{V(x,y,d(x,y))},
         \qquad \|y-y'\|\leq\frac{d(x,y)}2.
\end{cases}
\end{equation}
The endpoint volumes are comparable at the orbit scale, so this symmetric normalization is equivalent to the one in
\cite[Definition~1.1]{TanHanHanLeeLi2025} after interchanging the two kernel variables.

For a function $\varphi:\R^N\to\bbC$, write
$$
    [\varphi]_{C^\eta}   :=\sup_{x\ne y}\frac{|\varphi(x)-\varphi(y)|}{\|x-y\|^\eta}.
$$
The weak boundedness property in \cite[Definition~1.2]{TanHanHanLeeLi2025} is the following joint-test condition.  There is a constant
$C_{\mathrm{WBP}}$ such that
$$
    |\langle K,F\rangle|   \leq C_{\mathrm{WBP}}    \max\{\omega(B(x_0,r)),\omega(B(y_0,r))\}
$$
whenever $F\in C_0^\eta(\R^N\times\R^N)$ is supported in
$B(x_0,r)\times B(y_0,r)$ and satisfies
$$
    \|F\|_\infty\leq1,\qquad   \sup_{y\in\R^N}[F(\cdot,y)]_{C^\eta}\leq r^{-\eta},\qquad
    \sup_{x\in\R^N}[F(x,\cdot)]_{C^\eta}\leq r^{-\eta}.
$$
We denote this property by $T\in\operatorname{WBP}$.

Before $L^2$ boundedness is known, the two testing distributions are defined as follows.  If
$\phi\in C_0^\eta(\R^N)$ has mean zero and is supported in $B(x_0,R)$, choose a $G$-invariant
$\chi\in C_0^\eta(\R^N)$ which equals one on $\mathcal O(B(x_0,2R))$ and vanishes outside
$\mathcal O(B(x_0,4R))$.  Then
\begin{align}\label{cc:eq:T1-def}
     \langle T1,\phi\rangle  &:=\langle T\chi,\phi\rangle
     +\iint_{\R^N\times\R^N}\bigl[K(x,y)-K(x_0,y)\bigr](1-\chi(y))\overline{\phi(x)}
     \,d\omega(y)d\omega(x).
\end{align}
The smoothness estimate and orbit-annular summation make the second term absolutely convergent and show that the
definition is independent of $\chi$.  The distribution $T^*1$ is defined in the same way with the adjoint kernel.
Here
$$
     \|u\|_{\operatorname{BMO}(d\omega)}    :=\sup_B\frac1{\omega(B)}\int_B|u-u_B|\,d\omega,
$$
where the supremum is taken over Euclidean balls.

We shall use the following specialization of the full-space Dunkl theory.

\begin{theorem}[{\cite[Theorem~1.3]{TanHanHanLeeLi2025}}]\label{cc:thm:dunkl-T1}
Let $T$ be a Dunkl--Calder\'on--Zygmund singular integral satisfying
\eqref{cc:eq:dCZ-size}--\eqref{cc:eq:dCZ-smooth}.  Then $T$ extends boundedly to $L^2(d\omega)$ if and
only if
$$
      T1,\,T^*1\in\operatorname{BMO}(d\omega)   \qquad\text{and}\qquad    T\in\operatorname{WBP}.
$$
Once the $L^2$ extension is known, \cite[Theorem~2.1]{TanHanHanLeeLi2025} supplies boundedness on $L^p(d\omega)$
for every $1<p<\infty$.  The operator
norms are controlled by the kernel, testing, and weak boundedness constants and by the fixed Dunkl structure.
\end{theorem}

\section[Full-space operator theorem]{The full-space operator theorem through the
half-Laplacian}\label{cc:sec:full}

Define
\begin{equation*}
        D=(-\Delta_\kappa)^{1/2},    \qquad H_t=e^{t\Delta_\kappa},     \qquad \mathcal R_j=-T_jD^{-1}.
\end{equation*}
The multiplier at the origin is set equal to zero.  Thus each
$\mathcal R_j$ is an $L^2(d\omega)$ contraction and, on the natural
first-order domain,
\begin{equation}\label{cc:eq:D-Riesz}
        D\mathcal R_j=-T_j,    \qquad T_j=-\mathcal R_jD,     \qquad D=\sum_{k=1}^N\mathcal R_kT_k.
\end{equation}
The last identity follows immediately from the Dunkl multipliers:
$\mathcal R_kT_k$ has multiplier $\xi_k^2/\lVert\xi\rVert$.

The purpose of this section is to prove, without assuming uniform bounds for the prescribed heat truncations, that
\begin{equation*}
        \lVert[M_b,D]\rVert_{L^p(d\omega)\to L^p(d\omega)}   \lesssim_p \lVert b\rVert_{\operatorname{Lip}_d},      
        \qquad 1<p<\infty,
\end{equation*}
and then to recover the Calder\'on commutator from an exact operator
identity.  Recall that $b\in\operatorname{Lip}_d$ is $G$-invariant and is
Euclidean Lipschitz with the same seminorm.  In particular,
$L_b:=\lVert b\rVert_{\operatorname{Lip}_d}<\infty$,
$\partial_kb\in L^\infty$ and
\begin{equation}\label{cc:eq:weak-product}
        T_k(bf)=(\partial_kb)f+bT_kf
\end{equation}
for $f\in C_c^\infty$, and for every $f\in\operatorname{Dom}(D)$ when
$b\in L^\infty$.  Indeed, let $b_\varepsilon$ be as in
Lemma~\ref{cc:lem:smooth-symbol}.  For $\phi\in C_c^\infty$, apply Lemma~\ref{cc:lem:product} to
$\overline{b_\varepsilon}\phi$ and let $\varepsilon\to0+$.
Uniform convergence, almost-everywhere derivative convergence with a
uniform $L^\infty$ bound, and $T_k^*=-T_k$ imply \eqref{cc:eq:weak-product} distributionally.  If
$b\in L^\infty$ and $f\in\operatorname{Dom}(T_k)$, the right-hand side
belongs to $L^2$, hence $bf\in\operatorname{Dom}(T_k)$ and the identity holds in $L^2$.

\subsection{The kernel of the half-Laplacian}

We first retain the Euclidean buffer in the heat-kernel estimate.  Let
\begin{equation*}
        \delta=d(x,y),  \qquad \rho=\lVert x-y\rVert,    \qquad V_s=V(x,y,s).
\end{equation*}
Thus $0<\delta\leq\rho$ off the orbit diagonal.

\begin{lemma}
Suppose that $0<\delta\leq\rho$.  Then, for every $h\geq0$,
\begin{equation}\label{cc:eq:D-int-size}
      \int_0^\infty \frac{1}{s^2V_s}\left(1+\frac{\rho}{s}\right)^{-2}
     \exp\left(-c\frac{\delta^2}{s^2}\right)\,ds
     \lesssim \frac{1}{\rho V(x,y,\delta)}.
\end{equation}
\begin{equation}\label{cc:eq:D-int-smooth}
       h\int_0^\infty \frac{1}{s^3V_s}\left(1+\frac{\rho}{s}\right)^{-2}
       \exp\left(-c\frac{\delta^2}{s^2}\right)\,ds
       \lesssim \frac{h}{\rho\delta V(x,y,\delta)}.
\end{equation}
The implicit constants depend only on the Dunkl data and on $c>0$.
\end{lemma}

\begin{proof}
Write $V_\delta=V(x,y,\delta)$ and let $Q$ be a doubling exponent for
$d\omega$.  On $0<s\leq\delta$,
\begin{equation*}
        V_s^{-1}\lesssim     \left(\frac{\delta}{s}\right)^Q V_\delta^{-1},   \qquad
        \left(1+\frac{\rho}{s}\right)^{-2}    \leq \frac{s^2}{\rho^2}.
\end{equation*}
After the substitution $u=\delta/s$, this part of   \eqref{cc:eq:D-int-size} is bounded by
$C\delta/(\rho^2V_\delta)$, and the corresponding part of   \eqref{cc:eq:D-int-smooth} is bounded by
$Ch/(\rho^2V_\delta)$.  On $\delta<s<\rho$ one may use   $V_s\geq V_\delta$ and the same bound $s^2/\rho^2$ for the Euclidean
buffer.  The two resulting integrals are
\begin{equation*}
        \frac{C}{\rho^2V_\delta}\int_\delta^\rho ds   \leq \frac{C}{\rho V_\delta}\qquad 
\text{and}\qquad 
        \frac{Ch}{\rho^2V_\delta}\int_\delta^\rho\frac{ds}{s}    \leq \frac{Ch}{\rho\delta V_\delta},
\end{equation*}
where the last inequality uses $\log u\leq u$ for $u\geq1$.
Finally, by volume monotonicity on $s\geq\rho$,
\begin{equation*}
        \int_\rho^\infty\frac{ds}{s^2V_s}    \lesssim \frac{1}{\rho V_\rho}    
        \leq \frac{1}{\rho V_\delta}  \qquad \text{and}\qquad 
        h\int_\rho^\infty\frac{ds}{s^3V_s}    \lesssim \frac{h}{\rho^2V_\rho}
        \leq \frac{h}{\rho\delta V_\delta}.
\end{equation*}
Combining the three scale ranges establishes both estimates.
\end{proof}

The Balakrishnan formula is
\begin{equation}\label{cc:eq:Bal-D}
        Df=\frac{1}{2\sqrt\pi}\int_0^\infty     (f-H_tf)\,\frac{dt}{t^{3/2}}
\end{equation}
on the spectral domain of $D$.  The identity term in
\eqref{cc:eq:Bal-D} is supported on the ordinary diagonal.  Hence, for
$d(x,y)>0$, define the off-orbit-diagonal kernel associated with $D$ by
\begin{equation}\label{cc:eq:D-kernel}
       J_D(x,y):=-\frac{1}{2\sqrt\pi}    \int_0^\infty h_t(x,y)\,\frac{dt}{t^{3/2}}.
\end{equation}

\begin{proposition}\label{cc:prop:D-kernel}
For $d(x,y)>0$,
\begin{equation}\label{cc:eq:D-kernel-size}
        |J_D(x,y)|   \lesssim    \frac{1}{\lVert x-y\rVert\,      V(x,y,d(x,y))}.
\end{equation}
If $h=\lVert x-x'\rVert\leq d(x,y)/2$, then
\begin{equation}\label{cc:eq:D-kernel-reg}
       \big|J_D(x,y)-J_D(x',y)\big|   \lesssim    \frac{h}{\lVert x-y\rVert\,d(x,y)    V(x,y,d(x,y))}.
\end{equation}
The analogous estimate holds in the second variable.
\end{proposition}

\begin{proof}
Let
$$
    \delta=d(x,y),\qquad    \rho=\lVert x-y\rVert,\qquad     V_s=V(x,y,s).
$$
By Lemma~\ref{cc:lem:heat-buffer} and the change of variables $t=s^2$,
\begin{align*}
        |J_D(x,y)|  &\lesssim   \int_0^\infty   \frac{1}{s^2V_s}
     \left(1+\frac{\rho}{s}\right)^{-2}  \exp\left(-c\frac{\delta^2}{s^2}\right)\,ds \lesssim
    \frac{1}{\rho V(x,y,\delta)},
\end{align*}
where the last inequality is \eqref{cc:eq:D-int-size}, proving
\eqref{cc:eq:D-kernel-size}.

Let
$$
    h=\lVert x-x'\rVert\leq\frac{\delta}{2},\qquad
    \delta'=d(x',y),\qquad    \rho'=\lVert x'-y\rVert.
$$
The triangle inequalities show
$$
    \frac{\delta}{2}\leq\delta'\leq\frac{3\delta}{2},
    \qquad     \frac{\rho}{2}\leq\rho'\leq\frac{3\rho}{2}.
$$
If $Q$ is a doubling exponent for $d\omega$, then
\begin{align*}
       B(x,s)\subset B(x',s+h)  \quad\Longrightarrow\quad   V(x',y,s)^{-1}
       \lesssim   \left(1+\frac{h}{s}\right)^Q V_s^{-1}.
\end{align*}
For $s\geq h$, Lemma~\ref{cc:lem:heat-buffer} implies
\begin{align*}
       |h_{s^2}(x,y)-h_{s^2}(x',y)|   &\lesssim   \frac{h}{s}\frac{1}{V_s}
       \left(1+\frac{\rho}{s}\right)^{-2}   \exp\left(-c\frac{\delta^2}{s^2}\right).
\end{align*}
For $0<s<h$, the two heat-kernel size estimates and the preceding comparisons imply
\begin{align*}
       |h_{s^2}(x,y)-h_{s^2}(x',y)|   &\lesssim  \frac{1}{V_s}
       \left(1+\frac{\rho}{s}\right)^{-2}   \exp\left(-c\frac{\delta^2}{s^2}\right) +
      \left(1+\frac{h}{s}\right)^Q  \frac{1}{V_s}  \left(1+\frac{\rho}{s}\right)^{-2}  \exp\left(-\frac{c\delta^2}{4s^2}\right).
\end{align*}
Since $0<s<h\leq\delta/2$,
$$
    \left(1+\frac{h}{s}\right)^Q   \exp\left(-\frac{c\delta^2}{4s^2}\right)
    \lesssim  \frac{h}{s}\exp\left(-c'\frac{\delta^2}{s^2}\right).
$$
Consequently, for every $s>0$,
\begin{align*}
    |h_{s^2}(x,y)-h_{s^2}(x',y)|  &\lesssim  \frac{h}{s}\frac{1}{V_s}
     \left(1+\frac{\rho}{s}\right)^{-2}   \exp\left(-c'\frac{\delta^2}{s^2}\right).
\end{align*}
Using \eqref{cc:eq:D-kernel} and then
\eqref{cc:eq:D-int-smooth}, we obtain
\begin{align*}
     |J_D(x,y)-J_D(x',y)|   &\lesssim  h\int_0^\infty   \frac{1}{s^3V_s}
     \left(1+\frac{\rho}{s}\right)^{-2}
     \exp\left(-c'\frac{\delta^2}{s^2}\right)\,ds \lesssim
     \frac{h}{\rho\delta V(x,y,\delta)}.
\end{align*}
This is \eqref{cc:eq:D-kernel-reg}; the second-variable estimate follows from
$J_D(x,y)=J_D(y,x)$.
\end{proof}

\subsection{The half-Laplacian commutator is a Dunkl CZO}

For the commutator convention used in this paper, write
\begin{equation*}
        B_b=[M_b,D]=M_bD-DM_b.
\end{equation*}
More precisely, for $f,g\in C_c^\infty(\R^N)$ define the
distributional commutator form by
\begin{equation}\label{cc:eq:half-form}
      \langle B_b f,g\rangle    :=\langle Df,\overline b g\rangle  -\langle bf,Dg\rangle.
\end{equation}
The Dunkl multiplier identities read
$$
    \operatorname{Dom}(D)=\bigcap_{k=1}^N\operatorname{Dom}(T_k),  \qquad
     \|Df\|_2^2=\sum_{k=1}^N\|T_kf\|_2^2.
$$
Since $b$ is locally bounded and Euclidean Lipschitz, the weak product rule
\eqref{cc:eq:weak-product} shows that $bf,\overline bg\in\operatorname{Dom}(D)$.  First suppose that the two supports
have positive orbit separation.  Insert a finite heat interval in
\eqref{cc:eq:Bal-D} and then use
\eqref{cc:eq:half-form}.  The identity part of $I-H_t$
cancels, while Fubini's theorem and passage through the improper
integral are justified by the positive orbit separation.  Under this separation hypothesis,
\begin{equation*}
      \langle B_b f,g\rangle    =\iint K_b^D(x,y)f(y)\overline{g(x)}\,   d\omega(y)d\omega(x),
\end{equation*}
where
\begin{equation*}
        K_b^D(x,y)=\bigl(b(x)-b(y)\bigr)J_D(x,y)
\end{equation*}
for this initial class of test functions.

For compactly supported test functions with disjoint Euclidean supports, the integral determined by $K_b^D$ is
absolutely convergent.  Indeed, $\|x-y\|$ is bounded below on the product of the two supports, and
\eqref{cc:eq:half-CZO-size} below provides
\begin{equation*}
        |K_b^D(x,y)|    \lesssim \lVert b\rVert_{\operatorname{Lip}_d}   \frac{d(x,y)}{\lVert x-y\rVert}   \frac{1}{V(x,y,d(x,y))}.
\end{equation*}
Since $\lVert x-y\rVert$ is bounded below on the product of the two supports, orbit-annular summation makes this majorant integrable near every reflected diagonal: the annulus $d(x,y)\simeq2^{-k}$ gains the
summable factor $2^{-k}/\lVert x-y\rVert$. 
 The identification of this integral with the distributional form,
including the exclusion of an additional distribution supported on the orbit diagonal, is proved in
Lemma~\ref{cc:lem:half-kernel}.

\begin{proposition}\label{cc:prop:half-CZO}
Let $b\in\operatorname{Lip}_d(\R^N)$ and write
$L_b=\lVert b\rVert_{\operatorname{Lip}_d}$.  Then
\begin{equation}\label{cc:eq:half-CZO-size}
      |K_b^D(x,y)|  \lesssim L_b   \frac{d(x,y)}{\lVert x-y\rVert}    \frac{1}{V(x,y,d(x,y))}.
\end{equation}
If $\lVert x-x'\rVert\leq d(x,y)/2$, then
\begin{equation}\label{cc:eq:half-CZO-reg}
      |K_b^D(x,y)-K_b^D(x',y)|   \lesssim L_b  \frac{\lVert x-x'\rVert}{\lVert x-y\rVert}   \frac{1}{V(x,y,d(x,y))}.
\end{equation}
The analogous estimate holds in the second variable.  At the orbit scale $d(x,y)$ the two endpoint volumes are comparable 
by doubling and $G$-invariance.  Hence the displayed normalization agrees with
\eqref{cc:eq:dCZ-size}--\eqref{cc:eq:dCZ-smooth}, and $B_b$ belongs, at the kernel level, to that
singular-integral class with exponent one.
\end{proposition}

\begin{proof}
Let
$$
    \delta=d(x,y),\qquad  \rho=\lVert x-y\rVert,\qquad     V_\delta=V(x,y,\delta).
$$
By the orbit Lipschitz estimate and \eqref{cc:eq:D-kernel-size},
\begin{align*}
     |K_b^D(x,y)|  &=|b(x)-b(y)|\,|J_D(x,y)| \leq L_b\delta\,|J_D(x,y)| \lesssim
     L_b\frac{\delta}{\rho V_\delta}.
\end{align*}
Let
$
    h=\lVert x-x'\rVert\leq\frac{\delta}{2}.
$
Then
$$
    d(x',y)\leq d(x,y)+\lVert x-x'\rVert    \leq\frac{3\delta}{2},
$$
and
\begin{align*}
     K_b^D(x,y)-K_b^D(x',y)  &=\bigl(b(x)-b(x')\bigr)J_D(x,y) +\bigl(b(x')-b(y)\bigr)
       \bigl(J_D(x,y)-J_D(x',y)\bigr).
\end{align*}
Therefore, by \eqref{cc:eq:D-kernel-size} and  \eqref{cc:eq:D-kernel-reg},
\begin{align*}
     |K_b^D(x,y)-K_b^D(x',y)|  &\leq   L_bh\,|J_D(x,y)|   +L_bd(x',y)\,
        |J_D(x,y)-J_D(x',y)|\\   
        &\lesssim   L_bh\frac{1}{\rho V_\delta}   +L_b\delta\frac{h}{\rho\delta V_\delta} \lesssim
       L_b\frac{h}{\rho V_\delta}.
\end{align*}
If $\lVert y-y'\rVert\leq\delta/2$, then
$
    d(x,y')\leq\frac{3\delta}{2}
$
and
\begin{align*}
     K_b^D(x,y)-K_b^D(x,y')  &=\bigl(b(y')-b(y)\bigr)J_D(x,y) +\bigl(b(x)-b(y')\bigr)
       \bigl(J_D(x,y)-J_D(x,y')\bigr).
\end{align*}
Using the second-variable estimate in Proposition~\ref{cc:prop:D-kernel},
\begin{align*}
     |K_b^D(x,y)-K_b^D(x,y')|  &\lesssim
     L_b\frac{\lVert y-y'\rVert}{\rho V_\delta}.
\end{align*}

Finally, choose $\mathsf g\in G$ such that
$\lVert x-\mathsf g y\rVert=\delta$.  Doubling and $G$-invariance imply
$$
    \omega(B(x,\delta))  \approx\omega(B(\mathsf g y,\delta))    =\omega(B(y,\delta)).
$$
Hence the preceding estimates agree with
\eqref{cc:eq:dCZ-size} and \eqref{cc:eq:dCZ-smooth} for  $\varepsilon=1$.
\end{proof}

The preceding pointwise kernel has the following canonical  distributional extension.  
This step excludes an additional term  supported on the orbit diagonal and will also be used in the weak
boundedness argument.

\begin{lemma}\label{cc:lem:half-kernel}
For every compactly supported Lipschitz function $F$ on $\R^N\times\R^N$, let
\begin{equation*}
      \langle\mathscr K_b^D,F\rangle  :=\frac12\iint K_b^D(x,y)\bigl(F(x,y)-F(y,x)\bigr)
      \,d\omega(y)d\omega(x).
\end{equation*}
The integral is absolutely convergent and is continuous on functions supported in a fixed compact set.  Its restriction to
$C_c^\infty(\R^N\times\R^N)$ is the Schwartz kernel of \eqref{cc:eq:half-form}.  In particular, if $f,g$ are
compactly supported Lipschitz functions with disjoint supports, then
\begin{equation*}
     \langle B_b f,g\rangle   =\iint K_b^D(x,y)f(y)\overline{g(x)}\,   d\omega(y)d\omega(x),
\end{equation*}
and the last integral is absolutely convergent even when the two
supports meet after a reflection.
\end{lemma}

\begin{proof}
Write
$$
    F^{\mathrm t}(x,y):=F(y,x),\qquad
    A:=\|F\|_\infty,\qquad
    L:=\sup_y[F(\cdot,y)]_{C^1}+\sup_x[F(x,\cdot)]_{C^1}.
$$
If $A=0$, there is nothing to prove.  Otherwise $L>0$, and let $r=2A/L$.  With
$\delta=d(x,y)$ and $\rho=\lVert x-y\rVert$, the separate Lipschitz estimates yield
\begin{align*}
     |F(x,y)-F^{\mathrm t}(x,y)|   &\leq\min\{2A,L\rho\}
     =2A\min\left\{1,\frac{\rho}{r}\right\}.
\end{align*}
Hence, by \eqref{cc:eq:half-CZO-size},
\begin{align*}
     |K_b^D(x,y)|\,|F(x,y)-F^{\mathrm t}(x,y)|  &\lesssim   L_bA\frac{\delta}{\rho}
     \min\left\{1,\frac{\rho}{r}\right\}   \frac1{V(x,y,\delta)} \leq
      L_bA\frac{\min\{1,\delta/r\}}{V(x,y,\delta)}.
\end{align*}

Choose $R>0$ such that
$$
    \operatorname{supp}F\cup\operatorname{supp}F^{\mathrm t}   \subset B(0,R)\times B(0,R),
$$
and let
$$
    E_k(x):=\{y:2^kr<d(x,y)\leq2^{k+1}r\},    \qquad k\in\mathbb Z.
$$
Doubling and $G$-invariance imply
\begin{align*}
     \int_{E_k(x)}\frac{d\omega(y)}{V(x,y,d(x,y))}  &\leq
     \frac{\omega(\mathcal O(B(x,2^{k+1}r)))}    {\omega(B(x,2^kr))}
     \lesssim1.
\end{align*}
Since $d(x,y)\leq2R$ on $B(0,R)\times B(0,R)$,
\begin{align*}
     \int_{B(0,R)}   \frac{\min\{1,d(x,y)/r\}}{V(x,y,d(x,y))}\,d\omega(y)
     &\lesssim   \sum_{2^kr<2R}\min\{1,2^{k+1}\}  \lesssim
     1+\log\left(1+\frac{2R}{r}\right).
\end{align*}
Consequently,
\begin{align*}
     |\langle\mathscr K_b^D,F\rangle|  &\lesssim_R
     L_bA\left[1+\log\left(1+\frac{2R}{r}\right)\right]  \lesssim_R L_b(A+L).
\end{align*}
Thus the defining integral is absolutely convergent and continuous in the compactly supported Lipschitz topology.

For $0<\varepsilon<\Lambda<\infty$, define
\begin{align*}
     D_{\varepsilon,\Lambda}u   :=\frac1{2\sqrt\pi}\int_\varepsilon^\Lambda
     (u-H_tu)\,\frac{dt}{t^{3/2}},\qquad   J_{\varepsilon,\Lambda}^D(x,y)
     :=-\frac1{2\sqrt\pi}\int_\varepsilon^\Lambda    h_t(x,y)\,\frac{dt}{t^{3/2}},
\end{align*}
and 
$$
     K_{b;\varepsilon,\Lambda}^D(x,y)   :=(b(x)-b(y))J_{\varepsilon,\Lambda}^D(x,y).
$$
The symmetry and nonnegativity of $h_t$ imply
\begin{align*}
     K_{b;\varepsilon,\Lambda}^D(y,x)  =-K_{b;\varepsilon,\Lambda}^D(x,y)\qquad \text{and}\qquad 
     |K_{b;\varepsilon,\Lambda}^D(x,y)|  \leq|K_b^D(x,y)|.
\end{align*}
For
$$
    F_{f,g}(x,y):=f(y)\overline{g(x)}, \qquad f,g\in C_c^\infty(\R^N),
$$
the bounded operator $D_{\varepsilon,\Lambda}$ is self-adjoint.  Hence the identity terms cancel, and Fubini's theorem gives
\begin{align*}
     &\langle D_{\varepsilon,\Lambda}f,\overline b g\rangle  -\langle D_{\varepsilon,\Lambda}(bf),g\rangle   =
     \frac12\iint K_{b;\varepsilon,\Lambda}^D(x,y)   \bigl(F_{f,g}(x,y)-F_{f,g}^{\mathrm t}(x,y)\bigr)
     \,d\omega(y)d\omega(x).
\end{align*}
Since
$$
    D_{\varepsilon,\Lambda}u\longrightarrow Du
    \qquad\text{in }L^2(d\omega),\qquad u\in\operatorname{Dom}(D),
$$
the preceding majorant permits dominated convergence:
\begin{align*}
     \langle B_bf,g\rangle  &=   \frac12\iint K_b^D(x,y)
     \bigl(F_{f,g}(x,y)-F_{f,g}^{\mathrm t}(x,y)\bigr)
     \,d\omega(y)d\omega(x)  =\langle\mathscr K_b^D,F_{f,g}\rangle.
\end{align*}
Hence the restriction of $\mathscr K_b^D$ to smooth tests is the Schwartz kernel of
\eqref{cc:eq:half-form}, with no additional distribution supported on the orbit diagonal.  We use the preceding
absolutely convergent formula as its canonical extension to compactly supported Lipschitz tests.

Finally, if $\operatorname{supp}f\cap\operatorname{supp}g=\varnothing$, let
$$
    \sigma  :=\inf\bigl\{\lVert x-y\rVert:
    x\in\operatorname{supp}g,\ y\in\operatorname{supp}f\bigr\}>0.
$$
On the support of $F_{f,g}$,
$$
    \frac{d(x,y)}{\lVert x-y\rVert}
    \leq\min\left\{1,\frac{d(x,y)}{\sigma}\right\},
$$
so the preceding orbit-annular estimate shows that
$$
    \iint|K_b^D(x,y)f(y)g(x)|\,d\omega(y)d\omega(x)<\infty.
$$
Using $K_b^D(y,x)=-K_b^D(x,y)$ and changing variables,
\begin{align*}
     \langle B_bf,g\rangle  &=\langle\mathscr K_b^D,F_{f,g}\rangle  =\iint K_b^D(x,y)f(y)\overline{g(x)}
     \,d\omega(y)d\omega(x).
\end{align*}
The same calculation applies to compactly supported Lipschitz $f,g$, because both displayed integrals are absolutely convergent.
\end{proof}

For every compact set $K$, the same majorant implies
$$
     |\langle B_bf,g\rangle|  \leq C_KL_b  \bigl(\|f\|_\infty+[f]_{C^1}\bigr)
     \bigl(\|g\|_\infty+[g]_{C^1}\bigr),
$$
where $[f]_{C^1}:=\sup_{x\ne y}|f(x)-f(y)|/\|x-y\|$.  This holds whenever $f,g\in C_0^1(\R^N)$ are supported in $K$.  Hence
$B_b:C_0^1(\R^N)\to(C_0^1(\R^N))'$ is continuous.  Together with
Proposition~\ref{cc:prop:half-CZO}, this shows that $B_b$ is a Dunkl--Calder\'on--Zygmund singular integral with $\eta=1$, which is the
specialization used in Theorem~\ref{cc:thm:dunkl-T1}.

We next verify the joint-test weak boundedness property stated above.  This avoids importing any operator bound for
$B_b$ into its own $T1$ proof.

\begin{lemma}\label{cc:lem:WBP-annuli}
For every fixed $A\geq1$ and every $r>0$,
$$
      \sup_x\int_{d(x,y)\leq Ar}  \frac{\min\{1,d(x,y)/r\}}   {V(x,y,d(x,y))}\,d\omega(y)   \lesssim_A 1.
$$
\end{lemma}

\begin{proof}
Decompose the orbit ball into the annuli
$$
      2^{-k-1}Ar<d(x,y)\leq2^{-k}Ar,   \qquad k\geq0.
$$
Their measures are at most
$$
      \omega\bigl(\mathcal O(B(x,2^{-k}Ar))\bigr)   \leq M_G\,\omega(B(x,2^{-k}Ar)).
$$
On the same annulus the denominator is at least  $\omega(B(x,2^{-k-1}Ar))$.  
Doubling therefore bounds the $k$th contribution by $C\min\{1,A2^{-k}\}$, and the resulting geometric
series is finite.
\end{proof}

\begin{proposition}\label{cc:prop:half-WBP}
The distribution kernel $K_b^D$ satisfies the joint-test weak boundedness property, with weak boundedness constant
at most $CL_b$.
\end{proposition}

\begin{proof}
Let
$$
     B_x:=B(x_0,r),\qquad B_y:=B(y_0,r),\qquad    F^{\mathrm t}(x,y):=F(y,x).
$$
Since $J_D$ is symmetric,
$$
     K_b^D(y,x)=-K_b^D(x,y).
$$
Hence, by Lemma~\ref{cc:lem:half-kernel},
$$
     \langle\mathscr K_b^D,F\rangle   =\frac12\iint K_b^D(x,y)   \bigl(F(x,y)-F(y,x)\bigr)\,d\omega(y)d\omega(x).
$$
The normalized joint-test estimates imply
\begin{align*}
     |F(x,y)-F(y,x)|  &\leq |F(x,y)-F(y,y)|
      +|F(y,y)-F(y,x)|  \lesssim \min\left\{1,\frac{\lVert x-y\rVert}{r}\right\}.
\end{align*}
Since $d(x,y)\leq\lVert x-y\rVert$,
$$
     \frac{d(x,y)}{\lVert x-y\rVert}   \min\left\{1,\frac{\lVert x-y\rVert}{r}\right\}
     \leq     \min\left\{1,\frac{d(x,y)}{r}\right\}.
$$
Therefore, by \eqref{cc:eq:half-CZO-size},
\begin{align*}
     |\langle\mathscr K_b^D,F\rangle|   &\lesssim L_b  \iint_{(B_x\times B_y)\cup(B_y\times B_x)}
     \frac{\min\{1,d(x,y)/r\}}  {V(x,y,d(x,y))}  \,d\omega(y)d\omega(x).
\end{align*}

Let $r_0=d(x_0,y_0)$.  If $r_0\leq4r$, then
$$
     d(x,y)\leq r+r_0+r\leq6r
$$
on $(B_x\times B_y)\cup(B_y\times B_x)$.  Thus
Lemma~\ref{cc:lem:WBP-annuli} applies:
\begin{align*}
     |\langle\mathscr K_b^D,F\rangle|   &\lesssim L_b
     \left(   \int_{B_x}\int_{d(x,y)\leq6r}   \frac{\min\{1,d(x,y)/r\}}   {V(x,y,d(x,y))}
     \,d\omega(y)d\omega(x)\right.\\
     &\hspace{32mm}\left.   +\int_{B_y}\int_{d(x,y)\leq6r}   \frac{\min\{1,d(x,y)/r\}}
     {V(x,y,d(x,y))}   \,d\omega(y)d\omega(x)   \right)\\
     &\lesssim L_b\bigl(\omega(B_x)+\omega(B_y)\bigr)  \lesssim L_b\max\{\omega(B_x),\omega(B_y)\}.
\end{align*}

If $r_0>4r$, then
$$
     d(x,y)\geq r_0-2r>2r
$$
on the same set.  For $(x,y)\in B_x\times B_y$,
$$
     B_x\subset B(x,2r)\subset B(x,d(x,y)),  \qquad
     B_y\subset B(y,2r)\subset B(y,d(x,y)),
$$
and hence
$$
     V(x,y,d(x,y))    \geq\max\{\omega(B_x),\omega(B_y)\}.
$$
The same estimate holds on $B_y\times B_x$.  Consequently,
\begin{align*}
     |\langle\mathscr K_b^D,F\rangle|   &\lesssim
      L_b\frac{\omega(B_x)\omega(B_y)}  {\max\{\omega(B_x),\omega(B_y)\}}  \lesssim
     L_b\max\{\omega(B_x),\omega(B_y)\}.
\end{align*}
\end{proof}

\subsection{The two T1 data}

The $T1$ data have an exact first-order form.  We first relate the formal identity to the cutoff definition
\eqref{cc:eq:T1-def}.  Let $\phi$ be a compactly supported mean-zero test function with
$\supp\phi\subset B(x_0,R)$, and choose a $G$-invariant function $\psi\in C_c^1$ which equals one on
$\mathcal O(B(x_0,2R))$ and vanishes outside $\mathcal O(B(x_0,4R))$.
Then
\begin{align*}
      \langle B_b 1,\phi\rangle   &=\langle B_b\psi,\phi\rangle +
      \iint\bigl[K_b^D(x,y)-K_b^D(x_0,y)\bigr]   (1-\psi(y))\overline{\phi(x)}\,
      d\omega(y)d\omega(x).
\end{align*}
This is the cutoff formula \eqref{cc:eq:T1-def}.  The subtraction is essential: the unsubtracted tail
need not be absolutely integrable for an unbounded symbol.  The displayed tail is supported where
$d(x_0,y)\gtrsim R$.  Its portion in a fixed orbit annulus is absolutely integrable by
\eqref{cc:eq:half-CZO-size}, while on the remaining far part the mean-zero condition permits the displayed
subtraction.
For $d(x_0,y)\simeq2^kR$, $k\geq1$, the smoothness estimate
\eqref{cc:eq:half-CZO-reg} implies, uniformly for $x\in B(x_0,R)$,
\begin{align*}
      &\int_{d(x_0,y)\simeq2^kR}|K_b^D(x,y)-K_b^D(x_0,y)|\,d\omega(y)  \lesssim
     L_b\frac{R}{2^kR}  \frac{\omega(\mathcal O(B(x_0,C2^kR)))}{\omega(B(x_0,c2^kR))}
     \lesssim L_b2^{-k}.
\end{align*}
Orbit-annular summation therefore proves absolute convergence of the far tail.

We identify the cutoff distribution directly, without applying a two-ended spectral truncation to the unbounded
function $b$.  Let $\chi_L$ be a radial smooth cutoff which equals one on $B(0,L)$, vanishes outside
$B(0,2L)$, and satisfies $\|\nabla\chi_L\|_\infty\lesssim L^{-1}$.  Define
$$
     u_L=(1-\psi)\chi_L.
$$
For large $L$, the supports of $u_L$ and $\phi$ have positive orbit separation.  The associated-kernel identity
and the mean-zero condition give
\begin{align*}
     \langle B_bu_L,\phi\rangle  &=\iint\bigl[K_b^D(x,y)-K_b^D(x_0,y)\bigr]
     u_L(y)\overline{\phi(x)}\,d\omega(y)d\omega(x)\\
     &=\langle Du_L,\overline b\phi\rangle   -\langle bu_L,D\phi\rangle.
\end{align*}
The double integral converges to the cutoff tail by the annular majorant above.  On $\supp\phi$, the functions $u_L$ vanish in a
fixed orbit neighborhood.  Proposition~\ref{cc:prop:D-kernel} and orbit-annular summation then imply, locally uniformly there,
\begin{align*}
     Du_L(x)  &=\int_{\R^N}J_D(x,y)(1-\psi(y))\chi_L(y)\,d\omega(y)\\
     &\longrightarrow\int_{\R^N}J_D(x,y)(1-\psi(y))\,d\omega(y)\\
     &=-\frac1{2\sqrt\pi}\int_0^\infty\bigl(1-H_t\psi(x)\bigr)\frac{dt}{t^{3/2}}
     =-D\psi(x).
\end{align*}
Here the last line uses $H_t1=1$ and $\psi(x)=1$ on $\supp\phi$; no spectral meaning is assigned to $D(1-\psi)$.
Moreover, the cancellation of $\phi$ and the second-variable estimate in
Proposition~\ref{cc:prop:D-kernel} imply, outside a fixed orbit ball,
$$
     |D\phi(y)|  \lesssim  \frac{R\|\phi\|_{L^1(d\omega)}}
     {\|y-x_0\|\,d(y,x_0)\,V(y,x_0,d(y,x_0))}.
$$
Let $E_k=\{y:2^kR\leq d(y,x_0)<2^{k+1}R\}$, $k\geq1$.  Since $\|y\| \leq d(y,x_0)+\|x_0\|$, and 
\begin{align*}
     |b(y)|\leq |b(0)|+L_b\bigl(d(y,x_0)+\|x_0\|\bigr),\qquad
     \|y-x_0\|\geq d(y,x_0),
\end{align*}
$G$-invariance, doubling, and the preceding estimate lead to
\begin{align*}
     \int_{E_k}|b(y)D\phi(y)|\,d\omega(y)    &\lesssim
     \|\phi\|_{L^1(d\omega)}   \left[  \frac{|b(0)|+L_b\|x_0\|}{4^kR}+L_b2^{-k}   \right].
\end{align*}
The series is summable.  On bounded sets, $b$ is bounded, $D\phi\in L^2(d\omega)$, and $d\omega$ is locally finite; hence
$bD\phi\in L^1(d\omega)$ by Cauchy--Schwarz.  Therefore
$$
     \langle bu_L,D\phi\rangle  \longrightarrow   \langle b(1-\psi),D\phi\rangle.
$$
The cutoff tail is consequently
$$
     -\langle D\psi,\overline b\phi\rangle   -\langle b(1-\psi),D\phi\rangle.
$$
Adding the local term
$$
     \langle B_b\psi,\phi\rangle   =\langle D\psi,\overline b\phi\rangle   -\langle b\psi,D\phi\rangle
$$
yields the exact identity
\begin{equation}\label{cc:eq:half-T1}
     \langle B_b1,\phi\rangle=-\langle b,D\phi\rangle.
\end{equation}

Only mean-zero tests occur here, as is appropriate for BMO modulo constants.  The kernel estimates in
\cite[Theorem~1.1]{HanLeeLiWick2024} show that every $\mathcal R_k$ is a
Dunkl--Calder\'on--Zygmund singular integral.  Together with its $L^2$ boundedness,
\cite[Theorem~2.1]{TanHanHanLeeLi2025} supplies
$$
     \mathcal R_k:H^1(d\omega)\longrightarrow L^1(d\omega),
    \qquad   \mathcal R_k:L^\infty(d\omega)\longrightarrow\operatorname{BMO}(d\omega).
$$
We now justify the identity for $Db$.  Since both $b$ and $\chi_L$ are $G$-invariant, the product rule and
\eqref{cc:eq:D-Riesz} imply
$$
     D(\chi_Lb)  =\sum_{k=1}^N\mathcal R_k\bigl(\chi_L\partial_kb+b\,\partial_k\chi_L\bigr).
$$
For every compactly supported mean-zero test function $\phi$,
$$
     \langle D(\chi_Lb),\phi\rangle   =\langle\chi_Lb,D\phi\rangle   \longrightarrow\langle b,D\phi\rangle,
$$
where the convergence follows from the integrability of $bD\phi$ proved above.  Such a $\phi$ is a scalar multiple of a
Coifman--Weiss $(1,\infty)$-atom and hence belongs to $H^1(d\omega)$; see \cite{DziubanskiHejna2020Atomic}.  Therefore
$\mathcal R_k\phi\in L^1(d\omega)$.  Since
$\mathcal R_k^*=-\mathcal R_k$,
$$
     \langle\mathcal R_k(b\,\partial_k\chi_L),\phi\rangle
     =-\langle b\,\partial_k\chi_L,\mathcal R_k\phi\rangle\longrightarrow0.
$$
Indeed, $\|b\,\partial_k\chi_L\|_\infty\lesssim L_b+|b(0)|/L$, and this function converges pointwise to zero.
Likewise, on mean-zero tests,
$$
     \mathcal R_k(\chi_L\partial_kb)\longrightarrow\mathcal R_k(\partial_kb).
$$
Consequently,
$$
     Db=\sum_{k=1}^N\mathcal R_k(\partial_kb)
$$
modulo constants.  Thus \eqref{cc:eq:half-T1} says precisely that $B_b1=-Db$ as a BMO class, and the
second endpoint estimate above implies
\begin{equation}\label{cc:eq:half-T1-BMO}
        \lVert B_b 1\rVert_{\operatorname{BMO}(d\omega)}   \lesssim L_b.
\end{equation}
Moreover,
$
        B_b^*=-B_{\overline b},
$
and hence
\begin{equation}\label{cc:eq:half-Tstar1}
        B_b^*1=D\overline b
        =\sum_{k=1}^N\mathcal R_k(\partial_k\overline b),
        \qquad
        \lVert B_b^*1\rVert_{\operatorname{BMO}(d\omega)}
        \lesssim L_b.
\end{equation}

\begin{theorem}\label{cc:thm:half-comm}
Let $b\in\operatorname{Lip}_d(\R^N)$.  Then $[M_b,D]$, initially
defined as a distribution on compactly supported Lipschitz test
functions, extends boundedly to $L^p(d\omega)$ for every
$1<p<\infty$, and
\begin{equation*}
        \lVert[M_b,D]\rVert_{L^p(d\omega)\to L^p(d\omega)}
        \lesssim_p \lVert b\rVert_{\operatorname{Lip}_d}.
\end{equation*}
\end{theorem}

\begin{proof}
Write
$$
     B_b=[M_b,D],\qquad L_b=\lVert b\rVert_{\operatorname{Lip}_d}.
$$
By Proposition~\ref{cc:prop:half-CZO}, the kernel $K_b^D$ satisfies
\eqref{cc:eq:dCZ-size}--\eqref{cc:eq:dCZ-smooth} with
$\varepsilon=1$ and kernel constant bounded by $CL_b$.  Moreover, Lemma~\ref{cc:lem:half-kernel} yields
$$
     \langle B_bf,g\rangle  =\iint K_b^D(x,y)f(y)\overline{g(x)}  \,d\omega(y)d\omega(x)
$$
whenever $f,g\in C_c^1(\R^N)$ have disjoint supports. Thus $B_b$ is
associated with $K_b^D$ in the sense required by
Theorem~\ref{cc:thm:dunkl-T1}.

Proposition~\ref{cc:prop:half-WBP} and
\eqref{cc:eq:half-T1-BMO}--\eqref{cc:eq:half-Tstar1} provide
\begin{align*}
     C_{\mathrm{WBP}}(B_b)\lesssim L_b\qquad \text{and}\qquad 
     \lVert B_b1\rVert_{\operatorname{BMO}(d\omega)}
     +\lVert B_b^*1\rVert_{\operatorname{BMO}(d\omega)}   \lesssim L_b.
\end{align*}
Theorem~\ref{cc:thm:dunkl-T1} now yields
$$
     \lVert B_b\rVert_{L^2(d\omega)\to L^2(d\omega)}   \lesssim L_b
$$
and, for every $1<p<\infty$,
$$
     \lVert B_b\rVert_{L^p(d\omega)\to L^p(d\omega)}   \lesssim_p L_b.
$$
\end{proof}

\subsection{Exact factorization of the Calder\'on commutator}

We now transfer Theorem~\ref{cc:thm:half-comm} to the first principal result.

\begin{lemma}\label{cc:lem:graph-factor}
Suppose that $b\in\operatorname{Lip}_d(\R^N)\cap L^\infty(d\omega)$.  Then multiplication by $b$ maps
$\operatorname{Dom}(D)$ continuously into itself, the bounded operator $B_b$ agrees on this domain with the graph
commutator $M_bD-DM_b$, and
\begin{align*}
     M_bT_i\mathcal R_j-T_i\mathcal R_jM_b  &=-M_{\partial_i b}\mathcal R_j
     -\mathcal R_iM_{\partial_jb}   +\mathcal R_iB_b\mathcal R_j
\end{align*}
on $\operatorname{Dom}(D)$.
\end{lemma}

\begin{proof}
Let $L_b=\lVert b\rVert_{\operatorname{Lip}_d}$.  The Dunkl multiplier identities read
$$
     \operatorname{Dom}(D)=\bigcap_{k=1}^N\operatorname{Dom}(T_k),
     \qquad   \lVert Df\rVert_2^2=\sum_{k=1}^N\lVert T_kf\rVert_2^2.
$$
For $f\in\operatorname{Dom}(D)$, the weak product rule
\eqref{cc:eq:weak-product} implies
\begin{align*}
     \lVert D(bf)\rVert_2  &=\left(\sum_{k=1}^N
     \lVert bT_kf+(\partial_kb)f\rVert_2^2\right)^{1/2} \leq \lVert b\rVert_\infty
     \left(\sum_{k=1}^N\lVert T_kf\rVert_2^2\right)^{1/2}
     +\left(\sum_{k=1}^N\lVert(\partial_kb)f\rVert_2^2\right)^{1/2}\\
     &\leq \lVert b\rVert_\infty\lVert Df\rVert_2   +\sqrt N\,L_b\lVert f\rVert_2.
\end{align*}
Consequently,
$$
     \lVert bf\rVert_2+\lVert D(bf)\rVert_2
     \leq\bigl(\lVert b\rVert_\infty+\sqrt N\,L_b\bigr)
     \bigl(\lVert f\rVert_2+\lVert Df\rVert_2\bigr),
$$
so $M_b$ maps $\operatorname{Dom}(D)$ continuously into itself.

Approximation on the Dunkl transform side in
$L^2((1+\lVert\xi\rVert^2)d\omega(\xi))$ shows that
$\mathcal S(\R^N)$ is a graph core for $D$.  Let
$\chi_R(x)=\chi(x/R)$, where $\chi\in C_c^\infty(\R^N)$ is radial and
equals one on $B(0,1)$.  Since $\chi_R$ is $G$-invariant,
$$
     T_k\bigl((\chi_R-1)f\bigr)   =(\chi_R-1)T_kf+(\partial_k\chi_R)f.
$$
For $f\in\mathcal S(\R^N)$,
\begin{align*}
     \lVert(\chi_R-1)f\rVert_2 \longrightarrow0\qquad \text{and}\qquad 
      \lVert T_k\bigl((\chi_R-1)f\bigr)\rVert_2  \leq\lVert(\chi_R-1)T_kf\rVert_2
     +\frac{C}{R}\lVert f\rVert_2    \longrightarrow0.
\end{align*}
Hence $C_c^\infty(\R^N)$ is a graph core for $D$.

Let $f\in\operatorname{Dom}(D)$ and choose
$f_n\in C_c^\infty(\R^N)$ such that
$$
     f_n\longrightarrow f,\qquad Df_n\longrightarrow Df
$$
in $L^2(d\omega)$.  On $C_c^\infty(\R^N)$, the distributional
definition of $B_b$ reads
$$
     B_bf_n=M_bDf_n-DM_bf_n.
$$
The graph continuity of $M_b$ and the $L^2$ boundedness of $B_b$ imply
\begin{align*}
     M_bDf_n\longrightarrow M_bDf,\qquad    DM_bf_n\longrightarrow DM_bf,\qquad 
     B_bf_n\longrightarrow B_bf.
\end{align*}
Therefore,
$$
     B_bf=M_bDf-DM_bf,\qquad f\in\operatorname{Dom}(D).
$$

The operators $\mathcal R_j$ commute with $D$ and preserve
$\operatorname{Dom}(D)$.  Thus
$$
     [M_b,\mathcal R_j]f\in\operatorname{Dom}(D),
     \qquad f\in\operatorname{Dom}(D).
$$
On this domain, \eqref{cc:eq:weak-product} and
\eqref{cc:eq:D-Riesz} imply
$
[M_b,T_i]=-M_{\partial_i b}
$
and
\begin{align*}
     B_b\mathcal R_j+D[M_b,\mathcal R_j]   &=[M_b,D\mathcal R_j]  =[M_b,-T_j]
     =M_{\partial_jb}.
\end{align*}
Hence
$$
     D[M_b,\mathcal R_j]  =M_{\partial_jb}-B_b\mathcal R_j.
$$
Using $T_i=-\mathcal R_iD$, we obtain
\begin{align*}
     [M_b,T_i\mathcal R_j]  &=[M_b,T_i]\mathcal R_j+T_i[M_b,\mathcal R_j]
     =-M_{\partial_i b}\mathcal R_j   -\mathcal R_iD[M_b,\mathcal R_j]   
     =-M_{\partial_i b}\mathcal R_j  -\mathcal R_iM_{\partial_jb}
     +\mathcal R_iB_b\mathcal R_j.
\end{align*}
\end{proof}
\medskip

\begin{proof}[Proof of Theorem~\ref{cc:thm:operator}]
Let
$$
     L_b=\lVert b\rVert_{\operatorname{Lip}_d},  \qquad    B_b=[M_b,D],
$$
and define $\mathcal C_{b,\mathrm{abs}}^{ij}$ by
\eqref{cc:eq:factor}.  Theorem~\ref{cc:thm:half-comm} supplies
$$
     \lVert B_b\rVert_{L^p(d\omega)\to L^p(d\omega)}
     \lesssim_p L_b,\qquad 1<p<\infty.
$$
If $b$ is bounded, Lemma~\ref{cc:lem:graph-factor} implies, for
$f,g\in C_c^\infty(\R^N)$,
\begin{align*}
     \mathfrak c_b^{ij}(f,g)    &=\langle[M_b,T_i\mathcal R_j]f,g\rangle
     =\left\langle   \left(   -M_{\partial_i b}\mathcal R_j   -\mathcal R_iM_{\partial_jb}
     +\mathcal R_iB_b\mathcal R_j   \right)f,g   \right\rangle
     =\langle\mathcal C_{b,\mathrm{abs}}^{ij}f,g\rangle.
\end{align*}

Suppose first that $b$ is real-valued and define
$ 
b_n=\max\{-n,\min\{b,n\}\}.
$ 
Then
$$
     \lVert b_n\rVert_{\operatorname{Lip}_d}\leq L_b,  \qquad
     \partial_kb_n\longrightarrow\partial_kb\quad\text{a.e.},  \qquad
     |\partial_kb_n|\leq L_b.
$$
In particular,
$$
     \sup_n\lVert B_{b_n}\rVert_{L^2(d\omega)\to L^2(d\omega)}
     +\lVert B_b\rVert_{L^2(d\omega)\to L^2(d\omega)}   \lesssim L_b.
$$
For $\phi,\psi\in C_c^\infty(\R^N)$, one has $b_n=b$ on
$\operatorname{supp}\phi\cup\operatorname{supp}\psi$ for all sufficiently large $n$.  Hence
\begin{align*}
     \langle B_{b_n}\phi,\psi\rangle  &=\langle D\phi,b_n\psi\rangle-\langle b_n\phi,D\psi\rangle
     =\langle D\phi,b\psi\rangle-\langle b\phi,D\psi\rangle
     =\langle B_b\phi,\psi\rangle.
\end{align*}
For arbitrary $u,v\in L^2(d\omega)$ and
$\phi,\psi\in C_c^\infty(\R^N)$,
\begin{align*}
     &|\langle(B_{b_n}-B_b)u,v\rangle|\\ 
     \leq\,& |\langle(B_{b_n}-B_b)(u-\phi),v\rangle| +  |\langle(B_{b_n}-B_b)\phi,v-\psi\rangle|  +|\langle(B_{b_n}-B_b)\phi,\psi\rangle|\\
     \lesssim\,& L_b\bigl(\lVert u-\phi\rVert_2\lVert v\rVert_2  +\lVert\phi\rVert_2\lVert v-\psi\rVert_2\bigr) +  |\langle(B_{b_n}-B_b)\phi,\psi\rangle|.
\end{align*}
Letting $n\to\infty$ and then using the density of
$C_c^\infty(\R^N)$ in $L^2(d\omega)$ implies
$$
     \langle B_{b_n}u,v\rangle   \longrightarrow  \langle B_bu,v\rangle,    \qquad u,v\in L^2(d\omega).
$$

For every $h\in L^2(d\omega)$, dominated convergence yields
\begin{align*}
     \lVert(M_{\partial_kb_n}-M_{\partial_kb})h\rVert_2^2  &=\int_{\R^N}
     |\partial_kb_n-\partial_kb|^2|h|^2\,d\omega \longrightarrow0.
\end{align*}
Consequently, for every $u,v\in L^2(d\omega)$,
\begin{align*}
     \langle\mathcal R_iB_{b_n}\mathcal R_ju,v\rangle
      &=\langle B_{b_n}\mathcal R_ju,\mathcal R_i^*v\rangle \longrightarrow
     \langle B_b\mathcal R_ju,\mathcal R_i^*v\rangle
     =\langle\mathcal R_iB_b\mathcal R_ju,v\rangle.
\end{align*}
Furthermore, for every $f\in L^2(d\omega)$,
\begin{align*}
     \lVert(M_{\partial_i b_n}-M_{\partial_i b})\mathcal R_jf\rVert_2  &\longrightarrow0,\\
     \lVert\mathcal R_i(M_{\partial_jb_n}-M_{\partial_jb})f\rVert_2  &\leq
     \lVert\mathcal R_i\rVert_{2\to2}  \lVert(M_{\partial_jb_n}-M_{\partial_jb})f\rVert_2
     \longrightarrow0.
\end{align*}

Let $A_{ij}=T_i\mathcal R_j$.  Its Dunkl multiplier is
$\xi_i\xi_j/\lVert\xi\rVert$, and hence $A_{ij}$ is self-adjoint.
For $f,g\in C_c^\infty(\R^N)$ and all sufficiently large $n$,
\begin{align*}
     \mathfrak c_{b_n}^{ij}(f,g)  &=\langle A_{ij}f,b_ng\rangle-\langle b_nf,A_{ij}g\rangle
     =\langle A_{ij}f,bg\rangle-\langle bf,A_{ij}g\rangle
     =\mathfrak c_b^{ij}(f,g).
\end{align*}
Applying the bounded-symbol identity to $b_n$ and passing to the limit, we obtain
\begin{align*}
     \mathfrak c_b^{ij}(f,g)   &=\lim_{n\to\infty}\mathfrak c_{b_n}^{ij}(f,g)\\
     &=\lim_{n\to\infty}  \left\langle   \left(  -M_{\partial_i b_n}\mathcal R_j
     -\mathcal R_iM_{\partial_jb_n}  +\mathcal R_iB_{b_n}\mathcal R_j  \right)f,g  \right\rangle\\
     &=\left\langle   \left(   -M_{\partial_i b}\mathcal R_j  -\mathcal R_iM_{\partial_jb} 
      +\mathcal R_iB_b\mathcal R_j  \right)f,g  \right\rangle 
      =\langle\mathcal C_{b,\mathrm{abs}}^{ij}f,g\rangle.
\end{align*}

For a complex-valued symbol, write $b=u+iv$.  Since
$$
      \mathfrak c_b^{ij}  =\mathfrak c_u^{ij}+i\mathfrak c_v^{ij},
     \qquad   B_b=B_u+iB_v,   \qquad  \partial_kb=\partial_ku+i\partial_kv,
$$
the preceding identity applies separately to $u$ and $v$.

Finally, the $L^p$ boundedness of the Dunkl Riesz transforms
\cite[Theorem~3.3]{AmriSifi2012} and
$\lVert\partial_kb\rVert_\infty\leq L_b$ together imply
\begin{align*}
     \lVert\mathcal C_{b,\mathrm{abs}}^{ij}\rVert_{L^p\to L^p}  &\leq
     \lVert\partial_i b\rVert_\infty  \lVert\mathcal R_j\rVert_{L^p\to L^p}
     +\lVert\mathcal R_i\rVert_{L^p\to L^p}  \lVert\partial_jb\rVert_\infty
     +  \lVert\mathcal R_i\rVert_{L^p\to L^p}   \lVert B_b\rVert_{L^p\to L^p}
     \lVert\mathcal R_j\rVert_{L^p\to L^p}   \lesssim_p L_b.
\end{align*}
If $S_1$ and $S_2$ are two bounded realizations of
$\mathfrak c_b^{ij}$ on $L^p(d\omega)$, then
$$
     \langle(S_1-S_2)f,g\rangle=0,   \qquad f,g\in C_c^\infty(\R^N).
$$
The density of $C_c^\infty(\R^N)$ in $L^p(d\omega)$ and
$L^{p'}(d\omega)$ forces $S_1=S_2$.
\end{proof}

\begin{remark}
The factorization constructs the Calder\'on commutator independently of the prescribed heat truncations.  If
$f,g\in C_c^\infty$, then $f,bf\in\operatorname{Dom}(D)$, and the order-one finite heat multipliers converge in
$L^2$ on this domain.  Hence the finite heat formula converges to
\eqref{cc:eq:factor} in such test-function pairings.  The separated-support kernel is identified
below in Section~\ref{cc:sec:matrix}.

This argument does not, by itself, assert the stronger statement
\begin{equation*}
      \sup_{0<a<B<\infty}    \lVert C_{a,B}^{ij,b}\rVert_{2\to2}<\infty
\end{equation*}
for the particular two-ended heat truncations used earlier in the
paper.  Uniform boundedness of that prescribed approximation family
requires an additional scale-sensitive argument.  The integrated
chamber $T1$ subsection below supplies precisely that argument: it
proves the uniform estimate, joint endpoint convergence, and
identification with the operator in
\eqref{cc:eq:factor}.  Thus the factorization provides an
independent full-space construction, whereas chamber lifting proves uniform boundedness and joint convergence for
the prescribed heat family and retains its matrix-kernel information.
\end{remark}

\section{Kernel estimates}\label{cc:sec:kernels}

Throughout this section fix $i,j\in\{1,\ldots,N\}$ and
$b\in\operatorname{Lip}_d(\R^N)$.  We write
$$
     A_t^{ij}(x,y)=T_{i,x}T_{j,x}h_t(x,y).
$$
The estimates in this section have two consequences.  The first is a scale estimate for
$$
          \theta_s(x,y)=s\bigl(b(x)-b(y)\bigr)T_{i,x}T_{j,x}h_{s^2}(x,y),
$$
retaining an orbit Gaussian and first-order Euclidean regularity in both variables.  These bounds drive the
component and integrated tests.  The second is the standard-kernel estimate for
$$
     K_b^{ij}(x,y)=c_\ast\bigl(b(x)-b(y)\bigr)   \int_0^\infty T_{i,x}T_{j,x}h_t(x,y)\frac{dt}{\sqrt t}.
$$
The proof first resolves $T_{i,x}T_{j,x}h_t$ into its principal, scalar, and reflected heat-kernel terms.  The
Euclidean correction in the heat estimate absorbs the principal coefficient, while the orbit Lipschitz difference
supplies the factor needed at every reflected diagonal.  We keep the heat parameter visible until the scale bounds
are complete and only then integrate in $t$.

\subsection{First and second Dunkl derivatives of the heat kernel}

\begin{lemma}\label{cc:lem:heat-identity}
For every $j$,
\begin{equation*}
    T_{j,x}h_t(x,y)=\frac{y_j-x_j}{2t}h_t(x,y).
\end{equation*}
\end{lemma}

\begin{lemma}\label{cc:lem:heat-first}
For every $\ell$,
\begin{equation}\label{cc:eq:heat-first}
    |T_{\ell,x}h_t(x,y)|\leq C t^{-1/2}V(x,y,\sqrt t)^{-1}   e^{-cd(x,y)^2/t}.
\end{equation}
Equivalently,
\begin{equation}\label{cc:eq:heat-first-scale}
    |sT_{\ell,x}h_{s^2}(x,y)|\leq C V(x,y,s)^{-1}e^{-cd(x,y)^2/s^2}.
\end{equation}
\end{lemma}

\begin{proof}
By Lemma~\ref{cc:lem:heat-identity},
$$
    |T_{\ell,x}h_t(x,y)|   \leq \frac{\|x-y\|}{2t}|h_t(x,y)|.
$$
By Lemma~\ref{cc:lem:heat-buffer},
\begin{align*}
    |T_{\ell,x}h_t(x,y)|
    &\lesssim \frac{\|x-y\|}{t} \frac1{V(x,y,\sqrt t)} \left(1+\frac{\|x-y\|}{\sqrt t}\right)^{-2}  e^{-cd(x,y)^2/t}\\
    &= \frac1{\sqrt t}  \left(\frac{\|x-y\|}{\sqrt t}\right)  \left(1+\frac{\|x-y\|}{\sqrt t}\right)^{-2}
    \frac1{V(x,y,\sqrt t)}e^{-cd(x,y)^2/t}.
\end{align*}
Since $u(1+u)^{-2}$ is bounded on $[0,\infty)$, this is \eqref{cc:eq:heat-first}; setting $t=s^2$ gives
\eqref{cc:eq:heat-first-scale}.
\end{proof}

\begin{lemma}\label{cc:lem:heat-structure}
For every $i,j$,
\begin{align}\label{cc:eq:heat-structure}
    T_{i,x}T_{j,x}h_t(x,y)  &=\frac{(y_i-x_i)(y_j-x_j)}{4t^2}h_t(x,y)
       -\frac{\delta_{ij}}{2t}h_t(x,y)  -\frac1t\sum_{\alpha\in R}\frac{\kappa(\alpha)}2
    \frac{\alpha_i\alpha_j}{\|\alpha\|^2}h_t(\sigma_\alpha x,y).
\end{align}
\end{lemma}

\begin{proof}
Let
$$
    a_j(x,y,t)=\frac{y_j-x_j}{2t}.
$$
By Lemma~\ref{cc:lem:heat-identity},
$T_{j,x}h_t(x,y)=a_j(x,y,t)h_t(x,y),$
hence
$$
    T_{i,x}T_{j,x}h_t(x,y)=T_i(a_jh_t)(x).
$$
Expanding the Dunkl operator, we obtain
\begin{align*}
    T_i(a_jh_t)(x) &=\partial_i(a_jh_t)(x)  +
    \sum_{\alpha\in R}\frac{\kappa(\alpha)}2\alpha_i
    \frac{a_j(x)h_t(x,y)-a_j(\sigma_\alpha x)h_t(\sigma_\alpha x,y)}    {\langle\alpha,x\rangle}.
\end{align*}

We insert and subtract the term $a_j(x)h_t(\sigma_\alpha x,y)$ in the numerator.
The part involving $h_t(x,y)-h_t(\sigma_\alpha x,y)$ combines with the ordinary derivative term
$a_j(x)\partial_i h_t(x,y)$ and reconstructs $a_j(x)T_{i,x}h_t(x,y)$.
Thus
\begin{align*}
    T_i(a_jh_t)(x)  &=(\partial_i a_j)(x)h_t(x,y)+a_j(x)T_{i,x}h_t(x,y)  +
    \sum_{\alpha\in R}\frac{\kappa(\alpha)}2\alpha_i
    \frac{a_j(x)-a_j(\sigma_\alpha x)}{\langle\alpha,x\rangle}   h_t(\sigma_\alpha x,y).
\end{align*}
The first two factors are
$$
    \partial_i a_j=-\frac{\delta_{ij}}{2t},     \qquad
    T_{i,x}h_t(x,y)=\frac{y_i-x_i}{2t}h_t(x,y).
$$
For the reflected coefficient we use
$$
    \sigma_\alpha x=x-2\frac{\langle x,\alpha\rangle}{\|\alpha\|^2}\alpha.
$$
Then
\begin{align*}
    a_j(x,y,t)-a_j(\sigma_\alpha x,y,t)  &=\frac{y_j-x_j}{2t}-\frac{y_j-(\sigma_\alpha x)_j}{2t}
    =\frac{(\sigma_\alpha x)_j-x_j}{2t}  =-\frac{\langle x,\alpha\rangle}{t}\frac{\alpha_j}{\|\alpha\|^2}.
\end{align*}
Therefore
$$
    \frac{a_j(x,y,t)-a_j(\sigma_\alpha x,y,t)}{\langle\alpha,x\rangle}  =-\frac{\alpha_j}{t\|\alpha\|^2}.
$$
Substitution yields \eqref{cc:eq:heat-structure}.
The summands corresponding to roots with $\kappa(\alpha)=0$ vanish identically.
\end{proof}

\begin{remark}
Formula~\eqref{cc:eq:heat-structure} is written with the all-root normalization in \eqref{cc:eq:dunkl-T}.
If one sums only over a positive subsystem $R_+$, the coefficient $\kappa(\alpha)/2$ is replaced by $\kappa(\alpha)$, and the same
reflected term is obtained.
The sign is fixed by
$$
    (\sigma_\alpha x)_j-x_j =-2\frac{\langle x,\alpha\rangle}{\|\alpha\|^2}\alpha_j,
    \qquad\frac{a_j(x,y,t)-a_j(\sigma_\alpha x,y,t)}{\langle\alpha,x\rangle}
    =-\frac{\alpha_j}{t\|\alpha\|^2}.
$$
Although the computation is first written away from the reflecting hyperplanes, the displayed quotient is constant and hence extends
continuously across $\alpha^\perp$.
Thus \eqref{cc:eq:heat-structure} is an identity of smooth kernels after removable extension, not merely an
almost-everywhere formula off the walls.
\end{remark}

\begin{proposition}
There exist $C,c>0$ such that
\begin{equation}\label{cc:eq:heat-second}
    |A_t^{ij}(x,y)|\leq C t^{-1}V(x,y,\sqrt t)^{-1}e^{-cd(x,y)^2/t}.
\end{equation}
If $\|x-x'\|\leq\sqrt t$, then
\begin{equation}\label{cc:eq:heat-second-x}
     |A_t^{ij}(x,y)-A_t^{ij}(x',y)|  \leq C\frac{\|x-x'\|}{\sqrt t}t^{-1}V(x,y,\sqrt t)^{-1}e^{-cd(x,y)^2/t}.
\end{equation}
If $\|y-y'\|\leq\sqrt t$, then
\begin{equation}\label{cc:eq:heat-second-y}
     |A_t^{ij}(x,y)-A_t^{ij}(x,y')|  \leq C\frac{\|y-y'\|}{\sqrt t}t^{-1}V(x,y,\sqrt t)^{-1}e^{-cd(x,y)^2/t}.
\end{equation}
\end{proposition}

\begin{proof}
Let $s=\sqrt t$.  For this proof write
$$
        \mathcal E_s(u,v)=V(u,v,s)^{-1}\exp\left(-c\frac{d(u,v)^2}{s^2}\right),
$$
with the usual convention that the positive constant $c$ may decrease from line to line.
We use the formula of Lemma~\ref{cc:lem:heat-structure}.
Thus $A_t^{ij}(x,y)$ is the sum of
$$
    \frac{(y_i-x_i)(y_j-x_j)}{4t^2}h_t(x,y),   \qquad
    -\frac{\delta_{ij}}{2t}h_t(x,y),
$$
and the finitely many reflected terms
$$
    -\frac1t  \frac{\kappa(\alpha)}2   \frac{\alpha_i\alpha_j}{\|\alpha\|^2}
    h_t(\sigma_\alpha x,y),  \qquad \alpha\in R.
$$
We estimate these terms separately.  For the principal term, the buffered heat-kernel estimate yields
\begin{align*}
     \frac{|y_i-x_i|\,|y_j-x_j|}{t^2}|h_t(x,y)|   &\quad\lesssim
    \frac{\|x-y\|^2}{t^2}  \left(1+\frac{\|x-y\|}{s}\right)^{-2}\mathcal E_s(x,y)\\
    &\quad=  \frac1t  \left(\frac{\|x-y\|}{s}\right)^2
    \left(1+\frac{\|x-y\|}{s}\right)^{-2}\mathcal E_s(x,y)
    \lesssim    \frac1t\mathcal E_s(x,y).
\end{align*}
This is the only place where the Euclidean decay factor is needed.
For the scalar term one simply writes
$$
    \frac1t |h_t(x,y)|  \leq \frac Ct\frac1{V(x,y,s)}
    \left(1+\frac{\|x-y\|}{s}\right)^{-2}
    \exp\left(-c\frac{d(x,y)^2}{s^2}\right)  \leq
    \frac Ct\mathcal E_s(x,y).
$$
For a reflected term with $\kappa(\alpha)>0$, the reflection $\sigma_\alpha$ belongs to $G$, and therefore
$$
    d(\sigma_\alpha x,y)=d(x,y),  \qquad
    V(\sigma_\alpha x,y,s)=V(x,y,s).
$$
Applying the scalar heat estimate at $(\sigma_\alpha x,y)$ and discarding its Euclidean factor, we obtain
$$
    \frac1t |h_t(\sigma_\alpha x,y)|   \leq  \frac Ct\mathcal E_s(x,y).
$$
The coefficients of the reflected terms are fixed constants depending only on $R$ and $\kappa$, and the root system is finite.
Summing the preceding bounds proves the size estimate.

We next prove the regularity in the first variable.
Let $\delta=\|x-x'\|\leq s$.  The scalar term satisfies
\begin{align*}
    \frac1t |h_t(x,y)-h_t(x',y)|   &\leq  C\frac{\delta}{s}\frac1t\frac1{V(x,y,s)}
    \left(1+\frac{\|x-y\|}{s}\right)^{-2}  \exp\left(-c\frac{d(x,y)^2}{s^2}\right)  
    \leq  C\frac{\delta}{s}\frac1t\mathcal E_s(x,y).
\end{align*}
For each reflected term we use  $\|\sigma_\alpha x-\sigma_\alpha x'\|=\delta,$
together with the invariance  $d(\sigma_\alpha x,y)=d(x,y)$
and  $V(\sigma_\alpha x,y,s)=V(x,y,s)$.
Applied to $h_t(\sigma_\alpha x,y)-h_t(\sigma_\alpha x',y)$, the heat-kernel regularity estimate yields
\begin{align*}
    &\frac1t |h_t(\sigma_\alpha x,y)-h_t(\sigma_\alpha x',y)| \leq  C\frac{\delta}{s}\frac1t
    \left(1+\frac{\|\sigma_\alpha x-y\|}{s}\right)^{-2}\mathcal E_s(x,y) \leq
    C\frac{\delta}{s}\frac1t\mathcal E_s(x,y).
\end{align*}

Thus no comparison between $\|\sigma_\alpha x-y\|$ and $\|x-y\|$ is used; the Euclidean decay factor is only discarded.  It remains
only to check the principal
term.  Let
$$
    P_t(x,y)=\frac{(y_i-x_i)(y_j-x_j)}{4t^2}h_t(x,y).
$$
Then
\begin{align*}
    P_t(x,y)-P_t(x',y)  &=\frac{(y_i-x_i)(y_j-x_j)-(y_i-x'_i)(y_j-x'_j)}{4t^2}h_t(x,y) \\
    &\quad+  \frac{(y_i-x'_i)(y_j-x'_j)}{4t^2}  \big(h_t(x,y)-h_t(x',y)\big)      \\
    &=:P_1+P_2.
\end{align*}
The coefficient expands as
\begin{align*}
&(y_i-x_i)(y_j-x_j)-(y_i-x'_i)(y_j-x'_j) =(x'_i-x_i)(y_j-x_j)  +(x'_j-x_j)(y_i-x'_i).
\end{align*}
Since $\delta\leq s$, we have
$|(y_i-x_i)(y_j-x_j)-(y_i-x'_i)(y_j-x'_j)| \leq C\delta(\|x-y\|+s)$.
Using the size estimate with this extra factor,
\begin{align*}
    |P_1|  &\lesssim  \frac{\delta(\|x-y\|+s)}{t^2}
    \left(1+\frac{\|x-y\|}{s}\right)^{-2}\mathcal E_s(x,y)
    \lesssim  \frac{\delta}{s}\frac1t\mathcal E_s(x,y).
\end{align*}
For the second term, $\delta\leq s$ implies
$|y_i-x'_i|\,|y_j-x'_j|\leq C(\|x-y\|+s)^2$.
By the first-variable heat-kernel regularity,
$$
    |h_t(x,y)-h_t(x',y)| \lesssim
     \frac{\delta}{s}\left(1+\frac{\|x-y\|}{s}\right)^{-2}\mathcal E_s(x,y).
$$
Consequently,
\begin{align*}
        |P_2|  &\lesssim  \frac{(\|x-y\|+s)^2}{t^2}
    \frac{\delta}{s}\left(1+\frac{\|x-y\|}{s}\right)^{-2}\mathcal E_s(x,y)
    \lesssim  \frac{\delta}{s}\frac1t\mathcal E_s(x,y).
\end{align*}
Combining the principal, scalar, and reflected terms establishes \eqref{cc:eq:heat-second-x}.
Finally, the real multiplier $-\xi_i\xi_j e^{-t\|\xi\|^2}$ implies the Hermitian symmetry
\begin{equation}\label{cc:eq:heat-symmetry}
    A_t^{ij}(x,y)=\overline{A_t^{ij}(y,x)}.
\end{equation}
Hence, for $\|y-y'\|\leq s$,
\begin{align*}
    |A_t^{ij}(x,y)-A_t^{ij}(x,y')|  &=|A_t^{ij}(y,x)-A_t^{ij}(y',x)| 
     \lesssim  \frac{\|y-y'\|}{s}\frac1tV(y,x,s)^{-1}   e^{-cd(y,x)^2/s^2},
\end{align*}
which is \eqref{cc:eq:heat-second-y}.
\end{proof}

\subsection{Scale kernels}

Recall
$$
    \theta_s(x,y)=s(b(x)-b(y))A_{s^2}^{ij}(x,y).
$$

\begin{corollary}\label{cc:cor:theta}
For $b\in\operatorname{Lip}_d(\R^N)$,
$$
    |\theta_s(x,y)|\leq C\|b\|_{\operatorname{Lip}_d}V(x,y,s)^{-1}e^{-cd(x,y)^2/s^2}.
$$
If $\|x-x'\|\leq s$, then
\begin{equation}\label{cc:eq:theta-x}
     |\theta_s(x,y)-\theta_s(x',y)|\leq C\|b\|_{\operatorname{Lip}_d}\frac{\|x-x'\|}{s}
     V(x,y,s)^{-1}e^{-cd(x,y)^2/s^2}.
\end{equation}
And similarly, if $\|y-y'\|\leq s$, then
$$
     |\theta_s(x,y)-\theta_s(x,y')|\leq C\|b\|_{\operatorname{Lip}_d}\frac{\|y-y'\|}{s}
 V(x,y,s)^{-1}e^{-cd(x,y)^2/s^2}.
$$
\end{corollary}

\begin{proof}
Write
$$
        \mathcal E_s(x,y)=V(x,y,s)^{-1}e^{-cd(x,y)^2/s^2}.
$$
The size estimate follows from the orbit Lipschitz condition and \eqref{cc:eq:heat-second}:
\begin{align*}
    |\theta_s(x,y)|  &\leq  s\|b\|_{\operatorname{Lip}_d}d(x,y)  \frac1{s^2}\mathcal E_s(x,y)
    =\|b\|_{\operatorname{Lip}_d}\frac{d(x,y)}s  \mathcal E_s(x,y) \lesssim
    \|b\|_{\operatorname{Lip}_d}\mathcal E_s(x,y).
\end{align*}
For the first-variable regularity, assume $\|x-x'\|\leq s$ and write
\begin{align*}
    \theta_s(x,y)-\theta_s(x',y)  &=s[b(x)-b(x')]A_{s^2}^{ij}(x,y) +s[b(x')-b(y)]
    \big(A_{s^2}^{ij}(x,y)-A_{s^2}^{ij}(x',y)\big)\\
    &=:I_1+I_2.
\end{align*}
For the first term, the Euclidean Lipschitz bound following from $b\in\operatorname{Lip}_d$ reads
$$
    |b(x)-b(x')|\leq \|b\|_{\operatorname{Lip}_d}\|x-x'\|.
$$
Using \eqref{cc:eq:heat-second} with $t=s^2$,
\begin{align*}
    |I_1|  &\leq  s\|b\|_{\operatorname{Lip}_d}\|x-x'\|\frac1{s^2}\mathcal E_s(x,y)
    =  \|b\|_{\operatorname{Lip}_d}\frac{\|x-x'\|}{s}  \mathcal E_s(x,y).
\end{align*}
For the second term, the orbit triangle inequality implies
$$
    |b(x')-b(y)|\leq \|b\|_{\operatorname{Lip}_d}d(x',y)
    \leq \|b\|_{\operatorname{Lip}_d}\big(d(x,y)+\|x-x'\|\big)
    \leq \|b\|_{\operatorname{Lip}_d}\big(d(x,y)+s\big).
$$
By \eqref{cc:eq:heat-second-x}, again with $t=s^2$,
\begin{align*}
    |I_2|  &\lesssim  s\|b\|_{\operatorname{Lip}_d}(d(x,y)+s)
    \frac{\|x-x'\|}{s}\frac1{s^2}  \mathcal E_s(x,y) =
    \|b\|_{\operatorname{Lip}_d}\frac{\|x-x'\|}{s}  \left(1+\frac{d(x,y)}s\right)  \mathcal E_s(x,y).
\end{align*}
The elementary inequality $(1+u)e^{-cu^2}\leq C e^{-c'u^2}$, $u\geq0$, absorbs the factor $1+d(x,y)/s$ into the Gaussian.  Hence
$|I_1|+|I_2|$ satisfies \eqref{cc:eq:theta-x}.

If $\delta=\|y-y'\|\leq s$, then the same decomposition yields
\begin{align*}
    \theta_s(x,y)-\theta_s(x,y')   &=s[b(y')-b(y)]A_{s^2}^{ij}(x,y) +s[b(x)-b(y')]
    \big(A_{s^2}^{ij}(x,y)-A_{s^2}^{ij}(x,y')\big),\\
    |\theta_s(x,y)-\theta_s(x,y')|   &\lesssim
    \|b\|_{\operatorname{Lip}_d}\frac{\delta}{s}
    \left(1+\frac{d(x,y)}s\right)\mathcal E_s(x,y)\lesssim
    \|b\|_{\operatorname{Lip}_d}\frac{\delta}{s}
    V(x,y,s)^{-1}e^{-c'd(x,y)^2/s^2},
\end{align*}
where we used \eqref{cc:eq:heat-second},
\eqref{cc:eq:heat-second-y}, and
$d(x,y')\leq d(x,y)+\delta$.
\end{proof}

\subsection{Gaussian mass and heat-parameter integrals}

We record the Gaussian mass bound and the heat-parameter integrals used to integrate the scale estimates.

\begin{lemma}\label{cc:lem:gaussian-mass}
For every $c,s>0$,
$$
    \int_{\R^N}V(x,y,s)^{-1}e^{-cd(x,y)^2/s^2}\,d\omega(y)\leq C.
$$
The constant may depend on $c$ and the Dunkl data.
\end{lemma}

\begin{proof}
Since $V(x,y,s)\geq\omega(B(x,s))$, it is enough to use the fixed denominator $\omega(B(x,s))$.  Let
$$
    E_0=\{y:d(x,y)<s\},
$$
and, for $k\geq1$, let
$$
    E_k=\{y:2^{k-1}s\leq d(x,y)<2^ks\}.
$$
The set $\{y:d(x,y)<2^ks\}$ is contained in the union of the $M_G$ Euclidean balls of radius $2^ks$ centered at the
orbit of $x$.  The $G$-invariance of $d\omega$ and the upper volume growth in
Lemma~\ref{cc:lem:volume} therefore imply
$$
    \omega(E_k)\leq C2^{k\mathbf N}\omega(B(x,s)).
$$
Consequently,
\begin{align*}
    \int_{\R^N}\frac{e^{-cd(x,y)^2/s^2}}{V(x,y,s)}\,d\omega(y)
    &\leq C+ C\sum_{k=1}^{\infty}2^{k\mathbf N}e^{-c4^{k-1}}   \leq C.
\end{align*}
\end{proof}

\begin{lemma}\label{cc:lem:heat-integrals}
Let $c>0$ and $r=d(z,y)>0$. Then
\begin{equation}\label{cc:eq:int-global}
    r\int_0^\infty t^{-3/2}V(z,y,\sqrt t)^{-1}e^{-cr^2/t}\,dt
    \leq C V(z,y,r)^{-1}.
\end{equation}
Moreover, if $0<\delta\leq r/2$, then
\begin{equation}\label{cc:eq:int-perturb}
    \int_0^{\delta^2} t^{-3/2}V(z,y,\sqrt t)^{-1}e^{-cr^2/t}\,dt  +
    \delta\int_{\delta^2}^\infty t^{-2}V(z,y,\sqrt t)^{-1}e^{-cr^2/t}\,dt
    \leq C\frac{\delta}{r^2V(z,y,r)}.
\end{equation}
\end{lemma}

\begin{proof}
Write $V_r=V(z,y,r)$.  Lemma~\ref{cc:lem:volume}, applied at both endpoints and then to the maximum, implies
\begin{equation}\label{cc:eq:volume-scale}
    V(z,y,\sqrt t)^{-1} \leq
        \begin{cases}
    C(r/\sqrt t)^{\mathbf N}V_r^{-1},&0<t\leq r^2,\\
    C(r/\sqrt t)^NV_r^{-1},&t\geq r^2.
        \end{cases}
\end{equation}
For \eqref{cc:eq:int-global}, split at $t=r^2$ and let $u=t/r^2$.  The two pieces are bounded by
$$
    \frac C{V_r}  \left(  \int_0^1u^{-(3+\mathbf N)/2}e^{-c/u}\,du
    + \int_1^\infty u^{-(3+N)/2}\,du  \right),
$$
which is finite.

For the perturbative estimate, let $q=\delta/r\leq1/2$.  On $0<t<\delta^2$, the first line of
\eqref{cc:eq:volume-scale} and the same change of variables yield
\begin{align*}
    \int_0^{\delta^2}t^{-3/2}V(z,y,\sqrt t)^{-1}e^{-cr^2/t}\,dt
    &\leq  \frac C{rV_r}  \int_0^{q^2}u^{-(3+\mathbf N)/2}e^{-c/u}\,du
    \leq  C\frac{q}{rV_r}.
\end{align*}
The last inequality follows from exponential decay and is uniform for $0<q\leq1/2$.  For the second integral, split again at
$r^2$.  By \eqref{cc:eq:volume-scale},
\begin{align*}
    &\delta\int_{\delta^2}^{\infty}t^{-2}V(z,y,\sqrt t)^{-1}e^{-cr^2/t}\,dt\\
     \leq\,&  C\frac{\delta}{r^2V_r}  \left(  \int_{q^2}^1u^{-2-\mathbf N/2}e^{-c/u}\,du
    + \int_1^\infty u^{-2-N/2}\,du  \right)
    \leq  C\frac{\delta}{r^2V_r}.
\end{align*}
Combining the two estimates establishes \eqref{cc:eq:int-perturb}.
\end{proof}

\begin{lemma}\label{cc:lem:stability}
Let $r=d(x,y)>0$. If $\|x-x'\|\leq r/4$ and $r'=d(x',y)$, then
$r'\approx r, V(x',y,r')\approx V(x,y,r)$.
If $\|y-y'\|\leq r/4$ and $r''=d(x,y')$, then
$r''\approx r, V(x,y',r'')\approx V(x,y,r)$.
\end{lemma}

\begin{proof}
The orbit triangle inequality implies $|r-r'|\leq\|x-x'\|$, hence
$$
    \frac34r\leq r'\leq\frac54r.
$$
Moreover,
$$
    B(x',r')\subset B(x,r'+\|x-x'\|)\subset B(x,3r/2),
$$
whereas
$$
    B(x,r)\subset B(x',r+\|x-x'\|)\subset B(x',5r'/3).
$$
By doubling,
$$
    \omega(B(x',r'))\approx\omega(B(x,r)).
$$
The same-center comparison also shows $\omega(B(y,r'))\approx\omega(B(y,r))$.  Taking maxima establishes the first volume
comparison; interchanging the endpoints establishes the second.
\end{proof}

\begin{corollary}\label{cc:cor:nearby}
Let $r=d(x,y)>0$ and let $0<\delta\leq r/4$.

\smallskip
\noindent
\emph{(i)} If $\|x-x'\|\leq\delta$ and $r'=d(x',y)$, then
\begin{equation*}
\int_0^{\delta^2} t^{-3/2}V(x',y,\sqrt t)^{-1}e^{-c(r')^2/t}\,dt
\leq
C\frac{\delta}{r^2V(x,y,r)}.
\end{equation*}

\noindent
\emph{(ii)} If $\|y-y'\|\leq\delta$ and $r''=d(x,y')$, then
\begin{equation*}
    \int_0^{\delta^2} t^{-3/2}V(x,y',\sqrt t)^{-1}e^{-c(r'')^2/t}\,dt
    \leq   C\frac{\delta}{r^2V(x,y,r)}.
\end{equation*}
The constants depend only on the structural constants and on the fixed exponent in the Gaussian factor.
\end{corollary}

\begin{proof}
In part (i), the orbit triangle inequality implies $3r/4\leq r'\leq5r/4$.  In particular,
$$
    \delta\leq\frac r4\leq\frac{r'}3<\frac{r'}2.
$$
Apply the first term of \eqref{cc:eq:int-perturb} to the pair $(x',y)$, with distance $r'$, and then use
Lemma~\ref{cc:lem:stability}.  Thus
$$
    \int_0^{\delta^2}t^{-3/2}V(x',y,\sqrt t)^{-1}e^{-c(r')^2/t}\,dt
    \leq   C\frac{\delta}{(r')^2V(x',y,r')}  \leq
    C\frac{\delta}{r^2V(x,y,r)}.
$$
For part (ii), the orbit triangle inequality similarly implies $3r/4\leq r''\leq5r/4$ and $\delta<r''/2$.  Apply the first term of
\eqref{cc:eq:int-perturb} to the pair $(x,y')$ and then use
Lemma~\ref{cc:lem:stability}.
\end{proof}

\subsection{The integrated kernel}

For $d(x,y)>0$, define
\begin{equation}\label{cc:eq:kernel-def}
    K_b^{ij}(x,y)=c_\ast(b(x)-b(y))\int_0^\infty A_t^{ij}(x,y)\frac{dt}{\sqrt t}.
\end{equation}

\begin{lemma}\label{cc:lem:K-size}
If $r=d(x,y)>0$, then \eqref{cc:eq:kernel-def} is absolutely convergent and
\begin{equation*}
    |K_b^{ij}(x,y)|\leq C\|b\|_{\operatorname{Lip}_d}V(x,y,r)^{-1}.
\end{equation*}
\end{lemma}

\begin{proof}
By the orbit Lipschitz condition, we have
$|b(x)-b(y)|\leq\|b\|_{\operatorname{Lip}_d}r$.
Combining this with \eqref{cc:eq:heat-second} yields
\begin{align*}
    |K_b^{ij}(x,y)|  &\lesssim  \|b\|_{\operatorname{Lip}_d}r
    \int_0^\infty t^{-3/2}V(x,y,\sqrt t)^{-1}e^{-cr^2/t}\,dt.
\end{align*}
Lemma~\ref{cc:lem:heat-integrals} bounds the integral by $C/(rV(x,y,r))$; both assertions follow.
\end{proof}

\begin{proposition}\label{cc:prop:K-x}
If $r=d(x,y)>0$ and $\delta=\|x-x'\|\leq r/4$, then
$$
    |K_b^{ij}(x,y)-K_b^{ij}(x',y)|\leq C\|b\|_{\operatorname{Lip}_d}\frac{\delta}{r}V(x,y,r)^{-1}.
$$
\end{proposition}

\begin{proof}
If $\delta=0$, then $x=x'$ and there is nothing to prove.  Hence assume that $\delta>0$.
For readability we write $A_t=A_t^{ij}$.  Let
$r'=d(x',y)$.
By Lemma~\ref{cc:lem:stability},
$r'\approx r, V(x',y,r')\approx V(x,y,r)$.
We decompose the difference according to whether the change falls on the Lipschitz factor or on the heat-kernel factor:
\begin{align*}
    K_b^{ij}(x,y)-K_b^{ij}(x',y)  &=c_\ast[b(x)-b(x')]
      \int_0^\infty A_t(x,y)\frac{dt}{\sqrt t}\\
    &\quad+c_\ast[b(x')-b(y)]   \int_0^\infty (A_t(x,y)-A_t(x',y))\frac{dt}{\sqrt t}\\
    &=:I+II.
\end{align*}
For the first term, $|b(x)-b(x')|\leq \|b\|_{\operatorname{Lip}_d}\delta$, and \eqref{cc:eq:heat-second} yields
$$
    \int_0^\infty |A_t(x,y)|\frac{dt}{\sqrt t}  \lesssim
    \int_0^\infty t^{-3/2}V(x,y,\sqrt t)^{-1}e^{-cr^2/t}\,dt.
$$
By \eqref{cc:eq:int-global}, the last integral is bounded by $C/(rV(x,y,r))$.  Hence
$$
    |I|\lesssim    \|b\|_{\operatorname{Lip}_d}  \frac{\delta}{r}V(x,y,r)^{-1}.
$$
For the second term, the orbit Lipschitz condition gives
$$
    |b(x')-b(y)|\leq \|b\|_{\operatorname{Lip}_d}d(x',y)
    \lesssim \|b\|_{\operatorname{Lip}_d}r.
$$
It remains to estimate the integral of the difference of $A_t$.
We split at the scale $t=\delta^2$.  On the small interval $0<t<\delta^2$, no regularity estimate is available, so we use two size
estimates:
\begin{align*}
    &\int_0^{\delta^2}|A_t(x,y)-A_t(x',y)|\frac{dt}{\sqrt t}  \leq
    \int_0^{\delta^2}|A_t(x,y)|\frac{dt}{\sqrt t}  +
    \int_0^{\delta^2}|A_t(x',y)|\frac{dt}{\sqrt t}.
\end{align*}
The first integral is bounded by
$$
    C\int_0^{\delta^2}t^{-3/2}V(x,y,\sqrt t)^{-1}e^{-cr^2/t}\,dt,
$$
whereas the second is bounded by
$$
    C\int_0^{\delta^2}t^{-3/2}V(x',y,\sqrt t)^{-1}e^{-c(r')^2/t}\,dt.
$$
The first short-time integral is controlled by \eqref{cc:eq:int-perturb} applied to the unperturbed pair
$(x,y)$.
The second is controlled by Corollary~\ref{cc:cor:nearby}, applied to the perturbed pair $(x',y)$.
Hence
$$
    \int_0^{\delta^2}|A_t(x,y)-A_t(x',y)|\frac{dt}{\sqrt t}
    \lesssim  \frac{\delta}{r^2V(x,y,r)}.
$$
On the complementary interval $t\geq\delta^2$, one has $\delta\leq\sqrt t$.   The heat-scale regularity
\eqref{cc:eq:heat-second-x} yields
\begin{align*}
    &\int_{\delta^2}^{\infty}|A_t(x,y)-A_t(x',y)|\frac{dt}{\sqrt t}  \lesssim
    \delta\int_{\delta^2}^{\infty}   t^{-2}V(x,y,\sqrt t)^{-1}e^{-cr^2/t}\,dt
    \lesssim  \frac{\delta}{r^2V(x,y,r)},
\end{align*}
again by \eqref{cc:eq:int-perturb}.  Combining the two ranges,
$$
    \int_0^\infty |A_t(x,y)-A_t(x',y)|\frac{dt}{\sqrt t} \lesssim
    \frac{\delta}{r^2V(x,y,r)}.
$$
Therefore
$$
    |II|\lesssim  \|b\|_{\operatorname{Lip}_d}r\frac{\delta}{r^2V(x,y,r)}
    =  C\|b\|_{\operatorname{Lip}_d}\frac{\delta}{r}V(x,y,r)^{-1}.
$$
The bounds for $I$ and $II$ complete the estimate.
\end{proof}

\begin{proposition}\label{cc:prop:K-y}
If $r=d(x,y)>0$ and $\delta=\|y-y'\|\leq r/4$, then
\begin{equation*}
    |K_b^{ij}(x,y)-K_b^{ij}(x,y')|\leq C\|b\|_{\operatorname{Lip}_d}\frac{\delta}{r}V(x,y,r)^{-1}.
\end{equation*}
\end{proposition}

\begin{proof}
The case $\delta=0$ is immediate.  Assume $\delta>0$ and let $r''=d(x,y')$.  By
Lemma~\ref{cc:lem:stability},
$r''\approx r$ and $V(x,y',r'')\approx V(x,y,r)$.  As in the
first-variable proof,
\begin{align*}
    K_b^{ij}(x,y)-K_b^{ij}(x,y')  &=c_\ast[b(y')-b(y)]   \int_0^\infty A_t^{ij}(x,y)\frac{dt}{\sqrt t}\\
    &\quad+c_\ast[b(x)-b(y')]   \int_0^\infty   \bigl(A_t^{ij}(x,y)-A_t^{ij}(x,y')\bigr)\frac{dt}{\sqrt t}\\
    &=:I+II,
\end{align*}
and \eqref{cc:eq:int-global} yields
$$
    |I|\lesssim  \|b\|_{\operatorname{Lip}_d}\frac{\delta}{rV(x,y,r)}.
$$
Using the two size estimates on $(0,\delta^2)$, the
second-variable regularity estimate on $(\delta^2,\infty)$,
\eqref{cc:eq:int-perturb}, and
Corollary~\ref{cc:cor:nearby}, we obtain
\begin{align*}
    &\int_0^\infty  |A_t^{ij}(x,y)-A_t^{ij}(x,y')|\frac{dt}{\sqrt t}\\
     \lesssim\,&  \sum_{z\in\{y,y'\}}\int_0^{\delta^2}  t^{-3/2}V(x,z,\sqrt t)^{-1}e^{-cd(x,z)^2/t}\,dt +
    \delta\int_{\delta^2}^\infty   t^{-2}V(x,y,\sqrt t)^{-1}e^{-cr^2/t}\,dt \lesssim\frac{\delta}{r^2V(x,y,r)}.
\end{align*}
Since $|b(x)-b(y')|\lesssim\|b\|_{\operatorname{Lip}_d}r$,
\begin{equation*}
    |II|\lesssim  \|b\|_{\operatorname{Lip}_d}r  \frac{\delta}{r^2V(x,y,r)}
    \lesssim   \|b\|_{\operatorname{Lip}_d}\frac{\delta}{rV(x,y,r)}.
\end{equation*}
The estimates for $I$ and $II$ establish the assertion.
\end{proof}

\begin{theorem}
For $d(x,y)>0$,
\begin{equation}\label{cc:eq:K-size}
    |K_b^{ij}(x,y)|\leq C\|b\|_{\operatorname{Lip}_d}V(x,y,d(x,y))^{-1}.
\end{equation}
If $\|x-x'\|\leq d(x,y)/4$, then
\begin{equation}\label{cc:eq:K-x}
    |K_b^{ij}(x,y)-K_b^{ij}(x',y)|\leq C\|b\|_{\operatorname{Lip}_d}\frac{\|x-x'\|}{d(x,y)}V(x,y,d(x,y))^{-1}.
\end{equation}
If $\|y-y'\|\leq d(x,y)/4$, then
\begin{equation}\label{cc:eq:K-y}
    |K_b^{ij}(x,y)-K_b^{ij}(x,y')|\leq C\|b\|_{\operatorname{Lip}_d}\frac{\|y-y'\|}{d(x,y)}V(x,y,d(x,y))^{-1}.
\end{equation}
\end{theorem}

\begin{proof}
The size estimate is Lemma~\ref{cc:lem:K-size}.  The two regularity estimates follow from
Propositions~\ref{cc:prop:K-x}--\ref{cc:prop:K-y}, respectively, with
$r=d(x,y)$ and $\delta$ equal to the corresponding endpoint displacement.
\end{proof}

\section{Component testing}\label{cc:sec:testing}

Throughout this section, $b\in\operatorname{Lip}_d(\R^N)$.

For $s>0$, let $\Theta_s$ be the integral operator with kernel
$\theta_s(x,y)=s(b(x)-b(y))A_{s^2}^{ij}(x,y)$.  After chamber lifting, constant coordinate vectors correspond on
the full space to
$$
    \chi_\tau=\one_{\mathsf g_\tau\mathcal C},   \qquad 1\leq\tau\leq M_G.
$$
These indicators are discontinuous across the reflection hyperplanes.  Smooth scale-kernel bounds alone therefore
do not imply the required component tests.  We first prove a wall-layer estimate and then introduce smooth chamber
cutoffs at scale $s$.  The cutoff error is controlled by Gaussian averaging near the walls; for the smooth part,
the product rule moves one Dunkl derivative onto the cutoff and leaves a vertical term involving $\partial_jb$.
No zero-moment cancellation of $\theta_s$ is used.

The resulting square Carleson tests are useful single-scale information, but they do not control the unsquared integral
$\int\Theta_s\,ds/s$.  Section~\ref{cc:sec:lifting} supplies the integrated BMO and weak
boundedness estimates needed for that step.

\subsection{Wall layers and chamber cutoffs}

Let
\begin{equation*}
    \delta_{\mathcal W}(x)=\dist(x,\mathcal W),  \qquad
    m_s(x)=\min\left\{1,\frac{s}{\delta_{\mathcal W}(x)}\right\},
\end{equation*}
with $m_s=1$ on $\mathcal W$.  Fix a sufficiently small structural constant
$c_{\mathcal G}>0$ and write
\begin{equation*}
    \mathcal G_sF(x)=\int_{\R^N}
    \frac{e^{-c_{\mathcal G}d(x,y)^2/s^2}}{V(x,y,s)}F(y)\,d\omega(y).
\end{equation*}

\begin{lemma}
For every ball $B=B(x_0,r)$ and $0<\lambda\leq1$,
\begin{equation}\label{cc:eq:wall-layer}
    \omega\bigl(B\cap\{\delta_{\mathcal W}\leq\lambda r\}\bigr)  \lesssim \lambda\omega(B).
\end{equation}
Moreover, for $q=1,2$,
\begin{equation}\label{cc:eq:wall-packing}
    \int_0^r\int_B  \bigl(m_s(x)^q+|\mathcal G_sm_s(x)|^q\bigr)  \,d\omega(x)\frac{ds}{s}
    \lesssim_q\omega(B).
\end{equation}
Consequently, for $L\geq1$,
\begin{equation}\label{cc:eq:orbit-wall}
    \int_0^r\int_{\mathcal O(B(x_0,Lr))}  \bigl(m_s(x)^q+|\mathcal G_sm_s(x)|^q\bigr)
    \,d\omega(x)\frac{ds}{s}  \lesssim_q L^{\mathbf N}\omega(B).
\end{equation}
\end{lemma}

\begin{proof}
For each reflecting hyperplane, its $\lambda r$-neighborhood in $B$ admits a cover
\begin{equation*}
    \bigcup_{m=1}^J B(z_m,c\lambda r),
    \qquad  J\lesssim\lambda^{-(N-1)},  \qquad z_m\in B(x_0,2r).
\end{equation*}
By lower volume growth, doubling, and a change of center,
\begin{align*}
    \omega\bigl(B\cap\{\delta_{\mathcal W}\leq\lambda r\}\bigr)
    &\lesssim  \lambda^{-(N-1)}\lambda^N\omega(B)  \lesssim\lambda\omega(B),
\end{align*}
after summing over the finitely many walls.

Up to the null set $B\cap\mathcal W$, decompose
\begin{equation*}
    B=B_\infty\cup\bigcup_{k\geq0}B_k,  \quad
    B_\infty=B\cap\{\delta_{\mathcal W}>r\},  \quad
    B_k=B\cap\{2^{-k-1}r<\delta_{\mathcal W}\leq2^{-k}r\}.
\end{equation*}
For $q=1,2$,
$$
    \int_0^r m_s(x)^q\frac{ds}{s}  \lesssim_q1, \qquad x\in B_\infty,
 $$
 and
 $$
    \int_0^r m_s(x)^q\frac{ds}{s}  \leq  \int_0^{\delta_{\mathcal W}(x)}
    \left(\frac{s}{\delta_{\mathcal W}(x)}\right)^q\frac{ds}{s}
    +\int_{\delta_{\mathcal W}(x)}^r\frac{ds}{s}
    \lesssim_q k+1, \qquad x\in B_k.
$$
Since \eqref{cc:eq:wall-layer} implies
$\omega(B_k)\lesssim2^{-k}\omega(B)$,
\begin{equation}\label{cc:eq:wall-only}
    \int_0^r\int_Bm_s^q\,d\omega\frac{ds}{s}
    \lesssim_q  \left(1+\sum_{k\geq0}(k+1)2^{-k}\right)\omega(B)
    \lesssim_q\omega(B).
\end{equation}

For nonnegative $F$, let
\begin{equation*}
    A_\rho^{\mathcal O}F(x)  =\frac1{\omega(\mathcal O(B(x,\rho)))}  \int_{\mathcal O(B(x,\rho))}F\,d\omega.
\end{equation*}
An orbit-annular decomposition, volume growth, and Gaussian absorption give
\begin{equation}\label{cc:eq:orbit-average}
    \mathcal G_sF(x)  \lesssim\sum_{\ell\geq0}e^{-c4^\ell}  A_{2^{\ell+1}s}^{\mathcal O}F(x).
\end{equation}
Indeed, for
$E_0=\{d(x,y)<2s\}$ and
$E_\ell=\{2^\ell s\leq d(x,y)<2^{\ell+1}s\}$, $\ell\geq1$,
\begin{align*}
    \int_{E_\ell}\frac{e^{-c_{\mathcal G}d(x,y)^2/s^2}}{V(x,y,s)}F(y)\,d\omega(y)
    &\lesssim  e^{-c4^\ell}2^{\ell\mathbf N}   A_{2^{\ell+1}s}^{\mathcal O}F(x)
    \lesssim e^{-c'4^\ell}A_{2^{\ell+1}s}^{\mathcal O}F(x).
\end{align*}
Also, Fubini, doubling, and
$d(x,y)<\rho\Rightarrow  \omega(\mathcal O(B(x,\rho)))\approx \omega(\mathcal O(B(y,\rho)))$ imply
\begin{equation}\label{cc:eq:average-local}
    \int_BA_\rho^{\mathcal O}F\,d\omega  \lesssim   \int_{\mathcal O(B(x_0,r+\rho))}F\,d\omega.
\end{equation}
Indeed, the inner integral in the Fubini formula vanishes unless
$d(x_0,y)<r+\rho$ and is then bounded by
$\omega(B\cap\mathcal O(B(y,\rho)))/
\omega(\mathcal O(B(y,\rho)))\leq1$.

By \eqref{cc:eq:orbit-average}, Jensen's inequality, and weighted Cauchy--Schwarz when $q=2$,
\begin{equation*}
    |\mathcal G_sm_s|^q  \lesssim_q  \sum_{\ell\geq0}e^{-c4^\ell}  A_{2^{\ell+1}s}^{\mathcal O}(m_s^q).
\end{equation*}
Thus \eqref{cc:eq:average-local}, \eqref{cc:eq:wall-only}, and
$r+2^{\ell+1}s\lesssim2^\ell r$ for $s\leq r$ imply
\begin{align*}
    \int_0^r\int_B|\mathcal G_sm_s|^q\,d\omega\frac{ds}{s}
    &\lesssim_q  \sum_{\ell\geq0}e^{-c4^\ell}
    \int_0^r\int_{\mathcal O(B(x_0,C2^\ell r))}m_s^q\,d\omega\frac{ds}{s} 
     \lesssim_q  \sum_{\ell\geq0}e^{-c4^\ell}2^{\ell\mathbf N}\omega(B)  \lesssim_q\omega(B).
\end{align*}
Here the penultimate estimate follows by applying \eqref{cc:eq:wall-only} to the finitely many balls
$B(\mathsf g x_0,C2^\ell r)$, establishing \eqref{cc:eq:wall-packing}.  Applying that estimate to the
balls $B(\mathsf g x_0,Lr)$ gives
\begin{align*}
    \int_0^r\int_{\mathcal O(B(x_0,Lr))}  \bigl(m_s^q+|\mathcal G_sm_s|^q\bigr)\,d\omega\frac{ds}{s}
    &\leq  \sum_{\mathsf g\in G}\int_0^{Lr}\int_{B(\mathsf g x_0,Lr)}
    \bigl(m_s^q+|\mathcal G_sm_s|^q\bigr)\,d\omega\frac{ds}{s}\\
    &\lesssim_q \omega(B(x_0,Lr))   \lesssim_q L^{\mathbf N}\omega(B),
\end{align*}
which proves \eqref{cc:eq:orbit-wall}.
\end{proof}

For the chamber $\Omega_\tau=\mathsf g_\tau\mathcal C$, let $R_\tau^+$ be its positive root system.
Fix $\vartheta\in C^\infty(\R)$ such that $0\leq\vartheta\leq1$,
$\vartheta=0$ on $(-\infty,1]$, and $\vartheta=1$ on $[2,\infty)$, and let
\begin{equation*}
    \ell_{\tau,\alpha}(x)=\frac{\langle x,\alpha\rangle}{\|\alpha\|},  \qquad
    \eta_{\tau,s}(x)=\prod_{\alpha\in R_\tau^+}  \vartheta\left(\frac{\ell_{\tau,\alpha}(x)}s\right).
\end{equation*}

\begin{lemma}\label{cc:lem:cutoff}
For every $s>0$ and $\ell=1,\ldots,N$,
\begin{equation*}
    |\chi_\tau-\eta_{\tau,s}|+|sT_\ell\eta_{\tau,s}|
    \lesssim m_s  \qquad\text{$\omega$-a.e.}
\end{equation*}
uniformly in $\tau$.
\end{lemma}

\begin{proof}
The first difference vanishes unless $\delta_{\mathcal W}\lesssim s$.
Moreover,
\begin{equation*}
    |s\partial_\ell\eta_{\tau,s}|
    \lesssim\one_{\{\delta_{\mathcal W}\lesssim s\}}
    \lesssim m_s.
\end{equation*}
Since $0\leq\eta_{\tau,s}\leq1$, $\eta_{\tau,s}$ is $Cs^{-1}$-Lipschitz, and
$\|x-\sigma_\alpha x\|=2|\langle x,\alpha\rangle|/\|\alpha\|$,
\begin{align*}
    s\left|  \frac{\eta_{\tau,s}(x)-\eta_{\tau,s}(\sigma_\alpha x)}  {\langle\alpha,x\rangle}  \right|
    &\lesssim  \min\left\{\frac{s}{|\langle\alpha,x\rangle|},1\right\}
    \lesssim  \min\left\{1,\frac{s}{\dist(x,\alpha^\perp)}\right\}  \leq m_s(x).
\end{align*}
The Dunkl formula and the finiteness of $R$ complete the proof; the walls are $\omega$-null.
\end{proof}

\begin{lemma}\label{cc:lem:heat-gradient}
For every ball $B_0=B(x_0,r)$, $g\in L^\infty(d\omega)$, and $\ell=1,\ldots,N$,
\begin{equation*}
    \int_0^r\int_{B_0}|sT_\ell H_{s^2}g(x)|^2\,d\omega(x)\frac{ds}{s}
    \lesssim\|g\|_\infty^2\omega(B_0).
\end{equation*}
\end{lemma}

\begin{proof}
Let $g_1=g\one_{\mathcal O(B(x_0,4r))}$ and $g_2=g-g_1$.  By Plancherel,
\begin{align*}
    \int_0^r\int_{B_0}|sT_\ell H_{s^2}g_1|^2\,d\omega\frac{ds}{s}
    &\leq\int_0^\infty\|sT_\ell H_{s^2}g_1\|_2^2\frac{ds}{s}
    =\int_{\R^N}\frac{\xi_\ell^2}{4\|\xi\|^2}
    |\mathcal F_\kappa g_1(\xi)|^2\,d\omega(\xi)\\
    &\leq\frac14\|g_1\|_2^2   \lesssim\|g\|_\infty^2\omega(B_0),
\end{align*}
where the value at $\xi=0$ is zero.

If $x\in B_0$, $s\leq r$, and $y\notin\mathcal O(B(x_0,4r))$, then
$d(x,y)\geq3r$.  By Lemma~\ref{cc:lem:heat-first},
\begin{align*}
    |sT_\ell H_{s^2}g_2(x)|  &\leq  \|g\|_\infty
    \int_{d(x,y)\geq3r}   \frac1{V(x,y,s)}e^{-cd(x,y)^2/s^2}\,d\omega(y).
\end{align*}
Orbit-annular summation and Gaussian absorption now imply
$$
    |sT_\ell H_{s^2}g_2(x)|    \lesssim    \|g\|_\infty e^{-cr^2/s^2}.
$$
Thus
$$
    \int_0^r |sT_\ell H_{s^2}g_2(x)|^2\frac{ds}{s}  \lesssim
    \|g\|_\infty^2\int_0^r e^{-cr^2/s^2}\frac{ds}{s}   \lesssim \|g\|_\infty^2.
$$
Integration over $B_0$ finishes the estimate.
\end{proof}

\subsection{Testing the chamber indicators}

\begin{lemma}\label{cc:lem:theta-adjoint}
For every $s>0$, $\Theta_{s,b}$ is bounded on $L^2(d\omega)$ and
\begin{equation*}
    \Theta_{s,b}^*=-\Theta_{s,\overline b}.
\end{equation*}
The same identity holds for chamber indicators through their absolutely convergent kernel integrals.
\end{lemma}

\begin{proof}
By Corollary~\ref{cc:cor:theta} and Lemma~\ref{cc:lem:gaussian-mass}, in both variables,
\begin{equation*}
    \sup_x\int|\theta_s(x,y)|\,d\omega(y)  +\sup_y\int|\theta_s(x,y)|\,d\omega(x)
    \lesssim\|b\|_{\operatorname{Lip}_d},
\end{equation*}
so Schur's test implies the $L^2$ bound.  By \eqref{cc:eq:heat-symmetry},
\begin{align*}
    \overline{s(b(y)-b(x))A_{s^2}^{ij}(y,x)}
    &=s\big(\overline b(y)-\overline b(x)\big)A_{s^2}^{ij}(x,y)
    =-\theta_{s,\overline b}(x,y).
\end{align*}
The Schur estimates imply the operator identity; Gaussian mass justifies the corresponding kernel identity for $\chi_\tau$.
\end{proof}

\begin{lemma}\label{cc:lem:theta-identity}
Assume that $b\in C^\infty(\R^N)$ is $G$-invariant and that $\varphi\in C_c^\infty(\R^N)$.
Then, for every $s>0$,
\begin{equation}\label{cc:eq:theta-identity}
    \Theta_{s,b}\varphi  =   s[M_b,T_iH_{s^2}](T_j\varphi)  -   sT_iH_{s^2}\big((\partial_jb)\varphi\big).
\end{equation}
Here the identity is understood pointwise, or equivalently after pairing with compactly supported bounded test functions.
\end{lemma}

\begin{proof}
\begin{align*}
    \Theta_{s,b}\varphi  &=s\bigl(M_bT_iH_{s^2}T_j\varphi   -T_iH_{s^2}T_j(b\varphi)\bigr)\\
    &=s\bigl(M_bT_iH_{s^2}(T_j\varphi)   -T_iH_{s^2}(bT_j\varphi)\bigr)   -sT_iH_{s^2}\bigl((\partial_jb)\varphi\bigr),
\end{align*}
where we used $T_j(b\varphi)=(\partial_jb)\varphi+bT_j\varphi$ and commutation with $H_{s^2}$.
\end{proof}

\begin{lemma}\label{cc:lem:component-split}
Assume that $b\in C^\infty(\R^N)\cap\operatorname{Lip}_d(\R^N)$ is
$G$-invariant and
$\max_k\|\partial_kb\|_\infty\leq\|b\|_{\operatorname{Lip}_d}$.
Let
\begin{equation*}
    \beta_{\tau,s}=\chi_\tau-\eta_{\tau,s},  \qquad
    g_\tau=(\partial_jb)\chi_\tau.
\end{equation*}
Then
\begin{equation}\label{cc:eq:component-split}
    \Theta_s\chi_\tau=-sT_iH_{s^2}g_\tau+E_{\tau,s},
\end{equation}
where
\begin{equation}\label{cc:eq:component-error}
    E_{\tau,s}  =\Theta_s\beta_{\tau,s}  +s[M_b,T_iH_{s^2}](T_j\eta_{\tau,s})
    +sT_iH_{s^2}\bigl((\partial_jb)\beta_{\tau,s}\bigr),
\end{equation}
and
\begin{equation}\label{cc:eq:error-size}
    |E_{\tau,s}(x)|  \lesssim  \|b\|_{\operatorname{Lip}_d}\mathcal G_sm_s(x).
\end{equation}
\begin{equation}\label{cc:eq:error-reg}
    |E_{\tau,s}(x)-E_{\tau,s}(x')|  \lesssim
    \|b\|_{\operatorname{Lip}_d}\frac{\|x-x'\|}{s},  \qquad \|x-x'\|\leq s.
\end{equation}
\end{lemma}

\begin{proof}
We first justify \eqref{cc:eq:theta-identity} for the noncompact cutoff.
Let $\psi_L(x)=\psi(x/L)$, where $\psi\in C_c^\infty$ is radial,
$0\leq\psi\leq1$, equals one on $B(0,1)$, and is supported in $B(0,2)$; let
$\eta_{\tau,s}^L=\eta_{\tau,s}\psi_L$.  On a fixed box
$B(x_0,r_0)\times(0,r_0)$ and for $L\geq4(1+\|x_0\|+r_0)$,
\begin{equation*}
    d(x,y)\geq\|y\|-\|x\|   \geq L-(\|x_0\|+r_0)\geq\frac{3L}{4}=:c_0L
\end{equation*}
whenever $x\in B(x_0,r_0)$ and $y\notin B(0,L)$.  Hence, for $M=0,1$,
\begin{equation}\label{cc:eq:cutoff-tail}
    \int_{\{d(x,y)\geq c_0L\}}   \left(1+\frac{d(x,y)}s\right)^M
    \frac{e^{-cd(x,y)^2/s^2}}{V(x,y,s)}\,d\omega(y)   \lesssim_M e^{-c_ML^2/s^2}.
\end{equation}
This follows from the orbit-annular argument in \eqref{cc:eq:orbit-average}, with lower cutoff $c_0L$.
Since $\psi_L$ is $G$-invariant,
$$
    T_j(\eta_{\tau,s}\psi_L)  =\psi_LT_j\eta_{\tau,s}+\eta_{\tau,s}\partial_j\psi_L,
    \qquad  |\partial_j\psi_L|\lesssim L^{-1}   \one_{\{L\leq\|y\|\leq2L\}}.
$$
Corollary~\ref{cc:cor:theta},
Lemma~\ref{cc:lem:heat-first}, and
\eqref{cc:eq:cutoff-tail} imply, uniformly on the fixed box,
\begin{align*}
    &|\Theta_s(\eta_{\tau,s}(1-\psi_L))(x)|  +\left|s[M_b,T_iH_{s^2}]
    \bigl((1-\psi_L)T_j\eta_{\tau,s}\bigr)(x)\right|+\left|sT_iH_{s^2}
    \bigl((\partial_jb)\eta_{\tau,s}(1-\psi_L)\bigr)(x)\right| \\
     \lesssim\,&   \|b\|_{\operatorname{Lip}_d}e^{-cL^2/s^2}.
\end{align*}
For $A_L=\{L\leq\|y\|\leq2L\}$, the remaining term satisfies
\begin{align*}
    \left|s[M_b,T_iH_{s^2}](\eta_{\tau,s}\partial_j\psi_L)(x)\right|  &\lesssim
    \|b\|_{\operatorname{Lip}_d}  \int_{A_L}\frac{d(x,y)}{L}   \frac{e^{-cd(x,y)^2/s^2}}{V(x,y,s)}\,d\omega(y)\\
    &\lesssim  \|b\|_{\operatorname{Lip}_d}  \int_{\{d(x,y)\geq c_0L\}}  \left(1+\frac{d(x,y)}s\right)
    \frac{e^{-cd(x,y)^2/s^2}}{V(x,y,s)}\,d\omega(y)   \lesssim
    \|b\|_{\operatorname{Lip}_d}e^{-cL^2/s^2},
\end{align*}
because $s\leq r_0\leq L$.  The preceding tail bounds permit $L\to\infty$ in
Lemma~\ref{cc:lem:theta-identity}, yielding
\begin{equation*}
\Theta_s\eta_{\tau,s}
=s[M_b,T_iH_{s^2}](T_j\eta_{\tau,s})
-sT_iH_{s^2}\bigl((\partial_jb)\eta_{\tau,s}\bigr).
\end{equation*}
Since $\eta_{\tau,s}=\chi_\tau-\beta_{\tau,s}$, this is
\eqref{cc:eq:component-split}--\eqref{cc:eq:component-error}.

By Lemma~\ref{cc:lem:cutoff}, the first-derivative heat estimate, and Gaussian absorption, we have
$$
    |\Theta_s\beta_{\tau,s}(x)|  \lesssim   \|b\|_{\operatorname{Lip}_d}\mathcal G_sm_s(x),
$$
$$
    \left|s[M_b,T_iH_{s^2}](T_j\eta_{\tau,s})(x)\right|   \lesssim
    \|b\|_{\operatorname{Lip}_d}  \int\frac{d(x,y)}s
    \frac{e^{-cd(x,y)^2/s^2}}{V(x,y,s)}m_s(y)\,d\omega(y)   \lesssim
    \|b\|_{\operatorname{Lip}_d}\mathcal G_sm_s(x),
$$
and
$$
    \left|sT_iH_{s^2}\bigl((\partial_jb)\beta_{\tau,s}\bigr)(x)\right|
     \lesssim   \|b\|_{\operatorname{Lip}_d}\mathcal G_sm_s(x),
$$
which proves \eqref{cc:eq:error-size}.

Finally, by \eqref{cc:eq:theta-x} and Gaussian mass,
\begin{equation*}
    |\Theta_s\chi_\tau(x)-\Theta_s\chi_\tau(x')|   \lesssim
    \|b\|_{\operatorname{Lip}_d}\frac{\|x-x'\|}{s},   \qquad \|x-x'\|\leq s.
\end{equation*}
Writing $\delta=\|x-x'\|$ and $\rho=\|x-y\|$, the first-derivative identity and the buffered heat estimates give
\begin{align*}
     s|T_{i,x}h_{s^2}(x,y)-T_{i,x}h_{s^2}(x',y)|    \lesssim\,&
    \frac{\delta}{s}|h_{s^2}(x,y)|  +\left(1+\frac{\rho}{s}\right)   |h_{s^2}(x,y)-h_{s^2}(x',y)|\\
     \lesssim\,&   \frac{\|x-x'\|}{s}   \frac{e^{-cd(x,y)^2/s^2}}{V(x,y,s)},   \qquad \|x-x'\|\leq s.
\end{align*}
By \eqref{cc:eq:component-split},
\begin{align*}
    E_{\tau,s}(x)-E_{\tau,s}(x')  &=\Theta_s\chi_\tau(x)-\Theta_s\chi_\tau(x') 
     +sT_iH_{s^2}g_\tau(x)-sT_iH_{s^2}g_\tau(x')
\end{align*}
The two preceding bounds, with the kernel difference integrated against
$|g_\tau|\leq\|b\|_{\operatorname{Lip}_d}$, yield \eqref{cc:eq:error-reg}.
\end{proof}

\begin{proposition}\label{cc:prop:component-test}
For every chamber indicator $\chi_\tau$ and every Euclidean ball $B_0=B(x_0,r)$,
\begin{equation*}
    \int_0^r\int_{B_0}|\Theta_s\chi_\tau(x)|^2\,d\omega(x)\frac{ds}{s}
    \leq C\|b\|_{\operatorname{Lip}_d}^2\omega(B_0).
\end{equation*}
The same estimate holds with $\Theta_s$ replaced by $\Theta_s^*$.
\end{proposition}

\begin{proof}
Write $L_b=\|b\|_{\operatorname{Lip}_d}$.  If $b$ is smooth, Lemma~\ref{cc:lem:component-split},
Lemma~\ref{cc:lem:heat-gradient}, and the $q=2$ case of
\eqref{cc:eq:wall-packing} imply
\begin{align*}
    \int_0^r\int_{B_0}|\Theta_s\chi_\tau|^2\,d\omega\frac{ds}{s}  &\lesssim
    \int_0^r\int_{B_0}|sT_iH_{s^2}g_\tau|^2\,d\omega\frac{ds}{s}
    +L_b^2\int_0^r\int_{B_0}|\mathcal G_sm_s|^2\,d\omega\frac{ds}{s}\\  &\lesssim L_b^2\omega(B_0).
\end{align*}

For general $b$, let $b_\varepsilon$ be given by
Lemma~\ref{cc:lem:smooth-symbol}; denote the corresponding operator and kernel by
$\Theta_{s,\varepsilon}$ and $\theta_{s,b_\varepsilon}$, and write $\theta_{s,b}$ for the kernel of $\Theta_s$.
The smooth estimate is uniform in $\varepsilon$:
\begin{equation}\label{cc:eq:smooth-test}
    \int_0^r\int_{B_0}|\Theta_{s,\varepsilon}\chi_\tau|^2  \,d\omega\frac{ds}{s}
    \lesssim L_b^2\omega(B_0).
\end{equation}
The global uniform convergence of $b_\varepsilon$,
\eqref{cc:eq:heat-second}, and
Lemma~\ref{cc:lem:gaussian-mass} imply
\begin{align*}
    &\sup_x\int_{\R^N}|\theta_{s,b_\varepsilon}(x,y)-\theta_{s,b}(x,y)|\,d\omega(y)
    +\sup_y\int_{\R^N}|\theta_{s,b_\varepsilon}(x,y)-\theta_{s,b}(x,y)|\,d\omega(x)\\
    &\hspace{45mm}\lesssim s^{-1}\|b_\varepsilon-b\|_\infty.
\end{align*}
In particular, for every fixed $s>0$,
\begin{align*}
    |\Theta_{s,\varepsilon}\chi_\tau(x)-\Theta_s\chi_\tau(x)|  &\leq
    2s\|b_\varepsilon-b\|_\infty   \int_{\R^N}|A_{s^2}^{ij}(x,y)|\,d\omega(y)\\
    &\lesssim s^{-1}\|b_\varepsilon-b\|_\infty    \longrightarrow0.
\end{align*}
Fatou's lemma and \eqref{cc:eq:smooth-test} extend the assertion to $b$.

Finally, Lemma~\ref{cc:lem:theta-adjoint} and the estimate for $\overline b$ imply
\begin{align*}
    \int_0^r\int_{B_0}|\Theta_{s,b}^*\chi_\tau(x)|^2\,d\omega(x)\frac{ds}{s}
    &= \int_0^r\int_{B_0}|\Theta_{s,\overline b}\chi_\tau(x)|^2\,d\omega(x)\frac{ds}{s} 
     \lesssim\|\overline b\|_{\operatorname{Lip}_d}^2\omega(B_0)
    \lesssim L_b^2\omega(B_0).
\end{align*}
\end{proof}

\section{Chamber lifting and the canonical heat realization}\label{cc:sec:lifting}

This section proves the uniform $L^2$ bound and joint heat limit in
Theorem~\ref{cc:thm:heat-limit}.  We first realize each truncation as a finite matrix on a chamber;
scalar $T1$ estimates for the integrated entries then provide bounds uniform in both heat endpoints.

We use the enumeration
$$
    G=\{\mathsf g_1,\ldots,\mathsf g_{M_G}\},   \qquad M_G=|G|,   \qquad \mathsf g_1=\mathsf e,
$$
and the open chamber $\mathcal C$ fixed in Section~\ref{cc:sec:prelim}.  By
Lemma~\ref{cc:lem:partition}, the sets $\mathsf g_\rho\mathcal C$ form a disjoint partition of $\R^N$ modulo null sets.  Define
$$
    Uf(x)=\bigl(f(\mathsf g_1x),\ldots,f(\mathsf g_{M_G}x)\bigr),   \qquad x\in\mathcal C.
$$
Conversely, if $F=(F_1,\ldots,F_{M_G})$, define
$$
    (U^{-1}F)(\mathsf g_\rho x)=F_\rho(x),   \qquad x\in\mathcal C.
$$
The wall ambiguity is null.  By reflection invariance of $d\omega$, for $1\leq p<\infty$,
\begin{align*}
    \|Uf\|_{L^p(\mathcal C,d\omega;\ell_{M_G}^p)}^p  
    &=\sum_{\rho=1}^{M_G}\int_{\mathcal C}|f(\mathsf g_\rho x)|^p\,d\omega(x) 
     =\sum_{\rho=1}^{M_G}\int_{\mathsf g_\rho\mathcal C}|f(z)|^p\,d\omega(z)
    =\|f\|_{L^p(\R^N,d\omega)}^p.
\end{align*}
The corresponding essential-supremum identity holds for $p=\infty$.  Hence $U$ is an isometric isomorphism for
every $1\leq p\leq\infty$ and is unitary when $p=2$.

The lift is a finite matrix representation, not a quotient.  The reflected values of a general function remain
independent coordinates, while Lemma~\ref{cc:lem:closest} ensures
$$
    d(\mathsf g_\rho x,\mathsf g_\tau y)=\|x-y\|,   \qquad x,y\in\overline{\mathcal C}.
$$
Thus every possible orbit singularity lies on the common chamber diagonal, and a full-space scalar kernel $K$
becomes the matrix
$$
    K^{\rho\tau}(x,y)=K(\mathsf g_\rho x,\mathsf g_\tau y).
$$
The scalar entries are therefore studied on the space of homogeneous type
$(\mathcal C,\|\cdot\|,d\omega_{\mathcal C})$; the finite coordinate dimension contributes only a structural
constant.

Define
$$
    \boldsymbol{\Theta}_s:=U\Theta_sU^{-1}.
$$
Let $F=(F_1,\ldots,F_{M_G})$ be a compactly supported bounded vector-valued function on $\mathcal C$.  For $x\in\mathcal C$,
\begin{align*}
    (\boldsymbol{\Theta}_sF)_\rho(x)  &=(\Theta_sU^{-1}F)(\mathsf g_\rho x)
    =  \int_{\R^N}\theta_s(\mathsf g_\rho x,z)(U^{-1}F)(z)\,d\omega(z).
\end{align*}
Since the chamber walls are $\omega$-null, $\R^N$ is the disjoint union, modulo null sets, of the reflected chambers
$\mathsf g_\tau\mathcal C$.
Moreover, each map $\mathsf g_\tau:\mathcal C\to\mathsf g_\tau\mathcal C$ preserves $d\omega$, and
$$
    (U^{-1}F)(\mathsf g_\tau y)=F_\tau(y),   \qquad y\in\mathcal C.
$$
Therefore
$$
    (\boldsymbol{\Theta}_sF)_\rho(x)= \sum_{\tau=1}^{M_G}
    \int_\mathcal C   \theta_s(\mathsf g_\rho x,\mathsf g_\tau y)F_\tau(y)\,d\omega(y).
$$
Accordingly the lifted scale kernel is
\begin{equation*}
    \theta_s^{\rho\tau}(x,y)    :=  \theta_s(\mathsf g_\rho x,\mathsf g_\tau y),
    \qquad x,y\in\mathcal C.
\end{equation*}

\begin{lemma}
For every $1\leq \rho,\tau\leq M_G$ and every $x,y\in\mathcal C$,
\begin{equation}\label{cc:eq:lift-scale}
    |\theta_s^{\rho\tau}(x,y)|   \leq   C\|b\|_{\operatorname{Lip}_d}
    \frac1{V_{\mathcal C}(x,y,s)}   \exp\left(-c\frac{\|x-y\|^2}{s^2}\right).
\end{equation}
Moreover, if $x,x',y\in\mathcal C$ and $\|x-x'\|\leq s$, then
\begin{align}\label{cc:eq:lift-scale-x}
    &|\theta_s^{\rho\tau}(x,y)  -\theta_s^{\rho\tau}(x',y)|
    \leq  C\|b\|_{\operatorname{Lip}_d}   \frac{\|x-x'\|}{s}
    \frac1{V_{\mathcal C}(x,y,s)}  \exp\left(-c\frac{\|x-y\|^2}{s^2}\right),
\end{align}
and, if $x,y,y'\in\mathcal C$ and $\|y-y'\|\leq s$, then
\begin{align}\label{cc:eq:lift-scale-y}
    &|\theta_s^{\rho\tau}(x,y)    -\theta_s^{\rho\tau}(x,y')|
    \leq   C\|b\|_{\operatorname{Lip}_d}  \frac{\|y-y'\|}{s}
    \frac1{V_{\mathcal C}(x,y,s)}  \exp\left(-c\frac{\|x-y\|^2}{s^2}\right).
\end{align}
The constants are independent of $\rho$ and $\tau$.
\end{lemma}

\begin{proof}
Let $z=\mathsf g_\rho x$ and $w=\mathsf g_\tau y$.  By the chamber identity and volume comparability,
$$
    d(z,w)=\|x-y\|,  \qquad  V(z,w,s)\approx V_{\mathcal C}(x,y,s),
$$
uniformly in $\rho$ and $\tau$.  The size estimate
\eqref{cc:eq:lift-scale} follows at once from Corollary~\ref{cc:cor:theta}.  If
$z'=\mathsf g_\rho x'$ and $w'=\mathsf g_\tau y'$, then
$$
    \|z-z'\|=\|x-x'\|,  \qquad   \|w-w'\|=\|y-y'\|.
$$
Applying the two regularity estimates in the same corollary, followed by the preceding distance and volume
comparisons, establishes \eqref{cc:eq:lift-scale-x} and \eqref{cc:eq:lift-scale-y}.
\end{proof}

Since $U\chi_\tau=e_\tau$, Proposition~\ref{cc:prop:component-test}, reflection invariance, and
Lemma~\ref{cc:lem:chamber-volume} transfer the square Carleson tests to
$\boldsymbol\Theta_se_\tau$ and $\boldsymbol\Theta_s^*e_\rho$ on $\mathcal C$, with norm
$O(\|b\|_{\operatorname{Lip}_d}^2)$.

\subsection{Why the single-scale tests do not close the scale integral}

The square Carleson estimates above are not, by themselves, a continuous-family $T1$ theorem: they contain no
cross-scale control.  The following Euclidean example isolates the obstruction and explains why we instead treat
each integrated truncation as a single operator.

\begin{proposition}\label{cc:prop:single-scale}
Scale-kernel size and H\"older estimates, together with the Carleson properties of
$$
    |\Theta_s1(x)|^2\,dx\frac{ds}{s}  \qquad\text{and}\qquad  |\Theta_s^*1(x)|^2\,dx\frac{ds}{s},
$$
do not imply uniform $L^2$ boundedness of $\int_a^B\Theta_s\,ds/s$, even on Euclidean space.
\end{proposition}

\begin{proof}
Let $P_s$ be the Gaussian approximate identity on $\R^N$, normalized by $P_s1=1$, and let
$$
    a(s)=
    \begin{cases}
        0,&0<s<e,\\
        (\log s)^{-1},&s\geq e,
        \end{cases}
    \qquad \text{and}\qquad  \Theta_s=a(s)P_s.
$$
The kernels of $\Theta_s$ satisfy Gaussian size estimates and first-order H\"older estimates uniformly in $s$, since
$|a(s)|\leq1$.  Moreover,
$$
    \Theta_s1=\Theta_s^*1=a(s),  \qquad  \int_0^r|a(s)|^2\frac{ds}{s}\leq
    \int_e^\infty\frac{ds}{s(\log s)^2}=1.
$$
Thus both displayed measures are Carleson.  On the other hand, the Fourier multiplier of
$\int_e^L\Theta_s\,ds/s$ is
$$
    m_L(\xi)=\int_e^L\frac{e^{-s^2\|\xi\|^2}}{\log s}\frac{ds}{s}.
$$
For $\|\xi\|\leq L^{-1}$, the exponential is bounded below by an absolute positive constant, and hence
$$
    m_L(\xi)\geq c\int_e^L\frac{ds}{s\log s}=c\log\log L.
$$
Therefore the $L^2$ operator norms of these truncations tend to infinity.
\end{proof}

\begin{remark}
For the commutator family, the component tests above provide exactly the square estimates ruled insufficient by the
counterexample.  We therefore prove uniform BMO tests, weak boundedness, and standard-kernel bounds for every
integrated chamber entry, and then apply the ordinary scalar $T1$ theorem.
\end{remark}

\subsection{Closure of the lifted scale integral}

We now combine the special commutator structure with the decomposition underlying the component Carleson estimates.
With reference measure $d\omega_{\mathcal C}$,  let
$$
    S_{a,B}^{\rho\tau}f  :=2c_\ast\int_a^B\boldsymbol\Theta_s^{\rho\tau}f\,\frac{ds}{s},
    \qquad 0<a<B<\infty.
$$
Thus the matrix $(S_{a,B}^{\rho\tau})_{\rho,\tau=1}^{M_G}$ is exactly
$UC_{a,B}^{ij,b}U^{-1}$.  The factor $2$ comes from $t=s^2$, as recorded in
Remark~\ref{cc:rem:normalization}.  The aim of this subsection is to prove bounds independent of
$a$ and $B$.

The $q=1$ estimates in
\eqref{cc:eq:wall-packing}--\eqref{cc:eq:orbit-wall}
provide the wall packing used below.

The calculation in the proof of Lemma~\ref{cc:lem:component-split} shows that, whenever
$\|x-x'\|\leq s$,
\begin{equation}\label{cc:eq:heat-first-x}
     |sT_{i,x}h_{s^2}(x,y)-sT_{i,x}h_{s^2}(x',y)|  \lesssim\frac{\|x-x'\|}{s}\frac{1}{V(x,y,s)}
    \exp\left(-c\frac{d(x,y)^2}{s^2}\right).
\end{equation}
For $0<a<B<\infty$, define the heat-truncated first Riesz operator
$$
    \mathscr R_{i;a,B}g:=\int_a^B T_iH_{s^2}g\,ds.
$$

\begin{lemma}\label{cc:lem:heat-Riesz}
The operators $\mathscr R_{i;a,B}$ are uniformly bounded on $L^2(d\omega)$ and uniformly map
$L^\infty(d\omega)$ to $\operatorname{BMO}(d\omega)$:
$$
    \sup_{0<a<B<\infty}\|\mathscr R_{i;a,B}\|_{L^2\to L^2}\leq C \qquad \text{and}\qquad
    \sup_{0<a<B<\infty} \|\mathscr R_{i;a,B}g\|_{\operatorname{BMO}(d\omega)}
    \leq C\|g\|_\infty.
$$
\end{lemma}

\begin{proof}
The Dunkl multiplier of $\mathscr R_{i;a,B}$ is
$$
    i\xi_i\int_a^B e^{-s^2\|\xi\|^2}\,ds,
$$
whose absolute value is bounded uniformly in $a$ and $B$; the $L^2$ estimate follows.

The kernel is
$$
     r_{i;a,B}(x,y)=\int_a^B T_{i,x}h_{s^2}(x,y)\,ds.
$$
We prove the far-kernel regularity needed for the BMO estimate directly; no pointwise size estimate for the
integrated Riesz kernel is used.  Given $Q=B(x_0,r)$, write
$$
    g=g_1+g_2,   \qquad   g_1=g\one_{\mathcal O(B(x_0,4r))}.
$$
The orbit-ball volume comparison and the $L^2$ estimate imply
\begin{align*}
    \frac1{\omega(Q)}\int_Q|\mathscr R_{i;a,B}g_1|\,d\omega
    &\leq \omega(Q)^{-1/2}\|\mathscr R_{i;a,B}g_1\|_2
    \lesssim\|g\|_\infty.
\end{align*}
For $x\in Q$ and $y\notin\mathcal O(B(x_0,4r))$, write
$$
    h=\|x-x_0\|,    \qquad r_y=d(x_0,y).
$$
Then $h\leq r\leq r_y/4$ and $d(x,y)\approx r_y$.  Split the scale integral into
$$
    (0,h),\qquad [h,r_y],\qquad (r_y,\infty),
$$
and intersect these ranges with $(a,B)$.  If $h=0$, there is nothing to prove.  Hence assume that $h>0$.  Write
$V_y=V(x_0,y,r_y)$.  On $0<s<h$, the two estimates
\eqref{cc:eq:heat-first-scale} require a comparison of the two volume normalizations.  Here
$$
    d(x,y)\geq r_y-h\geq\frac34r_y.
$$
Lemma~\ref{cc:lem:stability}, followed by doubling at the comparable radii $d(x,y)$ and $r_y$, implies
$$
    V(x,y,r_y)\approx V(x_0,y,r_y)=V_y.
$$
The upper volume growth from $s$ to $r_y$ therefore gives
$$
    V(x,y,s)^{-1}+V(x_0,y,s)^{-1}  \lesssim
    \left(\frac{r_y}{s}\right)^{\mathbf N}V_y^{-1}.
$$
Using \eqref{cc:eq:heat-first-scale} for the two kernels, we obtain
\begin{align*}
    &\int_0^h  \bigl|T_{i,x}h_{s^2}(x,y)-T_{i,x}h_{s^2}(x_0,y)\bigr|\,ds \lesssim
    \frac1{V_y}\int_0^h  \frac1s\left(\frac{r_y}{s}\right)^{\mathbf N}
    e^{-cr_y^2/s^2}\,ds  \lesssim  \frac h{r_yV_y}.
\end{align*}
The last estimate follows from $u=r_y/s$ and
$\int_{r_y/h}^\infty u^{\mathbf N-1}e^{-cu^2}\,du\lesssim h/r_y$, since $r_y/h\geq4$.  On
$h\leq s\leq r_y$, \eqref{cc:eq:heat-first-x} and upper volume growth give
\begin{align*}
    &\int_h^{r_y}  \bigl|T_{i,x}h_{s^2}(x,y)-T_{i,x}h_{s^2}(x_0,y)\bigr|\,ds \lesssim
    \frac h{V_y}\int_h^{r_y}  \frac1{s^2}\left(\frac{r_y}{s}\right)^{\mathbf N}
    e^{-cr_y^2/s^2}\,ds  \lesssim   \frac h{r_yV_y}.
\end{align*}
Finally, on $s>r_y$, the same regularity estimate and the lower volume growth in
Lemma~\ref{cc:lem:volume} imply
\begin{align*}
    &\int_{r_y}^{\infty}  \bigl|T_{i,x}h_{s^2}(x,y)-T_{i,x}h_{s^2}(x_0,y)\bigr|\,ds \lesssim
    \frac h{V_y}\int_{r_y}^{\infty}   \frac1{s^2}\left(\frac{r_y}{s}\right)^N\,ds
    \lesssim  \frac h{r_yV_y}.
\end{align*}
All three estimates are absolute, so their constants remain unchanged after intersection with $(a,B)$.  Consequently,
$$
    |r_{i;a,B}(x,y)-r_{i;a,B}(x_0,y)|   \lesssim    \frac{r}{d(x_0,y)}\frac1{V(x_0,y,d(x_0,y))}.
$$
Orbit-annular summation therefore shows that
$$
    |\mathscr R_{i;a,B}g_2(x)-\mathscr R_{i;a,B}g_2(x_0)|   \lesssim\|g\|_\infty.
$$
Using $\mathscr R_{i;a,B}g_2(x_0)$ as the comparison constant proves the BMO estimate.
\end{proof}

The next elementary lemma converts the pointwise wall domination from the component test into BMO control of the
scale integral.

\begin{lemma}\label{cc:lem:wall-BMO}
Suppose that a measurable family $\{E_s\}_{s>0}$ satisfies
$$
    |E_s(x)|\leq A\mathcal G_sm_s(x)
$$
and, whenever $\|x-x'\|\leq s$,
$$
    |E_s(x)-E_s(x')|\leq A\frac{\|x-x'\|}{s}.
$$
Then, uniformly in $0<a<B<\infty$,
$$
    \bigg\|\int_a^B E_s\,\frac{ds}{s}\bigg\|_{\operatorname{BMO}(d\omega)}\leq CA.
$$
\end{lemma}

\begin{proof}
Fix $Q=B(x_0,r)$ and  let
$$
    c_Q=
    \begin{cases}
        \displaystyle\int_{\max\{a,r\}}^B E_s(x_0)\,\frac{ds}{s},&B>\max\{a,r\},\\
        0,&B\leq\max\{a,r\}.
\end{cases}
$$
For the scales below $r$, the $q=1$ case of \eqref{cc:eq:wall-packing} gives
\begin{align*}
    \frac1{\omega(Q)}\int_Q  \int_a^B\one_{\{s<r\}}|E_s(x)|\,\frac{ds}{s}\,d\omega(x)
    &\leq  \frac{CA}{\omega(Q)}\int_0^r\int_Q\mathcal G_sm_s(x)\,d\omega(x)\frac{ds}{s}  \leq CA.
\end{align*}
For $s\geq r$ and $x\in Q$, the scale regularity gives
$$
    |E_s(x)-E_s(x_0)|\leq CA\frac{r}{s}.
$$
Consequently,
$$
    \int_r^\infty |E_s(x)-E_s(x_0)|\frac{ds}{s}   \leq CA\int_r^\infty\frac{r}{s}\frac{ds}{s}
    \leq CA.
$$
Combining the two ranges,
$$
    \frac1{\omega(Q)}\int_Q  \bigg|\int_a^B E_s(x)\frac{ds}{s}-c_Q\bigg|d\omega(x)  \leq CA.
$$
Taking the supremum over $Q$ establishes the lemma.
\end{proof}

We can now upgrade the scale component tests to the integrated $T1$ tests actually required by the single-operator
$T1$ theorem.

\begin{proposition}\label{cc:prop:component-BMO}
For every chamber index $\tau$ and all $0<a<B<\infty$,
$$
    \bigg\|\int_a^B\Theta_s\chi_\tau\,\frac{ds}{s}\bigg\|_{\operatorname{BMO}(d\omega)}
    \leq C\|b\|_{\operatorname{Lip}_d}.
$$
The same estimate holds with $\Theta_s$ replaced by $\Theta_s^*$.  Consequently, for every pair
$1\leq\rho,\tau\leq M_G$,
$$
    \sup_{0<a<B<\infty}  \left(  \|S_{a,B}^{\rho\tau}1\|_{\operatorname{BMO}(\mathcal C,d\omega_{\mathcal C})}
    +\|(S_{a,B}^{\rho\tau})^*1\|_{\operatorname{BMO}(\mathcal C,d\omega_{\mathcal C})}  \right)
    \leq C\|b\|_{\operatorname{Lip}_d}.
$$
\end{proposition}

\begin{proof}
Suppose first that $b$ is smooth.  Lemma~\ref{cc:lem:component-split} and
Lemma~\ref{cc:lem:wall-BMO} imply
\begin{equation*}
    \sup_{0<a<B<\infty}
    \left\|\int_a^BE_{\tau,s}\frac{ds}{s}\right\|_{\operatorname{BMO}(d\omega)}
    \lesssim\|b\|_{\operatorname{Lip}_d}.
\end{equation*}
Since $\|g_\tau\|_\infty\leq\|b\|_{\operatorname{Lip}_d}$,
\begin{align*}
    \int_a^B\Theta_s\chi_\tau\frac{ds}{s}   =-\mathscr R_{i;a,B}g_\tau+
    \int_a^BE_{\tau,s}\frac{ds}{s}\qquad\text{and}\qquad
    \left\|\int_a^B\Theta_s\chi_\tau\frac{ds}{s}\right\|_{\operatorname{BMO}(d\omega)}
     \lesssim\|b\|_{\operatorname{Lip}_d}
\end{align*}
by Lemma~\ref{cc:lem:heat-Riesz}.

For general $b$, let $b_\varepsilon$ be the approximants from
Lemma~\ref{cc:lem:smooth-symbol}.  The two-sided estimate in the proof of
Proposition~\ref{cc:prop:component-test} gives
\begin{align*}
    \left\|\int_a^B  (\Theta_{s,b_\varepsilon}-\Theta_{s,b})\chi_\tau\frac{ds}{s}\right\|_\infty
    &\lesssim  \|b_\varepsilon-b\|_\infty\int_a^B\frac{ds}{s^2}  \longrightarrow0.
\end{align*}
Thus the uniform BMO estimate passes to $b$.

The adjoint assertion follows from Lemma~\ref{cc:lem:theta-adjoint}, applied to $\overline b$.

Finally,
$$
    S_{a,B}^{\rho\tau}1(x)    =2c_\ast\left(\int_a^B\Theta_s\chi_\tau\frac{ds}{s}\right)(\mathsf g_\rho x),
    \qquad x\in\mathcal C,
$$
and
$$
    (S_{a,B}^{\rho\tau})^*1(x)   =2c_\ast\left(\int_a^B\Theta_s^*\chi_\rho\frac{ds}{s}\right)(\mathsf g_\tau x).
$$
Reflection preserves $d\omega$.  Moreover,
$\omega(B(x,r))\approx\omega(B(x,r)\cap\mathcal C)$ by
Lemma~\ref{cc:lem:chamber-volume}.  Restricting a full-space BMO function to a reflected
chamber therefore preserves its BMO norm up to a structural constant.  More precisely, if
$Q=B(x,r)\cap\mathcal C$, then $\mathsf g_\rho Q$ lies in $B(\mathsf g_\rho x,r)$ and
$\omega(B(\mathsf g_\rho x,r))\approx\omega(Q)$.  If $h_Q$ denotes the average of
$h\circ\mathsf g_\rho$ on $Q$, then
$$
    \int_Q|h(\mathsf g_\rho z)-h_Q|\,d\omega(z)
    \leq2\int_{\mathsf g_\rho Q}|h-h_{B(\mathsf g_\rho x,r)}|\,d\omega
    \lesssim\omega(Q)\|h\|_{\operatorname{BMO}(d\omega)}.
$$
The two entrywise assertions follow.
\end{proof}

We next verify weak boundedness uniformly in the scale endpoints.  We use the product-bump formulation of WBP
from the scalar $T1$ theorem on a space of homogeneous type.

Let $\eta_{\mathcal C}>0$ denote the H\"older regularity exponent of the Auscher--Hyt\"onen wavelets on the chamber
\cite[Section~7]{AuscherHytonenWavelets}, and fix once and for all $0<\gamma<\eta_{\mathcal C}$.  The next proposition proves precisely
$\operatorname{WBP}(\gamma)$.

\begin{proposition}\label{cc:prop:entry-WBP}
Let $Q=B_{\mathcal C}(x_0,r)$, and let $\varphi,\psi$ be supported in
$Q$ and satisfy
$$
     \|\varphi\|_\infty+r^\gamma[\varphi]_{C^\gamma(\mathcal C)}\leq1,
     \qquad    \|\psi\|_\infty+r^\gamma[\psi]_{C^\gamma(\mathcal C)}\leq1.
$$
Then, uniformly in $0<a<B<\infty$ and in $\rho,\tau$,
$$
     |\langle S_{a,B}^{\rho\tau}\varphi,\psi\rangle_{L^2(\mathcal C,d\omega_{\mathcal C})}|
     \leq C\|b\|_{\operatorname{Lip}_d}\omega_{\mathcal C}(Q).
$$
\end{proposition}

\begin{proof}
Let $\pi:\R^N\to\overline{\mathcal C}$ send a point to the representative of its
$G$-orbit in the closed chamber.  The H\"older functions $\varphi$ and $\psi$ have unique continuous extensions to
$\overline{\mathcal C}$, which we use without changing notation.  By
Lemma~\ref{cc:lem:closest},
$$
        \|\pi z-\pi w\|=d(z,w)\leq\|z-w\|.
$$
Hence
$$
        \Phi(z)=\varphi(\pi z),   \qquad    \Psi(z)=\psi(\pi z)
$$
are $G$-invariant $C^\gamma$ functions supported in a fixed orbit
enlargement of $B(x_0,r)$, with
$$
     \|\Phi\|_\infty+\|\Psi\|_\infty\leq2,   \qquad
     [\Phi]_{C^\gamma}+[\Psi]_{C^\gamma}\lesssim r^{-\gamma}.
$$
By the definition of the lifted entry,
\begin{align*}
     \langle S_{a,B}^{\rho\tau}\varphi,\psi\rangle_{L^2(\mathcal C,d\omega_{\mathcal C})}
     &=2c_\ast\int_a^B
     \langle\Theta_s(\Phi\chi_\tau),\Psi\chi_\rho\rangle_{L^2(\R^N,d\omega)}  \frac{ds}{s}.
\end{align*}

We first assume that $b$ is smooth and $G$-invariant.  Split the
scale integral at $r$.  For $s\geq r$, the lifted scale-size estimate
\eqref{cc:eq:lift-scale}, chamber volume comparability,
and lower volume growth apply through
$$
    \frac{\omega_{\mathcal C}(Q)}{V_{\mathcal C}(x,y,s)}  \lesssim
    \left(\frac rs\right)^N,  \qquad x,y\in Q.
$$
The lifted scale-size estimate therefore gives
$$
     |\langle\Theta_s(\Phi\chi_\tau),\Psi\chi_\rho\rangle|
     \leq C\|b\|_{\operatorname{Lip}_d}\omega_{\mathcal C}(Q)
     \left(\frac{r}{s}\right)^N.
$$
Thus the large scales contribute at most
$C\|b\|_{\operatorname{Lip}_d}\omega_{\mathcal C}(Q)$.

For $0<s\leq r$, let $\Phi_s$ and $\Psi_s$ be Euclidean
mollifications at radius comparable with $s$, using a radial
mollifier.  They remain $G$-invariant, are supported in a fixed orbit
enlargement of $B(x_0,r)$, and satisfy
\begin{align*}
     \|\Phi-\Phi_s\|_\infty+  \|\Psi-\Psi_s\|_\infty
     \lesssim\left(\frac{s}{r}\right)^\gamma,\qquad   s\|\nabla\Phi_s\|_\infty+
     s\|\nabla\Psi_s\|_\infty  \lesssim\left(\frac{s}{r}\right)^\gamma.
\end{align*}
The scale-size estimate and Gaussian mass therefore give
\begin{align*}
     &\big|\langle\Theta_s(\Phi\chi_\tau),\Psi\chi_\rho\rangle
     -\langle\Theta_s(\Phi_s\chi_\tau),\Psi_s\chi_\rho\rangle
     \big| \leq C\|b\|_{\operatorname{Lip}_d}\omega_{\mathcal C}(Q)  \left(\frac{s}{r}\right)^\gamma.
\end{align*}
This error is integrable in $ds/s$.

Let
\begin{equation*}
        f_s=\Phi_s\eta_{\tau,s}, \qquad    g_s=\Psi_s\eta_{\rho,s}.
\end{equation*}
Replacing the two chamber indicators by their cutoffs produces the
boundary errors
\begin{align*}
     \langle\Theta_s(\Phi_s(\chi_\tau-\eta_{\tau,s})),  \Psi_s\chi_\rho\rangle,\qquad \text{and}\qquad
     \langle\Theta_sf_s,  \Psi_s(\chi_\rho-\eta_{\rho,s})\rangle.
\end{align*}
Let $E=\mathcal O(B(x_0,Cr))$, where the structural constant $C$ is fixed.  By the cutoff estimates,
$$
    |\Phi_s(\chi_\tau-\eta_{\tau,s})|  +|\Psi_s(\chi_\rho-\eta_{\rho,s})|  \lesssim m_s\one_E.
$$
Corollary~\ref{cc:cor:theta} and Gaussian mass give, respectively,
\begin{align*}
    |\langle\Theta_s(\Phi_s(\chi_\tau-\eta_{\tau,s})),   \Psi_s\chi_\rho\rangle|
    &\lesssim   \|b\|_{\operatorname{Lip}_d}\int_E\mathcal G_sm_s\,d\omega,
\end{align*}
and
\begin{align*}
    |\langle\Theta_sf_s,  \Psi_s(\chi_\rho-\eta_{\rho,s})\rangle|
    &\lesssim  \|b\|_{\operatorname{Lip}_d}\int_E m_s\,d\omega.
\end{align*}
Thus their integrated absolute values are bounded by
\begin{equation*}
     C\|b\|_{\operatorname{Lip}_d}  \int_0^r\int_E
     \bigl(m_s+\mathcal G_sm_s\bigr)\,d\omega\frac{ds}{s}.
\end{equation*}
The $q=1$ case of \eqref{cc:eq:orbit-wall} bounds this by
$C\|b\|_{\operatorname{Lip}_d}\omega_{\mathcal C}(Q)$.

It remains to estimate $\langle\Theta_sf_s,g_s\rangle$.  Since
$\Phi_s$ and $\Psi_s$ are $G$-invariant, the Dunkl product rule and
Lemma~\ref{cc:lem:cutoff} imply
\begin{equation*}
     |sT_jf_s|+|sT_ig_s|  \leq C\left(\left(\frac{s}{r}\right)^\gamma+m_s\right)
     \one_{\mathcal O(B(x_0,Cr))}.
\end{equation*}
By Lemma~\ref{cc:lem:theta-identity},
\begin{align*}
     \Theta_sf_s   &=[M_b,T_iH_{s^2}](sT_jf_s)   -sT_iH_{s^2}((\partial_jb)f_s).
\end{align*}
The kernel of the first term is
$(b(x)-b(y))T_{i,x}h_{s^2}(x,y)$.  The orbit Lipschitz estimate and
Lemma~\ref{cc:lem:heat-first} bound its absolute
value by
\begin{equation*}
     C\|b\|_{\operatorname{Lip}_d}\frac1{V(x,y,s)}  \exp\left(-c\frac{d(x,y)^2}{s^2}\right).
\end{equation*}
Since $g_s$ is bounded by one and supported in $E$, the corresponding double integral satisfies
\begin{align*}
    &|\langle[M_b,T_iH_{s^2}](sT_jf_s),g_s\rangle| \lesssim  \|b\|_{\operatorname{Lip}_d}
    \int_E\int_E  \frac{e^{-cd(x,y)^2/s^2}}{V(x,y,s)}   \left(\left(\frac sr\right)^\gamma+m_s(y)\right)
    \,d\omega(y)d\omega(x).
\end{align*}
Gaussian mass, reflection invariance, and the fact that the fixed enlargement $E$ has measure comparable to
$\omega_{\mathcal C}(Q)$ therefore imply
\begin{align*}
     |\langle[M_b,T_iH_{s^2}](sT_jf_s),g_s\rangle|  &\leq C\|b\|_{\operatorname{Lip}_d}
     \left(\left(\frac{s}{r}\right)^\gamma\omega_{\mathcal C}(Q)  +\int_E\mathcal G_sm_s\,d\omega\right).
\end{align*}
For the second term, skew-adjointness of $T_i$, self-adjointness of
$H_{s^2}$, and their commutation give
\begin{align*}
     |\langle sT_iH_{s^2}((\partial_jb)f_s),g_s\rangle|  &=|\langle(\partial_jb)f_s,H_{s^2}(sT_ig_s)\rangle|\\
     &\leq C\|b\|_{\operatorname{Lip}_d}  \int_E\left[  \left(\frac sr\right)^\gamma+\mathcal G_sm_s(x)
     \right]d\omega(x)\\
     &\leq C\|b\|_{\operatorname{Lip}_d}  \left(\left(\frac{s}{r}\right)^\gamma\omega_{\mathcal C}(Q)
     +\int_E\mathcal G_sm_s\,d\omega\right).
\end{align*}
The factor $(s/r)^\gamma$ is integrable in $ds/s$, and the wall term
is controlled by \eqref{cc:eq:orbit-wall} with $q=1$, establishing the
required estimate for smooth $b$.

Finally, use the smooth symbol approximation in
Lemma~\ref{cc:lem:smooth-symbol}.  If $\Theta_{s,\varepsilon}$
is associated with $b_\varepsilon$, then the two-sided estimate in the proof of
Proposition~\ref{cc:prop:component-test} gives
\begin{align*}
    &\left|  \left\langle(\Theta_{s,\varepsilon}-\Theta_s)(\Phi\chi_\tau),  \Psi\chi_\rho\right\rangle\right|
    \lesssim  \omega_{\mathcal C}(Q)s^{-1}\|b_\varepsilon-b\|_\infty,
\end{align*}
and
\begin{align*}
     \left|  \int_a^B  \left\langle(\Theta_{s,\varepsilon}-\Theta_s)(\Phi\chi_\tau),
    \Psi\chi_\rho\right\rangle\frac{ds}{s}  \right|
    \lesssim  \omega_{\mathcal C}(Q)\|b_\varepsilon-b\|_\infty  \int_a^B\frac{ds}{s^2}  \longrightarrow0.
\end{align*}
Passing to the limit in the uniform smooth estimates proves the assertion for every
$b\in\operatorname{Lip}_d$.
\end{proof}

The scale estimates also provide uniform Calder\'on--Zygmund estimates for every finite integrated entry.

\begin{lemma}\label{cc:lem:entry-kernel}
The kernel
$$
    k_{a,B}^{\rho\tau}(x,y)   =2c_\ast\int_a^B\theta_s^{\rho\tau}(x,y)\,\frac{ds}{s}
$$
satisfies, uniformly in $0<a<B<\infty$ and in $\rho,\tau$,
$$
    |k_{a,B}^{\rho\tau}(x,y)|  \leq C\|b\|_{\operatorname{Lip}_d}
    \frac1{V_{\mathcal C}(x,y,\|x-y\|)}.
$$
If $\|x-x'\|\leq\|x-y\|/4$, then
$$
    |k_{a,B}^{\rho\tau}(x,y)-k_{a,B}^{\rho\tau}(x',y)| \leq C\|b\|_{\operatorname{Lip}_d}
    \frac{\|x-x'\|}{\|x-y\|}   \frac1{V_{\mathcal C}(x,y,\|x-y\|)},
$$
and the analogous estimate holds in the second variable.

Consequently, for every fixed $0<\gamma\leq1$, if $\|x-x'\|\leq\|x-y\|/2$, then
$$
    |k_{a,B}^{\rho\tau}(x,y)-k_{a,B}^{\rho\tau}(x',y)| \leq C_\gamma\|b\|_{\operatorname{Lip}_d}
    \left(\frac{\|x-x'\|}{\|x-y\|}\right)^\gamma   \frac1{V_{\mathcal C}(x,y,\|x-y\|)},
$$
and the analogous estimate holds in the second variable when $\|y-y'\|\leq\|x-y\|/2$.
\end{lemma}

\begin{proof}
Let
\begin{equation*}
    K_{b;a,B}^{ij}(z,w)  =c_\ast(b(z)-b(w))\int_{a^2}^{B^2}A_t^{ij}(z,w)\frac{dt}{\sqrt t}.
\end{equation*}
Then $k_{a,B}^{\rho\tau}(x,y)=K_{b;a,B}^{ij}(\mathsf g_\rho x,\mathsf g_\tau y)$.  The proofs of
Lemma~\ref{cc:lem:K-size} and Propositions~\ref{cc:prop:K-x}--\ref{cc:prop:K-y} use absolute
heat-parameter integrals; restricting them to $[a^2,B^2]$ preserves every bound and makes the constants independent
of $a,B$.  With $r=\|x-y\|$, the chamber distance and volume comparisons read
\begin{equation*}
    d(\mathsf g_\rho x,\mathsf g_\tau y)=\|x-y\|,  \qquad
    V(\mathsf g_\rho x,\mathsf g_\tau y,r)\approx V_{\mathcal C}(x,y,r),
\end{equation*}
and hence the stated size and first-order estimates follow.

Let $h=\|x-x'\|$.  The first-order estimate implies the order-$\gamma$ estimate when
$h\leq r/4$.  If $r/4<h\leq r/2$, then $r'=\|x'-y\|\approx r$ and
$V_{\mathcal C}(x',y,r')\approx V_{\mathcal C}(x,y,r)$; hence the two size estimates give
\begin{equation*}
    |k_{a,B}^{\rho\tau}(x,y)-k_{a,B}^{\rho\tau}(x',y)|
    \lesssim\frac{\|b\|_{\operatorname{Lip}_d}}{V_{\mathcal C}(x,y,r)}
    \lesssim_\gamma  \|b\|_{\operatorname{Lip}_d}\left(\frac hr\right)^\gamma
    \frac1{V_{\mathcal C}(x,y,r)}.
\end{equation*}
The second-variable estimate is identical.
\end{proof}

We can now apply the ordinary single-operator $T1$ theorem, not a continuous-scale substitute.

\begin{theorem}\label{cc:thm:entry-L2}
For every $1\leq\rho,\tau\leq M_G$,
$$
    \sup_{0<a<B<\infty}  \|S_{a,B}^{\rho\tau}\|_{L^2(\mathcal C,d\omega_{\mathcal C})\to 
    L^2(\mathcal C,d\omega_{\mathcal C})}  \leq C\|b\|_{\operatorname{Lip}_d}.
$$
Consequently,
$$
    \sup_{0<a<B<\infty}  \|UC_{a,B}^{ij,b}U^{-1}\|_{L^2(\mathcal C;\ell_{M_G}^2)\to L^2(\mathcal C;\ell_{M_G}^2)}
    \leq C\|b\|_{\operatorname{Lip}_d}.
$$
\end{theorem}

\begin{proof}
Lemma~\ref{cc:lem:chamber-volume} shows that the chamber is a metric space of homogeneous
type.  The restricted measure $d\omega_{\mathcal C}$ is a $\sigma$-additive Radon measure on the Borel sets and is
finite on bounded balls.  Thus it also satisfies the measure convention clarified in
\cite{AuscherHytonenAddendum}.

Fix the exponent $0<\gamma<\eta_{\mathcal C}$ chosen above and let
$$
    \mathcal V_\gamma=C_0^\gamma(\mathcal C).
$$
For fixed $a$ and $B$, the finite scale integral defines
$$
    S_{a,B}^{\rho\tau}:\mathcal V_\gamma\longrightarrow\mathcal V_\gamma'.
$$
For $f\in\mathcal V_\gamma$ and $x\notin\supp f$, the lifted scale-size estimate and Gaussian mass give
$$
    \int_a^B\int_{\mathcal C}  |\theta_s^{\rho\tau}(x,y)f(y)|\,d\omega(y)\frac{ds}{s}
    \leq  C\|b\|_{\operatorname{Lip}_d}\|f\|_\infty\log\frac Ba<\infty.
$$
Fubini's theorem therefore gives
$$
    S_{a,B}^{\rho\tau}f(x)  =\int_{\mathcal C}k_{a,B}^{\rho\tau}(x,y)f(y)\,d\omega(y)
$$
for almost every $x\notin\supp f$.  After pairing with a second function in
$\mathcal V_\gamma$, the same estimate gives the required continuity, with a constant which may depend on $a$ and $B$.
Lemma~\ref{cc:lem:entry-kernel} supplies the kernel estimates of order $\gamma$, uniformly in
$a$, $B$, $\rho$, and $\tau$.

The same Gaussian domination makes the finite-scale constant tests absolutely convergent.  Fubini's theorem
therefore identifies the function denoted by $S_{a,B}^{\rho\tau}1$ in
Proposition~\ref{cc:prop:component-BMO} with the distributional $T1$ datum.  With our convention that
the complex pairing is linear in the first variable, the transpose test in the scalar theorem is
$$
    {}^{\mathrm t}S_{a,B}^{\rho\tau}1  =\overline{(S_{a,B}^{\rho\tau})^*1}.
$$
It consequently has the same BMO norm as the adjoint test established in that proposition.  We have
$$
    \|S_{a,B}^{\rho\tau}1\|_{\operatorname{BMO}(\mathcal C,d\omega_{\mathcal C})}
    +\|(S_{a,B}^{\rho\tau})^*1\|_{\operatorname{BMO}(\mathcal C,d\omega_{\mathcal C})}
    \leq C\|b\|_{\operatorname{Lip}_d}.
$$
Proposition~\ref{cc:prop:entry-WBP} verifies $\operatorname{WBP}(\gamma)$ with the same constant.  Hence
\cite[Theorem~12.3]{AuscherHytonenWavelets} applies uniformly and yields the first estimate.  Its
quantitative bound depends on the homogeneous-type structure and on the displayed kernel, WBP, and BMO constants,
all of which are uniform here.  It is therefore independent of $a$, $B$, $\rho$, and $\tau$.  In particular, the
argument applies the scalar theorem separately to each integrated truncation; it uses no continuous-family $T1$
principle.

The matrix has only $M_G^2$ entries.  If $F=(F_1,\ldots,F_{M_G})$, then the first estimate and Cauchy--Schwarz in the
finite coordinate sums give
\begin{align*}
    \| (S_{a,B}^{\rho\tau})_{\rho,\tau=1}^{M_G}F\|_{L^2(\mathcal C;\ell_{M_G}^2)}^2
    &\leq M_G\sum_{\rho,\tau=1}^{M_G}\|S_{a,B}^{\rho\tau}F_\tau\|_2^2 \leq C_{M_G}\|b\|_{\operatorname{Lip}_d}^2
        \sum_{\tau=1}^{M_G}\|F_\tau\|_2^2.
\end{align*}
This establishes the second estimate directly, without invoking any convergence statement.
\end{proof}

It remains to prove joint convergence at the two scale endpoints.  This does not require another $T1$ argument.
The finite heat multipliers converge on a dense Sobolev core, and the uniform estimate just proved extends that
convergence to all of $L^2$.

\begin{proposition}\label{cc:prop:joint-limit}
The net $UC_{a,B}^{ij,b}U^{-1}$ converges in the weak-operator topology on
$L^2(\mathcal C,d\omega_{\mathcal C};\ell_{M_G}^2)$ as $a\to0+$ and $B\to+\infty$, jointly and independently of the path.
Its limit is
$$
    \mathbb T_b^{ij}:=U\mathcal C_{b,\mathrm{abs}}^{ij}U^{-1},
$$
and
$$
\|\mathbb T_b^{ij}\|_{L^2(\mathcal C;\ell_{M_G}^2)\to L^2(\mathcal C;\ell_{M_G}^2)}
    \leq C\|b\|_{\operatorname{Lip}_d}.
$$
\end{proposition}

\begin{proof}  
Let
$$
    \mathscr D:=C_c^\infty(\mathcal C;\bbC^{M_G})
    =\left\{F=(F_1,\ldots,F_{M_G}):F_\rho\in C_c^\infty(\mathcal C)\right\}.
$$
The measure $d\omega|_{\mathcal C}$ is a locally finite Borel measure and the chamber wall is $\omega$-null.
Approximation first by compactly supported continuous functions and then by Euclidean mollification inside the open
chamber shows that $C_c^\infty(\mathcal C)$ is dense in $L^2(\mathcal C,d\omega_{\mathcal C})$.  Hence $\mathscr D$
is dense in the lifted vector-valued space.  If $F=G=0$, there is nothing to prove.  Otherwise, compactness and the
finiteness of the coordinate set ensure
\begin{equation*}
    \delta_{F,G}:=  \min_{\substack{1\leq\tau\leq M_G\\   
    \supp F_\tau\cup\supp G_\tau\ne\varnothing}}
    \dist\bigl(\supp F_\tau\cup\supp G_\tau,\partial\mathcal C\bigr)>0.
\end{equation*}
Hence $f=U^{-1}F$ and $g=U^{-1}G$ vanish in a neighborhood of $\mathcal W$ and belong to
$C_c^\infty(\R^N)$.

Let
$$
    A_{a,B}^{ij}=2c_\ast\int_a^B T_iT_jH_{s^2}\,ds.
$$
Then
$ 
C_{a,B}^{ij,b}=M_bA_{a,B}^{ij}-A_{a,B}^{ij}M_b
$ 
on compactly supported test functions.  Let
$$
    m_{a,B}^{ij}(\xi)  =-2c_\ast\xi_i\xi_j\int_a^B e^{-s^2\|\xi\|^2}\,ds,  \qquad
    m^{ij}(\xi)=\frac{\xi_i\xi_j}{\|\xi\|},
$$
where $m^{ij}(0)=0$.  These are the Dunkl multipliers of $A_{a,B}^{ij}$ and $T_i\mathcal R_j$, respectively.
Since $c_\ast=-\pi^{-1/2}$ and
$$
    \int_0^\infty e^{-s^2\|\xi\|^2}\,ds=\frac{\sqrt\pi}{2\|\xi\|},
$$
for $\xi\ne0$ one has
$$
    m^{ij}(\xi)-m_{a,B}^{ij}(\xi)   =-2c_\ast\xi_i\xi_j   \left(\int_0^a+\int_B^\infty\right)e^{-s^2\|\xi\|^2}\,ds.
$$
The two tails satisfy
\begin{align*}
    \left|\xi_i\xi_j\int_0^a e^{-s^2\|\xi\|^2}\,ds\right|   \leq a\|\xi\|^2,\qquad
    \left|\xi_i\xi_j\int_B^\infty e^{-s^2\|\xi\|^2}\,ds\right|
     \leq C\|\xi\|e^{-B^2\|\xi\|^2/2}.
\end{align*}
At $\xi=0$, both multipliers vanish.  Hence $m_{a,B}^{ij}(\xi)\to m^{ij}(\xi)$ jointly as $a\to0+$ and
$B\to+\infty$.  Moreover,
$$
    |m_{a,B}^{ij}(\xi)-m^{ij}(\xi)|\leq C\|\xi\|,
$$
uniformly in $a$ and $B$.  For $u\in\operatorname{Dom}(D)$,
\begin{equation*}
    |(m_{a,B}^{ij}-m^{ij})\mathcal F_\kappa u|^2
    \leq C\|\xi\|^2|\mathcal F_\kappa u|^2\in L^1(d\omega).
\end{equation*}
Dominated convergence and Plancherel therefore give
$$
    \|(A_{a,B}^{ij}-T_i\mathcal R_j)u\|_2\longrightarrow0
$$
jointly at the two endpoints, independently of the path.

Applying \eqref{cc:eq:weak-product} to $f$ and $g$, with symbols
$b$ and $\overline b$, respectively, and using
$\displaystyle
        \|Du\|_2^2=\sum_{k=1}^N\|T_ku\|_2^2,
$
we obtain $bf,\overline b g\in\operatorname{Dom}(D)$.  Consequently,
\begin{align*}
    \langle C_{a,B}^{ij,b}f,g\rangle  &=\langle A_{a,B}^{ij}f,\overline b g\rangle
      -\langle A_{a,B}^{ij}(bf),g\rangle.
\end{align*}
The preceding pairing has a joint, path-independent limit for every
$F,G\in\mathscr D$.  Since $A_{a,B}^{ij}$ converges to
$T_i\mathcal R_j$ on $\operatorname{Dom}(D)$, the limiting pairing is the
commutator pairing in Theorem~\ref{cc:thm:operator}.
Thus the dense-core limit and the bounded operator constructed by the
half-Laplacian route are the same.

The uniform matrix estimate in Theorem~\ref{cc:thm:entry-L2} now extends the limiting bilinear
form from $\mathscr D\times\mathscr D$ to the whole lifted $L^2$ space.  To see both convergence to the stated
operator and path independence explicitly, approximate arbitrary $F,G$ in $L^2$ by $F_0,G_0\in\mathscr D$ and use
\begin{align*}
    &|\langle\mathbb S_{a,B}(F-F_0),G\rangle|
      +|\langle\mathbb S_{a,B}F_0,G-G_0\rangle| +|\langle\mathbb T_b^{ij}(F-F_0),G\rangle|
      +|\langle\mathbb T_b^{ij}F_0,G-G_0\rangle|\\
     \leq\,&  C\|b\|_{\operatorname{Lip}_d}
    \bigl(\|F-F_0\|_2\|G\|_2+\|F_0\|_2\|G-G_0\|_2\bigr).
\end{align*}
Here $\mathbb S_{a,B}=UC_{a,B}^{ij,b}U^{-1}$.  The displayed expression controls the four approximation
errors; the remaining core term
$|\langle(\mathbb S_{a,B}-\mathbb T_b^{ij})F_0,G_0\rangle|$ tends to zero jointly at the two endpoints.  Since the
right-hand side is independent of $a$ and $B$, the limiting bilinear form converges to
$\mathbb T_b^{ij}=U\mathcal C_{b,\mathrm{abs}}^{ij}U^{-1}$ on the whole
lifted $L^2$ space, with the asserted norm.
\end{proof}

Combining the preceding theorem and proposition proves the second principal result.

\begin{proof}[Proof of Theorem~\ref{cc:thm:heat-limit}]
Theorem~\ref{cc:thm:entry-L2} supplies a uniform $L^2$ bound for the lifted truncations, while
Proposition~\ref{cc:prop:joint-limit} identifies their joint, path-independent weak limit as
$U\mathcal C_{b,\mathrm{abs}}^{ij}U^{-1}$.  Unitarity of $U$ transfers both conclusions to the full space.
\end{proof}

\section{Matrix kernel realization}\label{cc:sec:matrix}

We now identify the canonical heat limit with its chamber kernels.  Write
$$
    T_b^{ij}:=[M_b,T_i\mathcal R_j]
$$
for the common operator in Theorems~\ref{cc:thm:operator} and~\ref{cc:thm:heat-limit}.  The point of
this section is to connect its canonical heat realization with the
kernel constructed in Section~\ref{cc:sec:kernels}, and then to identify the finite scalar kernels after
chamber lifting.  The orbit distance used here is always the full $G$-orbit distance.  If some multiplicities vanish, this convention
may include more orbit coincidences than the effective heat kernel actually sees, but it provides one uniform geometric description for
the whole argument.  The corresponding kernel estimates are uniform upper bounds; they need not describe the
sharp singular geometry of every individual matrix entry.

\subsection{Regularized operators and separated supports}

The limit defining $T_b^{ij}$ is a weak operator limit.  Therefore, before using its scalar kernel description, one must verify that
this limit is associated with the kernel $K_b^{ij}$ in the usual bilinear sense.  This verification is made only for separated
supports, which is the form needed later on the chamber.  No principal value representation of
$T_i\mathcal R_j$ is used.

For bounded compactly supported functions, finite regularized operators are ordinary integral operators.
With the normalization convention of Remark~\ref{cc:rem:normalization}, the change of variables $t=s^2$ gives
$$
    C_{a,B}^{ij,b}    =2c_\ast\int_a^B\Theta_s\,\frac{ds}{s}.
$$
The expression is first considered at the level of finite truncations, where Fubini is justified by the compact supports and by the
heat-kernel bounds away from $t=0$ and $t=\infty$.  The passage to the limiting kernel is then a separated-support dominated
convergence argument.

\begin{proposition}\label{cc:prop:full-kernel}
Let $f,g\in L_c^\infty(\R^N,d\omega)$, and suppose that
$$
    \dist_d(\supp g,\supp f)   :=    \inf\{d(x,y):x\in\supp g,  \ y\in\supp f\}>0.
$$
Then
\begin{equation}\label{cc:eq:full-kernel}
    \langle T_b^{ij}f,g\rangle_{L^2(\R^N,d\omega)}
    =  \iint_{\R^N\times\R^N}  K_b^{ij}(x,y)f(y)\overline{g(x)}\,d\omega(y)d\omega(x),
\end{equation}
where $K_b^{ij}$ is defined in \eqref{cc:eq:kernel-def}.
\end{proposition}

\begin{proof}
Let $\delta_0=\dist_d(\supp g,\supp f)>0$, and choose $M>0$ so that both supports lie in $B(0,M)$.
For $0<a<B<\infty$, Fubini gives
\begin{align}\label{cc:eq:trunc-pairing}
    \langle C_{a,B}^{ij,b}f,g\rangle  &=  c_\ast\iint_{\R^N\times\R^N}
     (b(x)-b(y)) \left(\int_{a^2}^{B^2} A_t^{ij}(x,y)\frac{dt}{\sqrt t}\right)
     f(y)\overline{g(x)}\,d\omega(y)d\omega(x).
\end{align}
Indeed, \eqref{cc:eq:heat-second} shows that $A_t^{ij}$ is bounded on
$B(0,M)^2\times[a^2,B^2]$, while $b$ is bounded on $B(0,M)$.

On $\supp g\times\supp f$, Lemma~\ref{cc:lem:K-size}, $d(x,y)\geq\delta_0$, and the volume formula imply
\begin{align*}
    &|b(x)-b(y)|\int_0^\infty|A_t^{ij}(x,y)|\frac{dt}{\sqrt t}  |f(y)g(x)|\\
     \lesssim\,&  \|b\|_{\operatorname{Lip}_d}  \frac{|f(y)g(x)|}{V(x,y,d(x,y))}
    \leq C_{M,\delta_0}\|b\|_{\operatorname{Lip}_d}|f(y)g(x)|\in L^1(d\omega\times d\omega).
\end{align*}
Dominated convergence in \eqref{cc:eq:trunc-pairing} therefore permits the limit
$$
    \lim_{\substack{a\to0+\\ B\to+\infty}}  \langle C_{a,B}^{ij,b}f,g\rangle
    = \iint K_b^{ij}(x,y)f(y)\overline{g(x)}\,d\omega(y)d\omega(x).
$$
Theorem~\ref{cc:thm:heat-limit} identifies the same limit as
$\langle T_b^{ij}f,g\rangle_{L^2(\R^N,d\omega)}$; comparison establishes \eqref{cc:eq:full-kernel}.
\end{proof}

\subsection{The lifted final kernel}

We now express the associated kernel after passing to the chamber.  Recall that
$\mathbb T_b^{ij}=UT_b^{ij}U^{-1}$.
For $x,y\in\mathcal C$ and $1\leq\rho,\tau\leq M_G$, define
\begin{equation*}
    \mathbb K_b^{ij,\rho\tau}(x,y)   :=K_b^{ij}(\mathsf g_\rho x,\mathsf g_\tau y).
\end{equation*}
This definition is made only off the chamber diagonal $x=y$.  Indeed, by the chamber representative identity,
$$
    d(\mathsf g_\rho x,\mathsf g_\tau y)=\|x-y\|,   \qquad x,y\in\mathcal C.
$$
Thus all possible full-space orbit singularities, including those involving different reflected chambers, are represented on the
single ordinary diagonal $x=y$ on the chamber.
The indices $\rho,\tau$ are not quotiented out; they remain as matrix entries.

The following lemma records the associated-kernel identity on the chamber.  Vector-valued pairings are taken in
$L^2(\mathcal C,d\omega;\ell^2_{M_G})$, and full-space scalar pairings in $L^2(\R^N,d\omega)$.

\begin{lemma}\label{cc:lem:lift-kernel}
Let $F,G\in L_c^\infty(\mathcal C,d\omega;\bbC^{M_G})$.  Assume that their coordinate supports are separated in the Euclidean metric
on $\mathcal C$, namely
\begin{equation*}
    \delta_{\mathcal C}(F,G) :=   \inf_{\substack{1\leq \rho,\tau\leq M_G\\
    \supp G_\rho\ne\varnothing,\\   
    \supp F_\tau\ne\varnothing}}    \inf\left\{\|x-y\|:  x\in\supp G_\rho,\\
    y\in\supp F_\tau\right\}>0,
\end{equation*}
with the convention $\inf\varnothing=+\infty$.
Then
\begin{align}\label{cc:eq:lift-kernel}
    \langle \mathbb T_b^{ij}F,G\rangle_{L^2(\mathcal C,d\omega;\ell^2_{M_G})}
    &=  \sum_{\rho,\tau=1}^{M_G}   \iint_{\mathcal C\times\mathcal C}
    \mathbb K_b^{ij,\rho\tau}(x,y)
    F_\tau(y)\overline{G_\rho(x)}\,d\omega(y)d\omega(x).
\end{align}
\end{lemma}

\begin{proof}
The assertion is immediate if $F=0$ or $G=0$; assume otherwise.
Let $f=U^{-1}F$ and $g=U^{-1}G$.  If
$z=\mathsf g_\rho x\in\supp g$ and $w=\mathsf g_\tau y\in\supp f$, then the chamber identity implies
$$
    d(z,w)=d(\mathsf g_\rho x,\mathsf g_\tau y)=\|x-y\|   \geq\delta_{\mathcal C}(F,G).
$$
Thus Proposition~\ref{cc:prop:full-kernel} applies to $f$ and $g$.  Decomposing both variables
into the reflected chambers and using the measure-preserving changes of variables
$z=\mathsf g_\rho x$ and $w=\mathsf g_\tau y$, we obtain
\begin{align*}
    \langle\mathbb T_b^{ij}F,G\rangle  &=\langle T_b^{ij}f,g\rangle =\sum_{\rho,\tau=1}^{M_G}
    \iint_{\mathcal C\times\mathcal C}   K_b^{ij}(\mathsf g_\rho x,\mathsf g_\tau y)
    F_\tau(y)\overline{G_\rho(x)}\,d\omega(y)d\omega(x),
\end{align*}
which is \eqref{cc:eq:lift-kernel}.
\end{proof}

\begin{lemma}\label{cc:lem:matrix-CZ}
For every $i,j=1,\ldots,N$ and every $1\leq\rho,\tau\leq M_G$, the kernel $\mathbb K_b^{ij,\rho\tau}$ is a standard
Calder\'on--Zygmund kernel on $(\mathcal C,\|\cdot\|,d\omega_{\mathcal C})$.  More explicitly, if $x,y\in\mathcal C$ and $x\ne y$,
then
\begin{equation}\label{cc:eq:matrix-size}
    |\mathbb K_b^{ij,\rho\tau}(x,y)|   \leq   C\|b\|_{\operatorname{Lip}_d}
    \frac1{V_{\mathcal C}(x,y,\|x-y\|)}.
\end{equation}
If $x,x',y\in\mathcal C$ and $\|x-x'\|\leq \|x-y\|/4$, then
\begin{align}\label{cc:eq:matrix-x}
    &|\mathbb K_b^{ij,\rho\tau}(x,y)    -\mathbb K_b^{ij,\rho\tau}(x',y)|
    \leq   C\|b\|_{\operatorname{Lip}_d}   \frac{\|x-x'\|}{\|x-y\|}
    \frac1{V_{\mathcal C}(x,y,\|x-y\|)}.
\end{align}
If $x,y,y'\in\mathcal C$ and $\|y-y'\|\leq \|x-y\|/4$, then
\begin{align}\label{cc:eq:matrix-y}
    &|\mathbb K_b^{ij,\rho\tau}(x,y)    -\mathbb K_b^{ij,\rho\tau}(x,y')|
    \leq   C\|b\|_{\operatorname{Lip}_d}   \frac{\|y-y'\|}{\|x-y\|}
    \frac1{V_{\mathcal C}(x,y,\|x-y\|)}.
\end{align}
All constants are independent of $\rho$ and $\tau$.
\end{lemma}

\begin{proof}
Let $z=\mathsf g_\rho x$, $w=\mathsf g_\tau y$, and $r=\|x-y\|$.  The chamber identity and chamber volume
comparability give, uniformly in $\rho$ and $\tau$,
$$
    d(z,w)=r,  \qquad   V(z,w,r)\approx V_{\mathcal C}(x,y,r).
$$
The definition of $\mathbb K_b^{ij,\rho\tau}$ and \eqref{cc:eq:K-size} imply \eqref{cc:eq:matrix-size}.

If $\|x-x'\|\leq r/4$ and $z'=\mathsf g_\rho x'$, then
$ 
\|z-z'\|=\|x-x'\|\leq d(z,w)/4.
$ 
Applying \eqref{cc:eq:K-x} and the same comparisons gives
\eqref{cc:eq:matrix-x}.  Likewise, if $\|y-y'\|\leq r/4$ and
$w'=\mathsf g_\tau y'$, then
$ 
\|w-w'\|=\|y-y'\|\leq d(z,w)/4,
$ 
so \eqref{cc:eq:K-y} gives \eqref{cc:eq:matrix-y}.
\end{proof}

\begin{remark}
For fixed $\rho,\tau$, the entry $\mathbb K_b^{ij,\rho\tau}(x,y)$ may represent a full-space interaction
between different chambers.  Its only possible singular set on the chamber is contained in $x=y$, because
$d(\mathsf g_\rho x,\mathsf g_\tau y)=\|x-y\|$; some nonidentity entries may be strictly smoother.  Thus each entry is a scalar
Calder\'on--Zygmund kernel on the usual space of homogeneous type $(\mathcal C,\|\cdot\|,d\omega_{\mathcal C})$, with the ordinary
chamber diagonal.  The finite matrix indices record the reflected geometry; they do not change the underlying diagonal of the
scalar theory.
\end{remark}

For $1\leq\rho,\tau\leq M_G$,  let
$$
    T^{\rho\tau}:=P_\rho\mathbb T_b^{ij}J_\tau,
$$
where $J_\tau$ injects a scalar function on $\mathcal C$ into the $\tau$-th coordinate and $P_\rho$ projects onto the $\rho$-th
coordinate.  Applying Lemma~\ref{cc:lem:lift-kernel} to $F=J_\tau u$ and $G=J_\rho v$ shows that, for separated
$u,v\in L_c^\infty(\mathcal C,d\omega)$,
$$
    \langle T^{\rho\tau}u,v\rangle_{L^2(\mathcal C,d\omega)}
    =  \iint_{\mathcal C\times\mathcal C}  \mathbb K_b^{ij,\rho\tau}(x,y)
    u(y)\overline{v(x)}\,d\omega(y)d\omega(x).
$$
Thus $T^{\rho\tau}$ is associated, in the usual separated-support sense on
$(\mathcal C,\|\cdot\|,d\omega_{\mathcal C})$, with the scalar kernel $\mathbb K_b^{ij,\rho\tau}$.

\begin{proof}[Proof of Corollary~\ref{cc:cor:matrix-kernel}]
Lemma~\ref{cc:lem:K-size} ensures absolute convergence, the preceding coordinate specialization provides separated-support
association, and Lemma~\ref{cc:lem:matrix-CZ} supplies the three uniform kernel estimates.
\end{proof}

Finally, Theorem~\ref{cc:thm:operator} and the $L^p$ isometry of $U$ give
$$
    \|\mathbb T_b^{ij}F\|_{L^p(\mathcal C,d\omega;\ell^p_{M_G})}
    \leq C_p\|b\|_{\operatorname{Lip}_d}  \|F\|_{L^p(\mathcal C,d\omega;\ell^p_{M_G})},
    \qquad 1<p<\infty.
$$
Since $P_\rho$ and $J_\tau$ are contractions, every entry $T^{\rho\tau}=P_\rho\mathbb T_b^{ij}J_\tau$ satisfies the
corresponding scalar $L^p$ estimate.  Thus each scalar entry inherits the full-space bound.

\bigskip
\noindent\textbf{Acknowledgements:}
Ming-Yi Lee is supported by NSTC 114-2115-M-008-006-MY2.  Ji Li is supported by the Australian Research Council through grants
DP220100285 and DP260100485.  Liangchuan Wu is supported by the National Natural Science Foundation of China through grant
12201002.  Ji Li would like to thank Professor Agnieszka Hejna for very helpful discussions on Dunkl transforms and convolution.

\raggedbottom

\end{document}